\documentclass[11pt]{amsart}
\usepackage[utf8]{inputenc}
\usepackage{amsmath,amssymb,amsthm, mathrsfs, fullpage,xcolor,tikz,tikz-cd,verbatim, enumerate,amssymb,thmtools}
\usepackage[colorlinks, linkcolor=blue, citecolor=blue]{hyperref}
\usepackage{graphicx}
\usepackage{color}
\usepackage[nameinlink]{cleveref}
\usepackage[]{appendix}

\usetikzlibrary{fit}

\newtheorem{lemma}{Lemma}[section]
\newtheorem{theorem}[lemma]{Theorem}
\newtheorem{proposition}[lemma]{Proposition}
\newtheorem{prop}[lemma]{Proposition}
\newtheorem{corollary}[lemma]{Corollary}
\theoremstyle{definition}

\newtheorem{xmp}[lemma]{Example}

\newtheorem{defn}[lemma]{Definition}
\newtheorem{remark}[lemma]{Remark}
\newtheorem{setting}[lemma]{Setting}
\newtheorem{ex}[lemma]{Example}
\newtheorem*{conv}{Conventions}
\newtheorem*{ack}{Acknowledgements}

\newtheorem*{mainthm}{Theorem}
\newtheorem*{maindefn}{Definition}

\numberwithin{equation}{section}

\newcommand{\mfm}{\mathfrak{m}}
\newcommand{\mfp}{\mathfrak{p}}
\newcommand{\mfq}{\mathfrak{q}}
\newcommand{\mbr}{\mathbb{R}}
\newcommand{\mbn}{\mathbb{N}}
\newcommand{\N}{\mathbb{N}}
\newcommand{\mfa}{\mathfrak{a}}
\newcommand{\mfb}{\mathfrak{b}}
\newcommand{\fbp}[1]{\left[#1 \right]}
\newcommand{\ifat}{I^{F\mfa^t}}
\newcommand{\atif}{{}^{\mfa^t}I^F}
\newcommand{\istarat}{I^{*\mfa^t}}
\newcommand{\atistar}{{}^{\mfa^t}I^*}
\newcommand{\atpe}{\mfa^{\lceil tp^e\rceil}}
\newcommand{\iatb}{\overline{I}^{\mfa^t}}
\newcommand{\atib}{{}^{\mfa^t}\overline{I}}
\newcommand{\Z}{\mathbb{Z}}
\newcommand{\R}{\mathbb{R}}
\renewcommand{\emptyset}{\varnothing}
\renewcommand{\vec}[1]{\mathbf{#1}}
\DeclareMathOperator{\Exp}{Exp}
\newcommand{\Frac}{\operatorname{Frac}}

\colorlet{DO}{green!50!black}

\newcommand{\changelocaltocdepth}[1]{%
  \addtocontents{toc}{\protect\setcounter{tocdepth}{#1}}%
  \setcounter{tocdepth}{#1}%
}

\begin{document}

\title{New versions of Frobenius and integral closure of ideals}
\date{}

\author{Kriti Goel}
\address{Kriti Goel: BCAM-Basque Center for Applied Mathematics, Bilbao, Bizkaia 48009 Spain}
\email{kritigoel.maths@gmail.com}
\author{Kyle Maddox} 
\address{Kyle Maddox: Department of Mathematical Sciences,  University of Arkansas, 850 West Dickson Street, Fayetteville, Arkansas 72701, United States}
\email{kmaddox@uark.edu}
\author{William D.\ Taylor}
\address{William D.\ Taylor: Department of Mathematical Sciences, Tennessee State University, 3500 John A Merritt Blvd, Nashville, Tennessee 37209, United States}
\email{wtaylo17@tnstate.com}

\thanks{2020 {\em Mathematics Subject Classification}.
	Primary 13A35, 13B22, 
	Secondary 13D40}
	
\thanks{Keyword: Integral closure, tight closure, Frobenius closure, semigroup rings, joint Hilbert-Kunz multiplicity, F-nilpotence of pairs}

\begin{abstract}
In this article, we define three new operations on ideals which generalize integral closure and Frobenius closure of ideals, whose definitions incorporate an auxiliary ideal and a real parameter. These additional ingredients are common in adjusting old definitions of ideal closures in order to generalize them to pairs, with an eye towards further applications in algebraic geometry. In the case of tight closure, similar generalizations exist due to N.\ Hara and K.I.\ Yoshida, as well as A.\ Vraciu, and in the case of Frobenius closure, to K. Schwede. We study their basic properties and give computationally effective calculations of the adjusted tight, Frobenius, and integral closures in the case of affine semigroup rings in terms of the convex geometry of the associated exponent sets. Finally, as applications, we study submodules of the fraction field of a domain defined in terms of our adjusted closures and an $F$-nilpotent property for pairs.
\end{abstract}

\maketitle

\setcounter{tocdepth}{2}

\tableofcontents

\section{Introduction}

The tight closure of an ideal was introduced by M.\ Hochster and C.\ Huneke in \cite{HH}, and has since been indispensable in the study of rings of prime characteristic. Tight closure even illuminates properties of rings of characteristic zero via reduction to prime characteristic techniques. In particular, deep uniformity theorems such as the Brian\c{c}on-Skoda theorem (see the discussion around \cite[Theorem~5.4]{HH}) eluded proofs with algebraic techniques in characteristic zero, but have relatively simple proofs in the language of tight closure. 

In studying tight closure, one is naturally led to study test ideals, which uniformly witness tight closure relations for every ideal in the ring simultaneously. N.\ Hara \cite{H} and K.\ E.\ Smith \cite{Smith} showed that test ideals in prime characteristic correspond in a natural way to multiplier ideals in birational geometry, and so are deeply connected to geometric concepts like resolutions of singularities. The initial correspondence made by Hara and Smith was to the multiplier ideal of the trivial pair, but a strength of the theory of multiplier ideals is that one may vary an ideal of the ring and obtain information about the singularities in the locus of primes containing that ideal.  For a survey of test ideals and their connections to multiplier ideals, we recommend \cite{ST}.

Hara and K.I.\ Yoshida \cite{HY} introduced an adjusted form of tight closure, called $\mfa^t$-tight closure (see \Cref{def: HY/V tight closure} below), which incorporates an auxiliary ideal and a real parameter. The benefit to this adjustment is that now the appropriate notion of test ideals for $\mfa^t$-tight closure corresponds to the multiplier ideal of the ideal pair $(R,\mfa^t)$. One disadvantage of the Hara-Yoshida tight closure is that it is not an ideal closure operation, in that the operation is not idempotent.\footnote{We would here like to caution the reader that, throughout the article, we follow the slightly unfortunate convention of referring to the Hara-Yoshida style operations as ``closures" even though they are only \textit{pre-closures}, see \Cref{rmk: closure preclosure and ideal operation}.} 

To rectify this issue, A.\ Vraciu \cite{Vraciu} introduced a new version of $\mfa^t$-tight closure, which is a true ideal closure operation. Her version is contained in the Hara-Yoshida $\mfa^t$-tight closure, and captures subtly different information. In some cases the test ideals of the Hara-Yoshida and Vraciu versions agree, see \cite[Theorem~4.3]{Vraciu}. Vraciu also introduced a Hilbert-Kunz-like multiplicity which detects equality of her $\mfa^t$-tight closure among other properties. Our aim in this article is to repeat the adjustments of tight closure made by Hara-Yoshida and Vraciu to Frobenius and integral closure, which are two more classical closures near to the tight closure. 

\begin{maindefn}[\Cref{defn: adjusted integral closures}, \Cref{defn: adjusted frobenius closures}]
Fix a real number $t \in \mathbb{R}_{> 0}$, and let $I$ and $\mfa$ be ideals in a ring $R$, which we typically assume are of positive height. Let $R^\circ$ be the set of all elements of $R$ not in any minimal prime of $R$.  For the case of Frobenius closure, we assume that $R$ is of prime characteristic $p>0$.
\begin{enumerate}[(a)]
\item The \textbf{Hara-Yoshida $\mfa^t$-Integral closure of $I$}, denoted $\iatb$, is defined by
	\[ \iatb = \left\lbrace x \in R \left| \text{ for some } c \in R^\circ \text{ and all } n \gg 0,\, cx^n \mfa^{\lceil tn \rceil} \subset I^n \right. \right\rbrace .\]

\item The \textbf{Hara-Yoshida $\mfa^t$-Frobenius closure of $I$}, denoted $\ifat,$ is defined by
	\[ \ifat = \left\lbrace x \in R \left|\text{ for all } e \gg 0,\, x^{p^e}\mfa^{\lceil tp^e\rceil} \subset I^{\fbp{p^e}} \right.\right\rbrace. \] 

\item The \textbf{Vraciu $\mfa^t$-Frobenius closure of $I$}, denoted $\atif$, is defined by
	\[ \atif = \left\lbrace x \in R \left|\text{ for all } e \gg 0,\, x^{p^e}\mfa^{\lceil tp^e\rceil} \subset \mfa^{\lceil tp^e\rceil}I^{\fbp{p^e}} \right.\right\rbrace. \]
\end{enumerate}
\end{maindefn}

Though it seems natural to include a Vraciu version of integral closure alongside the Hara-Yoshida version, we immediately show that the obvious definition coincides with the usual integral closure (see \Cref{prop: basic properties of integral closures}). The diagram below illustrates the relationships among these operations. 

\begin{center}\begin{tikzpicture}
\node[scale=1] at (0,0) {
	\begin{tikzcd}
		I \arrow[hook]{r} & I^F \arrow[hook]{r}\arrow[hook]{d}& \atif\arrow[hook]{d}\arrow[hook]{r} & \ifat \arrow[hook]{d} \\
		\, & I^* \arrow[hook]{r}\arrow[hook]{d} & \atistar \arrow[hook]{r}\arrow[hook]{d} & \istarat \arrow[hook]{d} \\
		\, & \overline{I} \arrow[equal]{r} & \atib \arrow[hook]{r} & \iatb
	\end{tikzcd}
};
\end{tikzpicture}
\end{center}

In \cite{Sch}, K. Schwede provided a similar adjustment to the Frobenius closure to study $F$-purity of pairs $(R,\mfa^t)$. Our version is larger than his in general, due to a slightly larger exponent on the auxiliary ideal, and we provide examples to compare them in \Cref{subsubsection: sharp F-closure}. In contrast, we also show our definitions for adjusted tight and integral closure are insensitive to certain linear changes in the exponent on the auxiliary ideal (see \Cref{lem: multiplier closures can have linear shift in exponent}).

Throughout this article, we show that the properties of these new adjusted closures are similar to both their unadjusted versions as well as the adjusted tight closures of Hara-Yoshida and Vraciu. One unfortunate aspect of the adjusted closure operations are that they are even more difficult to compute than their usual counterparts. However, in some cases, they are related to colon ideals of the usual closure operation. Intriguingly, we also relate the adjusted tight and integral closure to the rational powers $\mfa_{t}$ of an ideal (see \Cref{defn: rational powers}).

\begin{mainthm}[\Cref{prop: integral at-closure as colon ideal}, \Cref{tight at-closure as colon ideal}, \Cref{F-principal}]
Let $R$ be a ring and $I,\mfa\subset R$ be ideals of positive height. Further, let $t> 1$  be a rational number. Then,
\begin{enumerate}[(a)]
\item $\overline{I}^{\mfa^t} \subset \left(
\overline{I}^{\mfa^{t-1}}:\mfa 
\right)$ with equality if $\mfa$ is principal or if $t$ is an integer, and
\item $\overline{I}^{\mfa^t} \subset (\overline{I}:\mfa_t) \subset \overline{I}^{\mfa^{\left\lceil t \right\rceil}}$.
\end{enumerate}
Further, if $R$ is of prime characteristic $p>0$, then 
\begin{enumerate}[(a)]
\addtocounter{enumi}{2}
\item $I^{*\mfa^t} \subset (I^*:\mfa_t)$, and
\item $I^{F\mfa^t} = \left(I^{F\mfa^{t-1}}:\mfa\right)$ and $^{\mfa^t}I^F = \left(^{\mfa^{t-1}}(\mfa I)^F:\mfa\right)$ if $t > \mu(\mfa)$.
\end{enumerate}
\end{mainthm}

We also demonstrate a parallel between the valuation theory involved in the computation of ordinary integral closure and the Hara-Yoshida integral closure which we define.

\begin{mainthm}[\Cref{valuative criterion for integral closure:label}]
	Let $R$ be a ring, $I,\mfa\subset R$ ideals of positive height, and $t\in \R_{>0}$.  Let $\mathcal{RV}(I)$ be the set of Rees valuations of $I$.
	An element $x\in R$ is in $\overline{I}^{\mfa^t}$ if and only if for all $\nu\in\mathcal{RV}(I)$, $\nu(x)\geq \nu(I)-t\nu(\mfa)$.
\end{mainthm}

In \Cref{sec: comparing the new closures}, we compare the new closure operations to each other in detail. Perhaps of most interest in this section, we develop an $\mfa^t$-analogue for the tight closure Brian\c{c}on-Skoda theorem of Hochster-Huneke (\cite[Thm.~5.4]{HH}), which controls the integral closure of powers of an ideal by fitting them inside the tight closure of a smaller power of the same ideal. In \cite{HunekeUniformity}, Huneke also provided a uniform Brian\c{c}on-Skoda theorem which, under some mild hypotheses on the ring, makes uniform the shifted exponent. We obtain a similar uniform tight closure Brian\c{c}on-Skoda theorem for our adjusted closure operations.

\begin{mainthm}[\Cref{BS thm}]
	Let $\mfa, I$ be ideals of positive height in a ring $R$ of prime characteristic $p>0$. Let $I$ be generated by $d$ elements and let $t \in \mathbb{R}_{> 0}$. 
	\begin{enumerate}[{\rm(a)}]
		\item Then $\overline{I^{d+m}}^{\mfa^t} \subset$ $(I^{m+1})^{*\mfa^t}$ for all $m \ge 0,$ and
				
		\item $\overline{I^{d+m}}^{\mfa^t} \subset (I^m)^{F\mfa^t},$ for all $m \geq 1.$
	\end{enumerate}
Further, suppose $(R,\mfm)$ is local of dimension $\dim R = n$ and $R/\mathfrak{m}$ is infinite.
\begin{enumerate}[{\rm(a)}]
\addtocounter{enumi}{2}
\item Then $\overline{I^{n+m}}^{\mfa^t} \subset (I^{m+1})^{*\mfa^t}$ for all $m \ge 0$, and
\item $\overline{I^{n+m}}^{\mfa^t} \subset (I^m)^{F\mfa^t}$ for all $m \ge 1$. 
\end{enumerate}
\end{mainthm}
We also provide a graded version of the Brian\c{c}on-Skoda theorem for the adjusted tight closure (see \Cref{prop: graded briancon-skoda}) along the lines of \cite{Smith97}, which analyzes which graded components of the ring which are forced into the tight closure of an ideal primary to the homogeneous maximal ideal. 

As an application of our results, we generalize another collection of classical definitions to the pairs setting with our new closures. In particular, it is elementary to show that the normalization $\overline{R}$ of a domain $R$ is the set of fractions $x/y$ in the fraction field of $R$ such that $x \in \overline{(y)}$. Similarly, the weak normalization ${}^*R$ of a domain $R$ is the set of fractions $x/y$ such that $x \in (y)^F$. With these perspectives in hand, we define the following.

\begin{maindefn}
Let $R$ be a domain, $t \in \mbr_{>0}$, and let $\mfa$ be an ideal of positive height in $R$. Then, we define the \textbf{Hara-Yoshida $\mfa^t$-normalization of $R$} to be \[\overline{R}^{\mfa^t} = \left\lbrace  x/y \in \Frac(R)\, \left|\, x \in \overline{(y)}^{\mfa^t}\right.\right\rbrace.\] Analogously, we also define the \textbf{Hara-Yoshida $\mfa^t$-weak normalization} $R^{F\mfa^t}$, and the \textbf{Vraciu $\mfa^t$-weak normalization} ${}^{\mfa^t}R^F$ for domains of prime characteristic $p>0$.
\end{maindefn}

Because the adjusted tight closures of principal ideals agree with the corresponding integral closures, we do not have need to define similar objects involving the tight closures. We demonstrate that membership in these sets is well-defined, and further, that each forms an $R$-submodule of the fraction field (\Cref{prop: a^t integral closure of ring is submodule}), and indeed, $^{\mfa^t}R^F$ is even a subring (\Cref{prop: V a^t Frob closure of ring is subring of frac field}) of $\overline{R}.$ Through \Cref{ex: little dipper example 1} and \Cref{ex: little dipper example 2}, we show that each of the natural inclusions in the diagram below can be strict.

\begin{center}
\begin{tikzpicture}
\node at (0,0) {
\begin{tikzcd}
R \arrow[hook]{r} & {}^*R \arrow[hook]{r} & {}^{\mfa^t}R^F \arrow[hook]{r} \arrow[hook]{d} & R^{F\mfa^t} \arrow[hook]{d} \\
\,& \, & \overline{R} \arrow[hook]{r} & \overline{R}^{\mfa^t}
\end{tikzcd}
};
\end{tikzpicture}
\end{center}

Ideal closures in prime characteristic are intimately related to \textit{$F$-singularities}, which are singularities defined in terms of the Frobenius map on a ring of prime characteristic. For details about $F$-singularities in general, we direct the reader to recent surveys by Schwede-Smith \cite{SchwedeSmith} and L. Ma-T. Polstra \cite{MaPolstra}. One recently-defined $F$-singularity of interest is $F$-nilpotence (see \Cref{defn: F-nilp}), defined by V. Srinivas-S. Takagi in \cite{SriTak}, which is defined in terms of the Frobenius action on local cohomology, or equivalently (under some mild hypotheses) to equality of tight and Frobenius closure of parameter ideals (see \cite{PQ}). 

Just as the Hara-Yoshida tight closure and Schwede's sharp $F$-closure have been used to study $F$-singularities of pairs, we use the Hara-Yoshida tight closure to define $F$-nilpotence of pairs. 
 
\begin{maindefn}
Let $(R,\mfm)$ be an excellent, equidimensional, reduced local ring, $t \in \mathbb{R}_{> 0}$, and suppose $\mfa$ is an ideal of $R$ of positive height. We say the pair $(R,\mfa^t)$ is \textbf{$F$-nilpotent} if $\mfq^{*\mfa^t}=\mfq^F$ for all ideals $\mfq \subset R$ generated by a system of parameters for $R$.
\end{maindefn} 

Alongside studying basic properties of this new singularity of pairs, we also demonstrate an equivalence of the definition given above to a condition involving the top local cohomology module in the case that the ring is Cohen-Macaulay, although we expect the theorem to hold more generally. 

\begin{mainthm}[\Cref{thm: CM a^t F-nilpotent is characterized by LC}]
Let $(R,\mfm)$ be an excellent, Cohen-Macaulay local domain of prime characteristic $p>0$. Then, $(R,\mfa^t)$ is $F$-nilpotent if and only if $0^{*\mfa^t}_{H^d_\mfm(R)} = 0^F_{H^d_\mfm(R)}$.
\end{mainthm}

In Section \ref{sec:semigroup}, we carefully study the adjusted closure operations in the setting of monomial ideals in affine semigroup rings. In particular, we are able to provide explicit calculations of the adjusted closures in terms of the convex geometry of the associated exponent sets of the ideals involved, which has the added effect of providing concrete examples of the Hara-Yoshida and Vraciu tight closure which were otherwise missing from the literature. These results also extend well-known results for computing the ordinary integral closure and Hilbert-Kunz multiplicity.

\begin{mainthm}[\Cref{corollary - normal semigroup}, \Cref{thm: at integral closure convex geometry}]
Fix $t\in \R_{>0}$ and a field $k$ of prime characteristic $p>0$. Let  $\sigma^\vee\subset \R^d$ be a rational polyhedral cone, $S = \sigma^\vee\cap \mathbb{Z}^d$ the semigroup of lattice points in $\sigma^\vee$, $R=k[S]$, and let $I$ and $\mfa$ be monomial ideals in $R$. In particular, there are finite subsets $\mathcal{A}$ and $\mathcal{B}$ of $S$ so that $I=(x^{\vec b}\mid \vec b\in \mathcal{B})$ and $\mfa= (x^{\vec a} \mid \vec a\in \mathcal{A})$. We write $H$ for the convex hull of $\mathcal{A}$ in $\R^d$. For a monomial $x^{\vec v} \in R$, we have 
\begin{enumerate}[(a)]
	\item $x^{\vec v}\in \ifat$ if and only if $x^{\vec v} \in I^{*\mfa^t}$ if and only if $\vec v + tH\subset \mathcal{B}+\sigma^\vee$,
	\item $x^{\vec v}\in {{}^{\mfa^t}I^*}$ if and only if $\vec v + tH\subset \mathcal{B}+tH+\sigma^\vee$,
	\item if $x^{\vec v}\in \atif$ then $\vec v + tH\subset \mathcal{B}+tH+\sigma^\vee$,  
	\item if there exists $s>t$ such that $\vec v + tH\subset \mathcal{B}+sH+\sigma^\vee$, then $x^{\vec v}\in \atif$, and 
	\item $x^{\vec v} \in \overline{I}^{\mfa^t}$ if and only if $\vec v+tH \subset \operatorname{Hull} (\operatorname{Exp} I)$.
\end{enumerate}
\end{mainthm}

These theorems show that in the affine semigroup ring case, the formulas which compute the $\mfa^t$-tight and integral closures are independent of the characteristic of the underlying field. However, we show in an example (see \Cref{MegaExample}) that the Vraciu $\mfa^t$-Frobenius closure can depend on the characteristic. We also show in the affine semigroup setting that the joint Hilbert-Kunz multiplicity introduced by Vraciu is computable via the volume of a certain polytope, which reflects the situation for usual Hilbert-Kunz multiplicity by work of K.\ Eto \cite[Theorem~2.2]{eto} and K-i.\ Watanabe \cite[Theorem~2.1]{Watanabe}. 

\begin{mainthm}[\Cref{joint HK volume}]
Fix $t\in \R_{>0}$ and a field $k$ of prime characteristic $p>0$. Let $S\subset \mathbb{N}^d$ be an affine semigroup, $R=k[S]$, and let $I$ and $\mfa$ be monomial ideals in $R$ of finite colength. Let $\mathcal{A}$ and $\mathcal{B}$ be finite subsets of $S$ such that $I=(x^{\vec b}\mid \vec b\in \mathcal{B})$ and $\mfa= (x^{\vec a} \mid \vec a\in \mathcal{A})$. We write $H$ for the convex hull of $\mathcal{A}$ and $\sigma^\vee$ for the cone defining $S$. If we set \[ \mathcal{P}_t = \sigma^\vee \setminus \left(\mathcal{B} + tH + \sigma^\vee\right),\] then the joint Hilbert-Kunz multiplicity of $(\mfa^t,I)$ is given by
\[ e_{HK}(\mfa^t;I) = \lim\limits_{q \rightarrow \infty} \frac{1}{p^{ed}} \lambda_R \left(\frac{R}{\mfa^{\lceil tq \rceil}I^{[p^e]}}\right) = \operatorname{vol} \overline{\mathcal{P}_t}, \] 
where $\overline{\mathcal{P}_t}$ is the closure of $\mathcal{P}_t$ in the usual Euclidean topology on $\mathbb{R}^d$ and $\operatorname{vol}$ denotes the usual Euclidean volume.
\end{mainthm}

Finally, in order to demonstrate the theory to the fullest, we provide a thorough analysis of the adjusted closures in a specific affine semigroup ring, see \Cref{MegaExample}. These examples demonstrate discrete shifts in the adjusted closures in finitely many intervals for the parameter $t$. For this example, we also compute the joint Hilbert-Kunz multiplicity as a function of $t$. 

\begin{ack} The authors are grateful to Rankeya Datta, Mark Johnson, Lance E.\ Miller, Janet Vassilev, and Thomas Polstra for useful conversations. Furthermore, we want to thank Kevin Tucker and Yairon Cid-Ruiz for suggestions about the uniform Brian\c{c}on-Skoda theorem. The first author was supported by grant number CEX2021-001142-S (Excelencia Severo Ochoa) funded by MICIU/AEI/10.13039/501100011033 (Ministerio de Ciencia, Innovaci\'on y Universidades, Spain). 
\end{ack}
\section{Preliminaries}

\begin{conv}
First, $0 \in \mbn$ and $p$ will always denote a positive prime integer. All rings are assumed to be commutative with unit and noetherian. For a ring $R$, we will let $R^\circ$ denote the multiplicative set of elements outside of the union of the minimal primes of $R$. That is, $R^\circ = R\setminus \bigcup_{\mfp \in \operatorname{Min}(R)}\mfp$. For an ideal $I\subset R$, $\mu(I)$ denotes the minimal number of generators of $I$. We will use $\lambda_R(\bullet)$ for the length of an $R$-module. Also, non-positive powers of ideals will be interpreted as the entire ring.

For an $\mathbb{N}$-graded ring $R=\bigoplus_{n\geq 0} R_n$ and a real number $t$, we will use $R_{<t}$  to denote the $R_0$-submodule generated by elements of $R$ of degree less than $t$ (similarly $R_{\le t}$, $R_{\ge t}$, $R_{>t}$). We will slightly overload this notation for the usual number rings $\mathbb{Z}$, $\mathbb{Q}$, $\mathbb{R}$, e.g. using $\mbr_{>0}$ for the set of positive real numbers.
\end{conv}

Our article primarily concerns ideal closure operations.

\begin{defn}
Let $R$ be a ring and let $I\subset R$ be an ideal. We will call a function from the set of ideals of $R$ to itself an \textbf{operation on ideals}. An operation on ideals $\bullet^c$ is called a \textbf{pre-closure} if \begin{enumerate}[{\rm(a)}]
\item $\bullet^c$ is extensive, that is, $I\subset I^c$, and 
\item $\bullet^c$ is order-preserving, that is, $I\subset J$ implies $I^c\subset J^c$.
\end{enumerate}

If a preclosure $\bullet^c$ additionally satisfies \begin{enumerate}[{\rm(a)}]
\addtocounter{enumi}{2}
\item $\bullet^c$ is idempotent, that is, $(I^c)^c = I^c$,
\end{enumerate}

then we say $\bullet^c$ is an \textbf{ideal closure operation}. 
\end{defn}

We will soon survey the definitions of several important ideal closure operations in rings of prime characteristic $p>0$.

\begin{remark}\label{rmk: closure preclosure and ideal operation}
Even though several operations on ideals we consider later are simply pre-closures, due to prior authors' convention, they are commonly referred to as closures. We will indicate when defining an operation whether it is a pre-closure or closure, but will typically refer to both kinds of operations as ``closure operations" in the text.
\end{remark}

We begin with a corollary to the usual Determinantal Trick for integral closure. While a similar result (without proof) appears in \cite[Proposition 1.58]{Vasc} for integrally closed ideals, we provide one for completeness. The ideal $IJ:_RJ$ studied below is called the \textit{$J$-basically full closure of $I$}, so one can summarize the following theorem by saying the $J$-basically full closure of $I$ is contained in its integral closure. See \cite[Section~4]{EJRG} for a modern treatment of the basically fully closure.

\begin{theorem} \label{det}
If $R$ is a ring and $I,J\subset R$ are ideals of positive height, then $(IJ:_RJ)\subset \overline{I}$, the integral closure of $I$.
\end{theorem}

\begin{proof} Let $x\in (IJ:_RJ)$, so that $xJ\subset IJ$.  Let $\varphi:J\to J$ be given by $\varphi(y)=xy$. By \cite[Lemma~2.1.8]{hunekeSwanson}, there exist $r_i\in I^i$ such that $\varphi^n+r_1\varphi^{n-1}+\cdots +r_n = 0$.  Let $y\in J\cap R^\circ$.  Now we have that
\[ 0 = (\varphi^n+r_1\varphi^{n-1}+\cdots +r_n)(y) = yx^n+r_1yx^{n-1}+\cdots + r_ny = (x^n+r_1x^{n-1}+\cdots +r_n)y.
\]
Let $\mfp$ be a minimal prime of $R$.  Reducing the above equation modulo $\mfp$, we have that, in the domain $R/\mfp$, \[(x^n+r_1x^{n-1}+\cdots +r_n +\mfp)(y+\mfp) = \mfp.\]  Since $y+\mfp \neq \mfp$, it must be the case that $x^n+r_1x^{n-1}+\cdots +r_n \in \mfp$.  This holds for all minimal primes $\mfp$, and therefore we have 
\[x^n+r_1x^{n-1}+\cdots +r_n \in \bigcap_{\mfp\in\mathrm{Spec}\,R}\mfp=\sqrt{0}.
\]
Thus, there exists $m\in \N$ such that $0 = (x^n+r_1x^{n-1}+\cdots +r_n)^m = x^{mn} + s_1x^{mn-1}+\cdots + s_{mn}$ where each $s_j$ is a sum of terms of the form $r_{i_1}r_{i_2}\cdots r_{i_m}$ with $i_1+i_2+\cdots +i_m = j$.  Hence $s_j\in I^j$.  This shows that $x\in \overline{I}$.
\end{proof}

We now review the classical notions of tight and Frobenius closure in prime characteristic.

\begin{defn}\label{defn: classical ideal closures}
Let $R$ be a ring of prime characteristic $p>0$ and let $I\subset R$ be an ideal. \begin{itemize}
\item The \textbf{$e$-th Frobenius power of $I$}, denoted $I^{\fbp{p^e}},$ is the ideal of $R$ generated by $\{r^{p^e} \mid r \in I\}$. 
\item The \textbf{Frobenius closure of $I$}, denoted $I^F$, is the ideal of $R$ defined by \[
I^F=\left\lbrace x\in R \left|\text{ for some } e \in \mathbb{N}, \,x^{p^e} \in I^{\fbp{p^e}} \right.\right\rbrace.\]
\item The \textbf{tight closure of $I$}, denoted $I^*$, is the ideal of $R$ defined by \[
I^* = \left\lbrace x \in R \left| \text{ for some } c\in R^\circ \text{ and all } e\gg 0,\, cx^{p^e} \in I^{\fbp{p^e}}\right.\right\rbrace.\]
\end{itemize} Both Frobenius and tight closure are ideal closure operations.
\end{defn}

For an ideal $I$ in a ring $R$ of prime characteristic $p>0$, we have that $I\subset I^F \subset I^*\subset \overline{I}$, but each of these containments are typically strict. However, in a regular ring, $I=I^*$ for all $I$ by \cite[Theorem~4.4]{HH}, so many singularity types in prime characteristic are defined in terms of whether some or all ideals are closed with respect to these operations. For example, a ring is called \textit{weakly $F$-regular} if $I^*=I$ for all ideals $I$. 

We will later discuss $F$-nilpotent rings, which are another singularity type that can be defined in terms of ideal closures. Initially studied by Blickle-Bondu in \cite{BB} under the name \textit{close to $F$-rational}, $F$-nilpotent rings were systematically studied by Srinivas-Takagi in \cite{SriTak}, who showed that low-dimensional rings in characteristic zero which are of $F$-nilpotent type satisfy certain Hodge-theoretic conditions related to the Weak Ordinarity conjecture. The initial definition of $F$-nilpotent rings involves the canonical Frobenius action on local cohomology, the technicalities of which we avoid here, but Polstra-Quy showed in \cite{PQ} that the following definition is equivalent to the one given by Blicke-Bondu and Srinivas-Takagi in the case that $R$ is excellent and equidimensional.

\begin{defn}\label{defn: F-nilpotent ring}
Let $(R,\mfm)$ be an excellent, equidimensional local ring. Then, we say that $R$ is \textbf{$F$-nilpotent} if $\mfq^*=\mfq^F$ for all ideals $\mfq$ of $R$ generated by a system of parameters.
\end{defn}

The computation of tight and integral closure reduce readily to domains (see \cite[Proposition~6.25.(a)]{HH} and \cite[Proposition~1.1.5]{hunekeSwanson}) as elements of the ring $R$ are in the tight or integral closure of an ideal $I$ if and only if they are in the tight or integral closure of $(I+\mfp)/\mfp$ in $R/\mfp$ where $\mfp$ ranges over all minimal primes of $R$ ($\operatorname{Min} R$). It is well-known that the Frobenius closure does not exhibit this property, that is, the inclusion \[I^F \subset \left\lbrace x \in R \mid x+\mathfrak{p} \in (I+\mathfrak{p})^F/\mathfrak{p} \text{ for all } \mathfrak{p}\in \operatorname{Min} R\right\rbrace\] may be proper; we provide an explicit example below. 

\begin{xmp}\label{example: frob closure can't be checked mod minimal primes}
Let $R=\mathbb{F}_3[X,Y,Z]/(Y^6-X^3,X^2-Z^2)$ and let $x,y,$ and $z$ denote the images of $X,Y,$ and $Z$ respectively in $R.$ Let $I=(x)$. It is easy to check that $\operatorname{Min} R=\{(x-z,y^2-z),(x+z,y^2+z) \}$ and that $z \in (x)^*\setminus (x)^F.$ Indeed, if $z \in (x)^F,$ then $z^{3^e} \in I^{[3^e]} = (x^{3^e})$ for some $e.$ 
Write $3^e = 2k+1$ for some $k.$ Then
\[ z^{3^e} = z^{2k+1} = (z^2)^kz = (x^2)^kz. \]
This element belongs to the ideal $(x^{3^e})$ only if $2k \geq 3^e,$ which is a contradiction. Note that modulo either minimal prime, the image of $z$ is a unit times the image of $x$, so 
\[ z+\mathfrak{p} \in ((x)+\mathfrak{p})/\mathfrak{p} \subset ((x)+\mathfrak{p})^F/\mathfrak{p}, \] 
but $z\not \in (x)^F$.
\end{xmp}

Finally, before we begin defining the adjusted closure operations, we remind the reader of the rational powers of an ideal, which are related to the normalization of the Rees algebra of that ideal. The rational powers of an ideal were introduced by Lejuene-Jalabert and Tessier in 1973 (see \cite{LJT}). 

\begin{defn}\label{defn: rational powers}
Let $R$ be a ring, $\mfa\subset R$ be an ideal, and $t=r/s \in \mathbb{Q}_{\ge 0}$. Then, the \textbf{$t$-th rational power of} $\mfa$, denoted $\mfa_t$, is defined by $\mfa_t = \left\lbrace
x\in R
\left|\, x^s \in \overline{\mfa^r}
\right.\right\rbrace.$
\end{defn} 

Despite being introduced over fifty years ago, the rational powers of ideals have unfortunately received little attention until recently. See \cite[Section~10.5]{hunekeSwanson} for a treatment of rational powers. For some modern results utilizing them, see \cite{CC} and \cite{BDHM}.

\section{The closure operations}

In this section we define new operations on ideals whose origin is a formal adaptation of the tight closure adjustments of Hara-Yoshida \cite{HY} and Vraciu \cite{Vraciu} (see \Cref{def: HY/V tight closure}) to integral and Frobenius closures. We want to remind the reader of \Cref{rmk: closure preclosure and ideal operation}, since in this section we use the word closure to refer to both ideal closure operations and pre-closures.

\subsection{\texorpdfstring{The $\mfa^t$-integral closures}{The adjusted integral closures}}

 We begin with the adjustments to the ordinary integral closure of an ideal $I$.  In this subsection, we consider rings $R$ of arbitrary characteristic, and the ideals $I,\mfa\subset R$ we consider will typically be assumed to be of positive height, so that they intersect $R^\circ$.

\begin{defn}\label{defn: adjusted integral closures}
	Let $I,\mfa\subset R$ be ideals and let $t \in \mathbb{R}_{>0}$. Then, we define the \textbf{Hara-Yoshida $\mfa^t$-Integral closure of $I$}, denoted $\iatb$, by
	\[ \iatb = \left\lbrace x \in R \left| \text{ for some } c \in R^\circ \text{ and all } n \gg 0,\, cx^n \mfa^{\lceil tn \rceil} \subset I^n \right. \right\rbrace .\]
	Similarly, we define the \textbf{Vraciu $\mfa^t$-Integral closure of $I$}, denoted $\atib$, by 
	\[ \atib = \left\lbrace x \in R \left| \text{ for some } c \in R^\circ \text{ and all } n \gg 0,\, cx^n\mfa^{\lceil tn\rceil} \subset \mfa^{\lceil tn\rceil}I^n \right. \right\rbrace . \] We will later show that the Hara-Yoshida $\mfa^t$-integral closure is a pre-closure (\Cref{prop: basic properties of integral closures}) and not an ideal closure operation (\Cref{xmp: hy integral closure is not idempotent}). Further, in virtue of being equal to the classical integral closure (\Cref{prop: basic properties of integral closures}), the Vraciu $\mfa^t$-integral closure is an ideal closure operation. 
\end{defn}

We first note that the ``for all $n \gg 0$" in the definitions above may be strengthened to ``for all $n$", so long as $I$ is of positive height.

\begin{proposition}\label{prop: integral closure for all n works}
Let $R$ be a ring and $I,\mfa\subset R$ be ideals where $I$ has positive height, and let $t \in \mathbb{R}_{>0}$. Then $x \in \iatb$ (respectively $x \in \atib$) if and only if there exists $c \in R^\circ$ such that, for all $n \geq 0$, $cx^n \mfa^{\lceil tn \rceil} \subset I^n$ (respectively $cx^n \mfa^{\lceil tn \rceil} \subset I^n \mfa^{\lceil tn \rceil}$).
\end{proposition}

\begin{proof}
The proofs of both parts are similar, so we provide the proof for the Hara-Yoshida version only. Let $x \in \iatb.$ Then there exist $n_0 \in \mathbb{N}$ and $c \in R^\circ$ such that $cx^n \mfa^{\lceil tn \rceil} \subset I^n$ for all $n \geq n_0.$ Choose $b \in R^\circ \cap I^{n_0}$. Then $bc\in R^\circ$ and $bcx^n \mfa^{\lceil tn \rceil} \subset I^n$ for all $n \geq 0$. 
\end{proof}

We now record several basic properties of these closure operations. 

\begin{prop}\label{prop: basic properties of integral closures}
	Let $R$ be a ring, $I,\mfa\subset R$ be ideals of positive height, and $t \in \mathbb{R}_{>0}$.
	\begin{enumerate}[{\rm(a)}]
		\item $\atib$ and $\iatb$ are ideals of $R.$ 

		\item $\overline{I} = \atib \subset \iatb.$

		\item If $I\subset J$, $\mathfrak{b}\subset\mfa$, and $t \le s$, then $\iatb \subset \overline{J}^{\mathfrak{b}^s}$.	
	\end{enumerate}
\end{prop}

\begin{proof}
	We will prove each in turn.
	\begin{enumerate}[(a)]
		\item In order to show that $\iatb$ is an ideal, it is enough to check that if $x,y \in \iatb,$ then $x+y \in \iatb.$ As $x,y \in \iatb,$ there exist $c,d \in R^\circ$ such that $cx^n \mfa^{\lceil tn \rceil} \subset I^n$ and $dy^n \mfa^{\lceil tn \rceil} \subset I^n$ for all $n \geq 0.$ Choose $c' \in \mfa \cap R^\circ.$ Then
		\begin{align*}
			c'cd(x+y)^n \mfa^{\lceil tn \rceil} 
			= \sum_{j=0}^{n} \binom{n}{j} (cx^j) (dy^{n-j}) c' \mfa^{\lceil tn \rceil} 
			&\subset \sum_{j=0}^{n} \binom{n}{j} (cx^j) (dy^{n-j}) \mfa^{\lceil tn \rceil +1}  \\
			&\subset \sum_{j=0}^{n} \binom{n}{j} (cx^j \mfa^{\lceil tj \rceil}) (dy^{n-j} \mfa^{\lceil t(n-j) \rceil}) 
			\subset I^n
		\end{align*}
		since for any $j$, $\lceil tn \rceil +1\geq \lceil tj \rceil + \lceil t(n-j) \rceil $
		Thus, $\iatb$ is an ideal.
		Similarly, let $x,y \in \atib.$ Then there exist $c,d \in R^\circ$ such that $cx^n \mfa^{\lceil tn \rceil} \subset \mfa^{\lceil tn \rceil} I^n,$ and $dy^n \mfa^{\lceil tn \rceil} \subset \mfa^{\lceil tn \rceil} I^n$ for all $n \geq 0.$ Choose $c' \in \mfa \cap R^\circ.$ Then 
		\begin{align*}
			c'cd(x+y)^n \mfa^{\lceil tn \rceil} 
			= \sum_{j=0}^{n} \binom{n}{j} (cx^j) (dy^{n-j}) c' \mfa^{\lceil tn \rceil} 
			&\subset \sum_{j=0}^{n} \binom{n}{j} (cx^j \mfa^{\lceil tj \rceil}) (dy^{n-j} \mfa^{\lceil t(n-j) \rceil}) \\
			&\subset \sum_{j=0}^{n} \mfa^{\lceil tj \rceil} \mfa^{\lceil t(n-j) \rceil} I^n  \subset \mfa^{\lceil tn\rceil}I^n.
		\end{align*}
		Thus, $\atib$ is an ideal.
		
		\item Since $I^n \subset I^n \mfa^{\lceil tn \rceil}:\mfa^{\lceil tn \rceil} \subset I^n:\mfa^{\lceil tn \rceil}$, for all $n$, we have $\overline{I} \subset \atib \subset \iatb$. Now let $x \in \atib.$ Then there exists $c \in R^\circ$ such that for all $n$, $cx^n\mfa^{\lceil tn\rceil} \subset \mfa^{\lceil tn\rceil}I^n$. Using \Cref{det}, we get $cx^n \in \overline{I^n}$ for all $n$ implying that $x \in \overline{I}$ (see \cite[Corollary~6.8.12]{hunekeSwanson}).
							
		\item Since $I\subset J$, $I^n\subset J^n$ for all $n$, so we have $I^n:\mfa^{\lceil tn \rceil} \subset J^n:\mathfrak{b}^{\lceil sn \rceil}$ for all $n$. Thus, $\iatb \subset \overline{J}^{\mathfrak{b}^s}$. \qedhere 
	\end{enumerate}
\end{proof}

\begin{remark}
Due to the fact that in most cases of consideration (e.g. nonzero ideals in an integral domain) we have that $^{\mfa^t}\overline{I} = \overline{I}$, we will not typically have reason to discuss the Vraciu $\mfa^t$-integral closure later in the article, and will refer to $\overline{I}^{\mfa^t}$ as simply the $\mfa^t$-integral closure of $I$.
\end{remark}

\begin{prop}\label{prop: integral at-closure as colon ideal}
Let $R$ be a ring, $I,\mfa\subset R$ be ideals of positive height and $t\in \R_{> 0}$. 
\begin{enumerate}[{\rm(a)}]
\item If $t>1$ then $\overline{I}^{\mfa^t} \subset \left(\overline{I}^{\mfa^{t-1}} : \mfa\right)$, and we have equality if $\mfa$ is a principal ideal or $t\in \Z$.
\item If $t \in \mathbb{Q}_{>0}$, then $\overline{I}^{\mfa^t} \subset (\overline{I} : \mfa_t)\subset \overline{I}^{\mfa^{\lceil t\rceil}},$ where $\mfa_t$ is the $t$-th rational power of $\mfa$. 
	\item If $t\in\Z_{\geq 1}$, then $\overline{I}^{\mfa^t} = (\overline{I}:\mfa^t).$
	
\end{enumerate}
\end{prop}

\begin{proof} We will show each in turn.
\begin{enumerate}[(a)] 
	\item Let $x\in \overline{I}^{\mfa^t}$, so that there exists $c\in R^\circ$ such that for all $n\gg 0$, $cx^n\mfa^{\lceil tn\rceil}\subset I^n$. For any $y\in \mfa$ and $n$ we have 
		\[c(xy)^n\mfa^{\lceil (t-1)n\rceil}\subset cx^n\mfa^n\mfa^{\lceil (t-1)n\rceil}
		 =cx^n\mfa^{\lceil tn\rceil} \subset I^n.\]
		 Therefore $xy\in \overline{I}^{\mfa^{t-1}}$.  Hence $x\in \left(\overline{I}^{\mfa^{t-1}} : \mfa\right)$.

		 Suppose $\mfa = (f)$ is principal, and let $x\in \left(\overline{I}^{\mfa^{t-1}} : \mfa\right)$, so that $xf\in \overline{I}^{\mfa^{t-1}}$.  Let $c\in R^\circ$ such that $c(xf)^n\mfa^{\lceil (t-1)n\rceil}\subset I^n$ for $n$.  Then for all $n$,
		 \[cx^n\mfa^{\lceil tn\rceil} = cx^n(f^{\lceil tn\rceil}) = c(xf)^n(f^{\lceil (t-1)n\rceil}) = c(xf)^n\mfa^{\lceil (t-1)n\rceil}\subset I^n.\]
		 Therefore $x\in \overline{I}^{\mfa^t}$, and so  $\left(\overline{I}^{\mfa^{t-1}} : \mfa\right)\subset \overline{I}^{\mfa^t}$, and we have equality.

		Finally, suppose $t\in \Z$ and $x\in  \left(\overline{I}^{\mfa^{t-1}} : \mfa\right)$.  Let $y_1,\ldots, y_r$ be the generators of $\mfa$.  For each $i$, $xy_i\in \overline{I}^{\mfa^{t-1}}$, and so there exists $c_i\in R^{\circ}$ such that for all $n\geq 0$, $c_i(xy_i)^n\mfa^{(t-1)n}\subset I^n$.  Let $c=\prod_ic_i\in R^{\circ}$, and fix $n\geq 0$.  All generators of $\mfa^{n}$ are of the form $\prod_i y_i^{n_i}$ with $\sum_in_i = n$.  For such a generator, we have that
		\[cx^n\left(\prod_i y_i^{n_i}\right)\mfa^{(t-1)n} = \prod_i\left(c_ix^{n_i}y^{n_i}\mfa^{(t-1)n_i}\right)\subset \prod_i I^{n_i} = I^n. \]
		Therefore $cx^n\mfa^{tn} = cx^n\mfa^n\mfa^{(t-1)n} \subset I^n$, meaning $x\in\overline{I}^{\mfa^t}$.  Hence  $\left(\overline{I}^{\mfa^{t-1}} : \mfa\right)\subset \overline{I}^{\mfa^t}$. 

\item Suppose $x$ is such that, for some $c \in R^\circ$ and all $n \ge 0$, we have $cx^n\mfa^{\lceil tn\rceil}\subset I^n$, and let $y\in \mfa_t$.  By \cite[Lemma~2.8]{Tay}, there exists $d\in R^\circ$ such that for all $n\gg 0$, $dy^n\in \mfa^{\lceil tn\rceil}$.  Now $cd\in R^\circ$ and for $n\gg 0$, $cd(xy)^n = cx^ndy^n\in cx^n\mfa^{\lceil tn\rceil}\subset I^n$.  Therefore $xy\in \overline{I}$ and so $\overline{I}^{\mfa^t} \subset (\overline{I} : \mfa_t)$.

Suppose that $x\in (\overline{I}:\mfa_t)$.  Let $c\in R^\circ$ such that for $n\gg 0$, $c(\overline{I})^n\subset I^n$ (\cite[Corollary 6.8.12]{hunekeSwanson}).  Now, for $n\gg 0$,
\[cx^n\mfa^{\lceil t\rceil n}= c(x\mfa^{\lceil t\rceil})^n\subset c(x\mfa_{\lceil t\rceil})^n \subset  c(x\mfa_{t})^n \subset c(\overline{I})^n \subset I^n.\]
Therefore $x\in \overline{I}^{\mfa^{\lceil t\rceil}}$.

		\item Applying part (b) with $t\in \Z_{\geq 1}$, we have that 
		$\overline{I}^{\mfa^t} = (\overline{I} : \mfa_t) = (\overline{I} : \overline{\mfa^t}) = (\overline{I} : \mfa^t).$\qedhere
		  
\end{enumerate}

\end{proof}

\begin{remark} 
	The reverse containment $(\overline{I}:\mfa_t)\subset \overline{I}^{\mfa^t}$ in \Cref{prop: integral at-closure as colon ideal} is not true in general if $t$ is not an integer.  For example, let $R=k[x,y]$, $\mfa = (x,y)$, $I=(x^2,y^2)$, and $1 < t < 2$.  Since $\mfa$ is generated by degree 1 elements, $\mfa_t$ is generated by elements of degree at least $t$, and so $\mfa_t\subset (x^2,xy,y^2)$.  On the other hand, $\mfa_t\supseteq \mfa_2 = \overline{\mfa^2} = (x^2,xy,y^2)$, and so we have equality.
	Therefore $\mfa_t = (x^2,xy,y^2) = \overline{I}$, and so $(\overline{I}:\mfa_t) = R$.  However, $1\notin \overline{I}^{\mfa^t}$.  If it were, then there would exist a nonzero $c\in R$ such that for all large $n$, $c\mfa^{\lceil tn\rceil}\subset I^n$.  However, for any $n\in\N$, $\mfa^{\lceil tn\rceil}$ is generated by elements of degree $\lceil tn\rceil$, and so $c\mfa^{\lceil tn\rceil}$ contains elements of degree $\deg(c)+\lceil tn\rceil$.  Since $I$ is generated by elements of degree 2, $I^n$ is generated by homogeneous elements of degree $2n$.  Thus, we would have that $2n \leq \deg (c)+  \lceil tn\rceil$ for all large $n$, meaning that $\deg(c)\geq \lceil (2-t)n\rceil$ for all large $n$, a contradiction.  Therefore $(\overline{I}:\mfa_t)\not\subset \overline{I}^{\mfa^t}$. 
\end{remark}

We may use \Cref{prop: integral at-closure as colon ideal} to observe that the $\mfa^t$-integral closure is not an idempotent operation on ideals in general.

\begin{xmp}\label{xmp: hy integral closure is not idempotent}
	Let $R$ be a ring, $I\subset R$ an ideal of positive height and $f\in R$ a nonzerodivisor.  By \Cref{prop: integral at-closure as colon ideal}, $\overline{I}^{(f)} = (\overline{I}:f)$ and
\[\overline{\Big(\overline{I}^{(f)}\Big)}^{(f)} = \Big(\overline{I}^{(f)}: f\Big)
=((\overline{I}:f):f) = (\overline{I}:f^2).\]
Since in general $(\overline{I}:f)$ may be strictly smaller than $(\overline{I}:f^2)$, we have that $\overline{I}^{\mfa^t}$ may be strictly smaller than $\overline{\left(\overline{I}^{\mfa^t}\right)}^{\mfa^t}.$
\end{xmp}

We now demonstrate that the computation of the $\mfa^t$-integral closure may be reduced to a computation in domains via going modulo the minimal primes, like the usual tight and integral closures.

\begin{prop}\label{prop: hy integral closure reduces to domains}
Let $\mfa, I$ be ideals of positive height in a reduced ring $R$, $t \in \mathbb{R}_{>0}$, and $x\in R$. Then $x \in \overline{I}^{\mfa^t}$ if and only if the image of $x$ in $R/\mathfrak{p}$ is in $\overline{((I+\mfp)/\mfp)}^{\mfb^t}$ for all minimal primes $\mfp$ of $R$,  where $\mfb$ denotes the image of $\mfa$ in $R/\mfp.$
\end{prop}

\begin{proof}
	Let $x \in \overline{I}^{\mfa^t}.$ By definition, there exists $c \in R^\circ$ such that $cx^n \mfa^{\lceil tn \rceil} \subset I^n$ for all $n.$ Thus, for any minimal prime $\mfp$ of $R,$ the image of $c$ in $R/\mfp$ is in $(R/\mfp)^\circ$ and $cx^n \mfa^{\lceil tn \rceil} + \mfp \subset I^n + \mfp$ for all $n$ implying that the image of $x$ in $R/\mfp$ is in $\overline{((I+\mfp)/\mfp)}^{\mfb^t}.$
	
Conversely, assume that for each minimal prime $\mfp$ of $R,$ the image of $x$ in $R/\mfp$ is in $\overline{((I+\mfp)/\mfp)}^{\mfb^t}.$ Let $c_\mfp \in R \setminus \mfp$ such that $c_\mfp x^n \mfa^{\lceil tn \rceil} \subset c_\mfp x^n \mfa^{\lceil tn \rceil} + \mfp \subset I^n + \mfp$ for all $n.$ By prime avoidance, there exists a non-zero element $d_\mfp \in R$ which is not in $\mfp$ but is in every other minimal prime of $R.$ Then $d_\mfp c_\mfp x^n \mfa^{\lceil tn \rceil} \subset I^n + \sqrt{0},$  for all $n.$ Set $c' = \sum_{\mfp \in \operatorname{Min R}} d_\mfp c_\mfp.$ Then $c' \in R^{\circ}$ and for all $n$, we have \[ c'x^n \mfa^{\lceil tn \rceil} \subset I^n + \sqrt{0} = I^n \] implying that $x \in \overline{I}^{\mfa^t}.$  
\end{proof}

\subsubsection{Valuative Theory of $\mfa^t$-Integral Closure}

We now explore the connections between our new integral closure operation and valuation theory. As for ordinary integral closure, we will use the Rees valuations of an ideal $I$ to understand $\overline{I}^{\mfa^t}$. For the sake of concision, we will avoid a complete review of Rees valuations here, and direct the reader to \cite[Chapter~10]{hunekeSwanson} for a thorough treatment. However, we will record some of their properties below.

Each ideal $I$ in a noetherian ring $R$ has a finite set of valuations $\mathcal{RV}(I)$ with the following properties (\cite{ReesValuations}).\footnote{Following \cite[page 192]{hunekeSwanson}, we slightly abuse the term ``valuation'' by allowing the functions to be defined on non-domains and to take the value $\infty$ on certain minimal primes.}
\begin{enumerate}
	\item Each $\nu\in \mathcal{RV}(I)$ is a valuation, i.e.\ a function $\nu:R\setminus\{0\}\to \Z_{\geq 0}\cup\{\infty\}$ such that for each $0\neq x,y\in R$, $\nu(xy) = \nu(x) + \nu(y)$ and $\nu(x + y)\geq \min\{\nu(x),\nu(y)\}$.  For a nonzero ideal $J\subset R$ and Rees valuation $\nu\in \mathcal{RV}(I)$, we define $\nu(J) = \min\{\nu(x)\mid x\in J\}$.
	\item  For $x\in R$, $x \in \overline{I^n}$ if and only if $\nu(x) \geq n\nu(I)$ for all $\nu\in \mathcal{RV}(I)$ (see \cite[Remark~10.1.4]{hunekeSwanson}).
	\item $J\subset \overline{I^n}$ if and only if $\nu(J)\geq n \nu(I)$.
\end{enumerate}

\begin{theorem}\label{valuative criterion for integral closure:label}
	Let $R$ be a ring, $I,\mfa\subset R$ ideals of positive height, and $t\in \R_{>0}$.  Let $\mathcal{RV}(I)$ be the set of Rees valuations of $I$.
	An element $x\in R$ is in $\overline{I}^{\mfa^t}$ if and only if for all $\nu\in\mathcal{RV}(I)$, $\nu(x)\geq \nu(I)-t\nu(\mfa)$.
\end{theorem}

\begin{proof}
	 Suppose $x\in \overline{I}^{\mfa^t}$, and let $\nu\in\mathcal{RV}(I)$.  There exists $c\in R^\circ$ such that for $n\gg 0$, $cx^n\mfa^{\lceil tn\rceil}\subset I^n\subset \overline{I^n}$.  Applying $\nu$ to both sides and dividing by $n$, we have \[\frac{\nu(c)}{n} + \nu(x) + \frac{\lceil tn\rceil}{n}\nu(\mfa) \geq \nu(I).\]Then, taking the limit as $n\to\infty$ gives $\nu(x)+t\nu(\mfa) \geq \nu(I)$.	

			Now suppose that $\nu(x)\geq \nu(I)-t\nu(\mfa)$ for all $\nu\in \mathcal{RV}(I)$. Choose $s \leq t$ rational such that $\nu(x) \geq  \nu(I)-s\nu(\mfa)$ for all $\nu\in \mathcal{RV}(I)$.  Let $m\in \N$ such that $sm\in \N$.  We have that 
		\[\nu(x^m\mfa^{sm}) = m\nu(x)+sm\nu(\mfa) = m(\nu(x) + s\nu(\mfa))\geq  m\nu(I).\]
		Therefore $x^m\mfa^{sm} \subset \overline{I^m}$.  Thus there exists $c'\in R^\circ$ such that for all $n\in\N$, $c'\left(x^{m}\mfa^{sm}\right)^n\subset I^{mn}$.

		Let $c''\in I^{m-1}\cap R^\circ$ and $c = c'c''\in R^\circ$.  Now, for any $N\gg 0$, let $n\in \mathbb{N}$ such that $mn\leq N < m(n+1)$, and notice
		\[cx^N\mfa^{\lceil sN\rceil} \subset c'c''x^{mn}\mfa^{tmn}\subset c''I^{mn} \subset I^{mn+m-1} 
		\subset I^N.\]
		Therefore $x\in \overline{I}^{\mfa^s}\subset\overline{I}^{\mfa^t}$ by \Cref{prop: basic properties of integral closures}.
\end{proof}

The valuative criterion allows us to show that containment in the Hara-Yoshida integral closure of an ideal $I$ can be checked using infinitely many powers rather than all sufficiently large powers, which is also true for the ordinary integral closure (see \cite[Cor.~6.8.12]{hunekeSwanson}).

\begin{corollary}\label{cor: inf many n is enough for integral closure} 
	Let $R$ be a ring, $I, \mfa\subset R$ ideals of positive height, $t\in \R_{>0}$, and $x\in R$.  If there exists $c\in R^\circ$ such that $cx^n\mfa^{\lceil tn\rceil}\subset I^n$ for infinitely many $n\in\N$, then $x\in \overline{I}^{\mfa^t}$. 
\end{corollary}

\begin{proof}
	Let $\nu$ be a Rees valuation of $I$.  For infinitely many $n\in\N$, $cx^n\mfa^{\lceil tn\rceil}\subset I^n$, and therefore $\nu(c) + n\nu(x) + \lceil tn\rceil\nu(\mfa) \geq n\nu(I)$, implying $\nu(x) \geq \nu(I) - \dfrac{\lceil tn\rceil}{n}\nu(\mfa) - \dfrac{\nu(c)}{n}$.  Therefore
	\[\nu(x) \geq \liminf_{n\to\infty} \left(\nu(I) - \dfrac{\lceil tn\rceil}{n}\nu(\mfa) - \dfrac{\nu(c)}{n}\right) = \nu(I)-t\nu(\mfa).\]
	Since this holds for all Rees valuations $\nu$ of $I$, $x\in\overline{I}^{\mfa^t}$ by \Cref{valuative criterion for integral closure:label}.
\end{proof}

\Cref{valuative criterion for integral closure:label} is also related to a description of the $\mfa^t$-integral closure of a monomial ideal in an affine semigroup ring, see \Cref{thm: at integral closure convex geometry}.

\subsection{\texorpdfstring{The $\mfa^t$-tight closures}{The adjusted tight closures}}

Throughout this subsection, we assume $R$ is of prime characteristic $p>0$. Recall the following definitions from \cite{HY} and \cite{Vraciu}.

\begin{defn}\label{def: HY/V tight closure}
	Let $R$ be a ring, $I,\mfa \subset R$ be ideals, and $t \in \mathbb{R}_{>0}$. Then, the \textbf{Hara-Yoshida $\mfa^t$-tight closure of $I$}, denoted $\istarat,$ is defined by \[ \istarat = \left\lbrace x \in R \left| \text{ for some } c \in R^\circ \text{ and all } e\gg 0,\, cx^{p^e} \mfa^{\lceil tp^e \rceil} \subset I^{\fbp{p^e}}\right. \right\rbrace. \] 
	Similarly, the \textbf{Vraciu $\mfa^t$-tight closure of $I$}, denoted $\atistar$, is defined by
	\[ \atistar = \left\lbrace x \in R \left| \text{ for some } c \in R^\circ \text{ and all } e \gg 0,\, cx^{p^e} \mfa^{\lceil tp^e\rceil} \subset \mfa^{\lceil tp^e \rceil} I^{\fbp{p^e}} \right.\right\rbrace. \] Note that the Hara-Yoshida $\mfa^t$-tight closure is only a pre-closure (\cite[Observation~2.3(2)]{Vraciu}) whereas the Vraciu $\mfa^t$-tight closure is an ideal closure operation (\cite[Observation~2.3(4)]{Vraciu}). 
\end{defn}

We note that like the Hara-Yoshida integral closure, the Hara-Yoshida $\mfa^t$-tight closure is related to rational powers of $\mfa$; compare the following result with \Cref{prop: integral at-closure as colon ideal}.

\begin{prop}\label{tight at-closure as colon ideal}
	Let $R$ be a ring, $I,\mfa\subset R$ be ideals and $t \in \mathbb{Q}_{>0}$. Then $\istarat \subset (I^* : \mfa_t)$. 
\end{prop}

\begin{proof} 
Let $x\in \istarat$ with $c\in R^\circ$ such that $cx^{p^e}\mfa^{\lceil tp^e\rceil} \subset I^{[p^e]}$ for all $e\gg 0$, and let $y\in \mfa_t$.  By \cite[Lemma~2.8]{Tay}, there exists $d\in R^\circ$ such that for all $e\gg 0$, $dy^{p^e}\in \mfa^{\lceil tp^e\rceil}$. As $cd(xy)^{p^e} \in cx^{p^e} \mfa^{\lceil tp^e\rceil}\subset I^{[p^e]}$ for all $e$ large, we get $xy\in I^*$ and so $\istarat \subset (I^* : \mfa_t)$.		
\end{proof}

The following result shows that, like tight closure, the computation of $\mfa^t$-tight closure reduces to computation in domains via modding out by minimal primes.

\begin{prop} \label{prop: cl via min primes}
Let $R$ be a ring, $\mfa, I\subset R$ be ideals, and $t\in \R_{>0}$.  If $R$ is reduced or $\mfa$ is of positive height, then $x \in{}^{\mfa^t}I^*$ (respectively $x \in I^{*\mfa^t}$) if and only if, for each minimal prime $\mfp$ of $R$, we have $x+\mfp$ in $^{\mfb^t}((I+\mfp)/\mfp)^*$ (respectively $((I+\mfp)/\mfp)^{*\mfb^t}$), where $\mfb$ denotes the image of $\mfa$ in $R/\mfp.$
\end{prop}
\begin{proof}
	Let $x \in$ $^{\mfa^t}I^*.$ By definition, there exists $c \in R^\circ$ such that $cx^{p^e} \mfa^{\lceil tp^e\rceil} \subset \mfa^{\lceil tp^e\rceil}I^{[p^e]}$ for all large $e.$ Thus, for any minimal prime $\mfp$ of $R,$ the image of $c$ in $R/\mfp$ is in $(R/\mfp)^\circ$ and $cx^{p^e} \mfa^{\lceil tp^e\rceil} + \mfp \subset \mfa^{\lceil tp^e\rceil} I^{[p^e]} + \mfp$ for all large $e$ implying that the image of $x$ in $R/\mfp$ is in $^{\mfb^t}((I+\mfp)/\mfp)^*.$
	
	Conversely, assume that for each minimal prime $\mfp$ of $R$, the image of $x$ in $R/\mfp$ is in $^{\mfb^t}((I+\mfp)/\mfp)^*.$ Let $c_\mfp \in R \setminus \mfp$ such that $c_\mfp x^{p^e} \mfa^{\lceil tp^e\rceil} \subset c_\mfp x^{p^e} \mfa^{\lceil tp^e\rceil} + \mfp \subset \mfa^{\lceil tp^e\rceil} I^{[p^e]} + \mfp$ for all $e$ large. By prime avoidance, there exists a non-zero element $d_\mfp \in R$ which is not in $\mfp$ but is in every other minimal prime of $R.$ Then $d_\mfp c_\mfp x^{p^e} \mfa^{\lceil tp^e\rceil} \subset \mfa^{\lceil tp^e\rceil}I^{[p^e]} + \sqrt{0}$  for all $e$ large. Set $c' = \sum_{\mfp} d_\mfp c_\mfp.$ Then $c' \in R^{\circ}$ and for all $e$ large,
	\[ c'x^{p^e} \mfa^{\lceil tp^e\rceil} \subset \mfa^{\lceil tp^e\rceil} I^{[p^e]} + \sqrt{0}. \]  
	If $R$ is reduced, then we're done since $\sqrt{0}=(0)$.  So, assume that $\mfa$ is of positive height.
	Since $R$ is a noetherian ring, there exists $e'$ such that $(\sqrt{0})^{[p^{e'}]} = (0).$ Thus, for all $e$ large, we have $(c')^{p^{e'}} x^{p^{e+e'}} (\mfa^{\lceil tp^e\rceil})^{[p^{e'}]} \subset (\mfa^{\lceil tp^e\rceil})^{[p^{e'}]} I^{[p^{e+e'}]}$. From \cite[Lemma~2.1]{Tay}, we know that $\mfa^{\lceil tp^{e+e'} + \mu(\mfa) p^{e'}\rceil} \subset (\mfa^{\lceil tp^e\rceil})^{[p^{e'}]}$. So, let $b \in \mfa^{\mu(\mfa) p^{e'}} \cap R^\circ$. For all $e$ large,
	\[ b(c')^{p^{e'}} x^{p^{e+e'}} \mfa^{\lceil tp^{e+e'}\rceil} 	\subset 
	(c')^{p^{e'}} x^{p^{e+e'}} (\mfa^{\lceil tp^e\rceil})^{[p^{e'}]} \subset  (\mfa^{\lceil tp^e\rceil})^{[p^{e'}]} I^{[p^{e+e'}]} \subset \mfa^{\lceil tp^{e+e'}\rceil} I^{[p^{e+e'}]}. \]
	Since $b(c')^{p^{e'}}\in R^\circ$, this shows $x \in$ $^{\mfa^t} I^*$.
	The proof for the ideal $I^{*\mfa^t}$ follows similarly.
\end{proof}

\subsection{\texorpdfstring{The $\mfa^t$-Frobenius closures}{The adjusted Frobenius closures}}

Throughout this subsection, we assume $R$ is of prime characteristic $p>0$.

\begin{defn}\label{defn: adjusted frobenius closures}
	Let $R$ be a ring, $I,\mfa\subset R$ be ideals, and $t \in \mathbb{R}_{>0}$. Then, the \textbf{Hara-Yoshida $\mfa^t$-Frobenius closure of $I$}, denoted $\ifat,$ is defined by
	\[ \ifat = \left\lbrace x \in R \left|\text{ for all } e \gg 0,\, x^{p^e}\mfa^{\lceil tp^e\rceil} \subset I^{\fbp{p^e}} \right.\right\rbrace. \] 
	Similarly, the \textbf{Vraciu $\mfa^t$-Frobenius closure of $I$}, denoted $\atif$, is defined by
	\[ \atif = \left\lbrace x \in R \left|\text{ for all } e \gg 0,\, x^{p^e}\mfa^{\lceil tp^e\rceil} \subset \mfa^{\lceil tp^e\rceil}I^{\fbp{p^e}} \right.\right\rbrace. \] As before, we note that the Hara-Yoshida $\mfa^t$-Frobenius closure is a pre-closure (\Cref{prop: basics of adjusted Frobenius closure}) and not an ideal closure operation (\Cref{F-principal}(a)), whereas the Vraciu $\mfa^t$-Frobenius closure is an ideal closure operation (\Cref{prop: basics of adjusted Frobenius closure}).
\end{defn}

We now record several basic properties of these operations. 

\begin{proposition}\label{prop: basics of adjusted Frobenius closure}
Let $R$ be a ring, $I,\mfa\subset R$ be ideals, and $t \in \mathbb{R}_{>0}$. Then, we have the following.
\begin{enumerate}[{\rm(a)}]
\item $\atif$ and $\ifat$ are ideals of $R.$
\item $I^F \subset \atif\subset\ifat.$
\item If $I\subset J$, $\mathfrak{b}\subset\mfa$, and $t \le s$, then $\ifat \subset J^{F\mathfrak{b}^s}$ and $\atif \subset {}^{\mfa^s}J^F.$ 
\item $^{\mfa^t}(\atif)^F = \atif.$
\end{enumerate}
\end{proposition}

\begin{proof}
We will prove each in turn.

\begin{enumerate}[(a)]
\item Notice that $x \in \ifat$ if and only if $x^{p^e} \in I^{\fbp{p^e}}:\mfa^{\lceil tp^e \rceil}$ for all $e \gg 0$, from whence it is obvious that $\ifat$ is an ideal of $R$. Similarly, $x\in \atif$ if and only if $x^{p^e} \in I^{\fbp{p^e}}\mfa^{\lceil tp^e \rceil}:\mfa^{\lceil tp^e \rceil}$ for all $e\gg 0$, so $\atif$ is also an ideal of $R$.

\item Since $I^{\fbp{p^e}} \subset I^{\fbp{p^e}}\atpe:\atpe\subset I^{\fbp{p^e}}:\atpe$, for all $e\in \mbn$, we have $I \subset I^F \subset \atif \subset \ifat$.
 
\item Since $I\subset J$, $I^{\fbp{p^e}}\subset J^{\fbp{p^e}}$ for all $e$, we have $I^{\fbp{p^e}}:\atpe \subset J^{\fbp{p^e}}:\mathfrak{b}^{\lceil sp^e\rceil}$ for all $e$. Thus, $\ifat \subset J^{F\mathfrak{b}^s}$. Similarly, $I^{\fbp{p^e}}\atpe :\atpe \subset J^{\fbp{p^e}} \atpe:\atpe \subset J^{[p^e]} \mfa^{\lceil sp^e \rceil} : \mfa^{\lceil sp^e \rceil}$ for all $e$, implying that $\atif \subset {}^{\mfa^s}J^F$. 

\item From part b), it follows that $\atif \subset {}^{\mfa^t}(\atif)^F.$ Now, let $x \in {}^{\mfa^t}(\atif)^F.$ Then for all $e \gg 0,$ $x^{p^e} \atpe \subset \atpe (\atif)^{[p^e]} \subset \atpe I^{[p^e]},$ where the last containment holds by checking the generators of $\atif$ and using that $\atif$ is a finitely generated ideal.
\qedhere
\end{enumerate}\end{proof}

Note the differences in the montonicity properties of the Hara-Yoshida and Vraciu $\mfa^t$-Frobenius closures outlined in (c) above. In fact, as for the Vraciu $\mfa^t$-tight closure (see \cite[Observation~5.1]{Vraciu}), the following example illustrates that it is not always true that $^\mfa I^F \subset$ $^\mathfrak{b} I^F$ if $\mathfrak{b} \subset \mfa.$

\begin{xmp}
Let $R = k[x,y],$ where $k$ is a field of prime characteristic $p>0.$ Let $\mfa = (x,y)^2,$ $\mathfrak{b} = (f),$ where $f \in \mfa$ and $I = (x^2,y^2).$ From \Cref{F-principal}(b), we know that $^\mathfrak{b} I^F = I^F = I$ as $R$ is a regular ring. But $xy \in$ $^\mfa I^F$ as $(xy)^{p^e} \mfa^{p^e} \subset I^{[p^e]} \mfa^{p^e}$ for all $e$.
\end{xmp}

Unlike $\mfa^t$-integral closure and $\mfa^t$-tight closure, $\mfa^t$-Frobenius closure cannot be computed using the corresponding closure ideal after going modulo the minimal primes. We use \Cref{example: frob closure can't be checked mod minimal primes} to check that \Cref{prop: cl via min primes} may not hold true for $\atif$ and $\ifat$.

\begin{xmp}
	Let $R=\mathbb{F}_3[X,Y,Z]/(Y^6-X^3,X^2-Z^2)$ and let $x,y,$ and $z$ denote the images of $X,Y,$ and $Z$ respectively in $R.$ Let $I=(x)$ and $\mfa = (x,y^2)$. We claim that $z \notin$ $^{\mfa} I^F.$ Indeed, if $z \in$ $^{\mfa}I^F,$ then there exists $e$ such that $z^{3^e} (x,y^2)^{3^e} \subset \mfa^{3^e}I^{[3^e]}$. In particular, $z^{3^e}x^{3^e} \in \mfa^{3^e}I^{[3^e]} = (x,y^2)^{3^e}(x^{3^e})$.  Therefore, $zx^{3^e-1}x^{3^e} = z(x^2)^{(3^e-1)/2}x^{3^e} = z(z^2)^{(3^e-1)/2}x^{3^e} = z^{3^e}x^{3^e} \in (x,y^2)^{3^e}(x^{3^e})$.
	Since $x^{3^e}$ is a nonzerodivisor, we get $z \in (x,y^2)^{3^e} : x^{3^e-1} = (x,y^2)$ which is not true. Thus, the claim holds. Note that $\operatorname{Min} R=\{(x-z,y^2-z),(x+z,y^2+z) \}$ and so modulo either minimal prime, the image of $z$ is the same as the image of $x.$ Thus, $z+\mathfrak{p} \in ((x)+\mathfrak{p})/\mathfrak{p} \subset$ $^{\mfa} ((x)+\mathfrak{p})^F/\mathfrak{p}$, but $z\not \in$ $^{\mfa}(x)^F$.
	
	Now let $I = (x+y^2)$ and $\mfa = (x,y^2).$ Then one can check that modulo any minimal prime $\mathfrak{p},$ the element $(1+\mfp)/\mfp$ is in $((I+\mathfrak{p})/\mathfrak{p})^{F\mfa}$ but $1\notin I^{F\mfa}.$ In particular, it is enough to show that $(\mfa^{3^e}+\mathfrak{p})/\mathfrak{p} \subset (I^{[3^e]}+\mathfrak{p})/\mathfrak{p}$ for every minimal prime $\mathfrak{p}$ and for all $e$ large but $\mfa^{3^e} \not\subset I^{[3^e]}$ for any $e.$ Indeed, modulo any minimal prime $\mfp,$ 
	\[ \frac{I^{[3^e]} + \mfp}{\mfp} = \frac{(x^{3^e}+(y^2)^{3^e})+\mfp}{\mfp} = \frac{(x^{3^e} + x^{3^e}) + \mfp}{\mfp} = \frac{(x^{3^e}) + \mfp}{\mfp} = \frac{\mfa^{3^e}+\mfp}{\mfp}, \]
	for all $e.$ Whereas in the ring $R$, if $\mfa^{3^e} \subset I^{[3^e]}$ for some $e,$ then $x^{3^e-1}y^2 \in (x^{3^e} + (y^2)^{3^e}) = (x^{3^e}).$ But as $x$ is a nonzerodivisor, this means, $y^2 \in (x)$ which is not true. 
\end{xmp}

The following observations demonstrate several cases where the adjusted Frobenius closures are simple to calculate or trivially related to another operation.

\begin{proposition}  \label{F-principal}
	Let $R$ be a ring, $I,\mfa\subset R$ be ideals, and $t \in \mathbb{R}_{>0}$.
	\begin{enumerate}[{\rm(a)}]
		\item We have $I^{F\mfa} \subset (I^F : \mfa),$ and the inclusion is an equality if $\mfa$ is a principal ideal. In particular, if $\mfa = (f)$, we have $(I^{F\mfa})^{F\mfa} \neq I^{F\mfa}$ when $(R,\mfm)$ is local, $I$ is $\mfm$-primary, and $f \in \mfm \setminus I^F.$
		\item If $\mfa = (f)$ is a principal ideal, and $f$ is a non-zerodivisor on $R$, then $\,^{\mfa^t}I^F = I^F.$
		\item If $t > \mu(\mfa)$ then
		\begin{enumerate}[{\rm(i)}]
			\item $\ifat = (I^{F\mfa^{t-1}} : \mfa),$ and
			\item $\atif = (^{\mfa^{t-1}}(\mfa I)^F : \mfa).$ 
		\end{enumerate}
		\item If $\mfa \subset \sqrt{0}$, then for all ideals $I\subset R$ and $t \in \mathbb{R}_{>0}$, we have $\atif = R$. 
	\end{enumerate}
\end{proposition}

\begin{proof}
	\begin{enumerate}[(a)]
		\item The first statement holds using the definition. Now, if $\mfa = (f)$ is a principal ideal, then 
		\begin{align*}
			x \in I^{F\mfa} 
			&\Leftrightarrow \text{ for all } e \gg 0, x^{p^e} (f)^{p^e} \in I^{[p^e]}\\
			&\Leftrightarrow \text{ for all } e \gg 0, (xf)^{p^e} \in I^{[p^e]}, 
			\\
			&\Leftrightarrow xf\in I^F\\
			&\Leftrightarrow x \in I^F : (f).
		\end{align*} Thus, $(I^{F(f)})^{F(f)} = (I^F : (f)^2)$ which is strictly larger than $I^{F(f)} = I^F : (f)$ under the given assumptions by Nakayama's lemma, as $xf \in I^F$ implies $xf^2 \in I^F\cdot(f) \subsetneq I^F$.
		
		\item Follows immediately from the definition of a non-zerodivisor.
		
		\item Note that as $t > \mu(\mfa),$ $\mfa^{\lceil tp^e \rceil} = \mfa^{\lceil (t-1)p^e \rceil} \mfa^{[p^e]}.$ Thus, 
		\[ x \in \ifat \Leftrightarrow \forall e\gg 0, x^{p^e} \in I^{[p^e]} : \mfa^{\lceil tp^e \rceil} = I^{[p^e]} : \mfa^{\lceil (t-1)p^e \rceil} \mfa^{[p^e]} \Leftrightarrow x \in (I^{F\mfa^{t-1}} : \mfa) \]
		and
		\[ x \in \atif \Leftrightarrow \forall e\gg 0, x^{p^e} \in I^{[p^e]} \mfa^{\lceil tp^e \rceil} : \mfa^{\lceil tp^e \rceil} = (\mfa I)^{[p^e]} \mfa^{\lceil (t-1)p^e \rceil} : \mfa^{\lceil (t-1)p^e \rceil} \mfa^{[p^e]} \Leftrightarrow x \in (^{\mfa^{t-1}}(\mfa I)^F : \mfa).\]
		\item If $\mfa^n=0$ for some $n$, then for all $e \gg 0$ and $x \in R$, we have $x^{p^e} \cdot 0 \subset I^{[p^e]}\cdot 0 = (0).$ \qedhere
	\end{enumerate}
\end{proof}

\section{Comparing the new closures}\label{sec: comparing the new closures}

In this subsection, we compare the different adjusted closure operations with one another. All statements which involve the adjusted integral closure only are characteristic independent, but whenever the adjusted tight or Frobenius closure operations are mentioned, we tacitly assume the ambient ring is of prime characteristic $p>0$. Furthermore, in any statement where we invoke the $\mfa^t$-integral closure, we assume $I$ and $\mfa$ have positive height.

\begin{remark} \label{rmk: closure comparison diagram}

By taking $c=1$ in the definition of the (adjusted) $\mfa^t$-tight closures, we can easily see $I^{F\mfa^t} \subset I^{*\mfa^t}$ and $\,^{\mfa^t} I^F \subset \,^{\mfa^t}I^*$. Similarly, the adjusted $\mfa^t$-tight closures are inside the corresponding adjusted $\mfa^t$-integral closure, since if $x\in I^{*\mfa^t}$, then for some $c\in R^\circ$ and all $e\gg 0$, $cx^{p^e}\mfa^{\lceil tp^e\rceil}\subset I^{[p^e]}\subseteq I^{p^e}$ and so $x\in \iatb$ by \Cref{cor: inf many n is enough for integral closure}, and similarly for $ \,^{\mfa^t}I^*\subset \atib = \overline{I}$.

In general, we have the following inclusions of ideals, fixing $I$, $\mfa$, and $t\in \R_{>0}$, and assuming $I$ and $\mfa$ are of positive height. 

\begin{center}\begin{tikzpicture}
\node[scale=1] at (0,0) {
	\begin{tikzcd}
		I \arrow[hook]{r} & I^F \arrow[hook]{r}\arrow[hook]{d}& \atif\arrow[hook]{d}\arrow[hook]{r} & \ifat \arrow[hook]{d} \\
		\, & I^* \arrow[hook]{r}\arrow[hook]{d} & \atistar \arrow[hook]{r}\arrow[hook]{d} & \istarat \arrow[hook]{d} \\
		\, & \overline{I} \arrow[equal]{r} & \atib \arrow[hook]{r} & \iatb
	\end{tikzcd}
};
\end{tikzpicture}
\end{center}
\end{remark}

Later in Section 6 (\Cref{MegaExample}), we see that the inclusions above are generally strict. We note that each of the adjusted closure operations are insensitive to replacing the base ideal $I$ with its corresponding unadjusted ideal closure operation.

\begin{proposition}\label{prop: may replace ideals with their ordinary closures}
Let $I,\mfa\subset R$ be ideals of positive height and $t\in \R_{>0}$. Then, we have the following equalities.
\begin{enumerate}[{\rm(a)}]
\item $(I^F)^{F \mfa^t} = I^{F\mfa^t }$ and $\,^{\mfa^t}(I^F)^F = \, ^{\mfa^t}I^F$, and
\item $(I^*)^{*\mfa^t} = I^{*\mfa^t}$ and $\,^{\mfa^t}(I^*)^* = \, ^{\mfa^t} I^*$, and
\item $\overline{J}^{\mfa^t} = \overline{I}^{\mfa^t}$, where $J=\overline{I}$. In particular, if $L \subset I$ is a reduction, then $\overline{L}^{\mfa^t} = \overline{I}^{\mfa^t}.$
\end{enumerate}
\end{proposition}
\begin{proof}
Since $^{\mfa^t}(\atif)^F = \atif$ and $^{\mfa^t}(\atistar)^* = \atistar,$ it is easy to check the equalities using the inclusions $I \subset I^F \subset \atif$ and $I \subset I^* \subset \atistar,$ so we turn to the Hara-Yoshida versions.
\begin{enumerate}[(a)]
\item We will show $(I^F)^{F\mfa^t}=I^{F\mfa^t}$. As $I\subset I^F$, we have $I^{F\mfa^t}\subset (I^F)^{F\mfa^t}$. To see the reverse containment, let $x \in (I^F)^{F\mfa^t}$. Then $x^{p^e} \mfa^{\lceil tp^e\rceil } \subset (I^F)^{\fbp{p^e}} = I^{\fbp{p^e}}$ for all $e$ large. Thus, $x \in I^{F\mfa^t}$.

\item  We will show $(I^*)^{*\mfa^t}=I^{*\mfa^t}$. As $I\subset I^*$, we have $I^{*\mfa^t}\subset (I^*)^{*\mfa^t}$. To see the reverse containment, let $x \in (I^*)^{*\mfa^t}$. Then there exists $c \in R^\circ$ such that $cx^{p^e} \mfa^{\lceil tp^e\rceil } \subset (I^*)^{\fbp{p^e}}$ for all $e$ large. 

For the moment, write $I^* = (x_1,\ldots,x_n)$. For each $i$ there is a $c_i \in R^\circ$ such that $c_ix_i^{p^e} \in I^{\fbp{p^e}}$ for all $e \gg 0$, and so taking $d = \prod_i c_i$, for all $e \gg 0$ and $y \in I^*$, $dy^{p^e} \in I^{\fbp{p^e}}$. Hence $d(I^*)^{\fbp{p^e}} \subset I^{\fbp{p^e}}$ for all $e \gg 0$. Thus for $e$ large, $cdx^{p^e} \mfa^{\lceil tp^e\rceil } \subset I^{\fbp{p^e}}$ implying that $x \in I^{*\mfa^t}$.

\item This proof is nearly identical to the tight closure case,  but we include it for completeness. First, as $I\subset J$, we have $\overline{I}^{\mfa^t}\subset \overline{J}^{\mfa^t}$ for any $t$. To see the reverse containment, let $x \in \overline{J}^{\mfa^t}$. Then, there is a $c \in R^\circ$ such that for all $n \gg 0$, $cx^n\mfa^{\lceil nt\rceil}\subset J^n$. 

For the moment, write $J=(x_1,\ldots,x_m)$. As $J = \overline{I},$ by \cite[Corollary~6.8.12]{hunekeSwanson} and Proposition \ref{prop: integral closure for all n works}, we have for each $i$ a $c_i \in R^\circ$ such that $c_i x_i^n \in I^n$ for all $n \geq 0$. So, by taking $d=\prod_i c_i$, we have for any $y \in J^n$, $dy \in I^n$ for all $n \geq 0$, namely $d J^n \subset I^n$ for all $n \geq 0$. Hence, $dcx^n \mfa^{\lceil nt \rceil} \subset dJ^n \subset I^n$ for all $n \gg 0$, so $x \in \overline{I}^{\mfa^t}$.
 \qedhere
\end{enumerate}
\end{proof}

In \cite[Proposition~1.3(4)]{HY}, the authors prove that the Hara-Yoshida $\mfa$-tight closure of an ideal $I$ is unaffected by replacing $\mfa$ with any ideal with the same integral closure. We prove that the same holds for Vraciu $\mfa^t$-tight closure and $\mfa^t$-integral closure as well.

\begin{proposition}
Let $R$ be a ring, $I,\mfa,\mfb \subset R$ be ideals, and suppose $\mfb \subset \mfa$ is a reduction. Then $^{\mfa^t}I^* =$ $^{\mfb^t}I^*$  and $\overline{I}^{\mfa^t} = \overline{I}^{\mfb^t}.$
\end{proposition}

\begin{proof}
	As $\mfb \subset \mfa$ is a reduction, there exists $k$ such that $\mfb^m \mfa^k = \mfa^{m+k},$ for all $m \geq 1.$ As $\mfa \cap R^\circ \neq \emptyset,$ pick $d \in \mfa^k \cap R^\circ.$
	
We first consider the adjusted tight closures. Let $x \in$ $^{\mfb^t}I^*.$ Then there exists $c \in R^\circ$ such that $cx^{p^e}\mfb^{\lceil tp^e\rceil} \subset I^{[p^e]} \mfb^{\lceil tp^e\rceil}$ for all $e$ large. Thus, for all $e$ large \[ cd x^{p^e} \mfa^{\lceil tp^e\rceil} \subset cx^{p^e} \mfa^{\lceil tp^e\rceil+k} \subset cx^{p^e} \mfb^{\lceil tp^e\rceil} \subset I^{[p^e]} \mfb^{\lceil tp^e\rceil} \subset I^{[p^e]} \mfa^{\lceil tp^e\rceil}, \] implying that $x \in$ $^{\mfa^t}I^*.$ This proves that $^{\mfb^t}I^* \subset$ $^{\mfa^t}I^*.$
		
		Now, let $x \in$ $^{\mfa^t}I^*.$ Then there exists $c \in R^\circ$ such that $cx^{p^e} \mfa^{\lceil tp^e\rceil} \subset I^{[p^e]} \mfa^{\lceil tp^e\rceil}$ for all $e$ large. For $e$ large, we have\[ 
		cd x^{p^e} \mfb^{\lceil tp^e\rceil} \subset cd x^{p^e} \mfa^{\lceil tp^e\rceil} \subset dI^{[p^e]}\mfa^{\lceil tp^e\rceil} \subset I^{[p^e]} \mfa^{\lceil tp^e\rceil+k} \subset I^{[p^e]} \mfb^{\lceil tp^e\rceil} 
		\] implying that $x \in$ $^{\mfb^t}I^*.$
		
We now turn to the $\mfa^t$-integral closure. Since $\mfb \subset \mfa,$ from \Cref{prop: basic properties of integral closures}(3), we know that $\overline{I}^{\mfa^t} \subset \overline{I}^{\mfb^t}.$ Let $x \in \overline{I}^{\mfb^t}.$ By definition, there exists $c \in R^\circ$ such that $cx^n \mfb^{\lceil tn\rceil} \subset I^n$ for all $n$ large. Thus for $n$ large, we have \[
		cdx^n \mfa^{\lceil tn\rceil} \subset cx^n\mfa^{\lceil tn\rceil+k} \subset cx^n\mfb^{\lceil tn\rceil} \subset I^n \] implying that $x \in \overline{I}^{\mfa^t}.$    
\end{proof}

As for the ordinary tight and integral closures, the adjusted versions agree for principal ideals.

\begin{prop}\label{prop: tight closure is integral closure for principal ideals}
Let $R$ be a ring, $\mfa\subset R$ an ideal of positive height, and $t \in \mathbb{R}_{> 0}$. Then, for any $x \in R^\circ$, $(x)^{*\mfa^t}=\overline{(x)}^{\mfa^t}$ and ${}^{\mfa^t}(x)^* = \overline{(x)}$.
\end{prop}

\begin{proof} Since the $\mfa^t$-tight closure is always contained in the $\mfa^t$-integral closure, it suffices to prove the reverse containment.  Let $y\in\overline{(x)}^{\mfa^t}$, so that there exists $c\in R^\circ$ such that for all $n\gg0,$ $cy^n\mfa^{\lceil tn\rceil} \subset (x)^n$.  Therefore, for $e\gg 0$, we may take $n=p^e$ and so
\[cy^{p^e}\mfa^{\lceil tp^e\rceil} \subset (x)^{p^e} = (x)^{[p^e]}\]
Therefore $y\in (x)^{*\mfa^t}$, and so $\overline{(x)}^{\mfa^t} \subset (x)^{*\mfa^t}$. The Vraciu version holds similarly, after recalling that the Vraciu $\mfa^t$-integral closure is the usual integral closure in this case.
\end{proof}

\subsection{Brian\c{c}on-Skoda theorems}

First shown by Brian\c{c}on-Skoda in \cite{BS74} for the ring of convergent power series in $n$ variables over the complex numbers, and then Lipman-Tessier in \cite{LT81} for pseduo-rational local rings, the Brian\c{c}on-Skoda theorem relates $\overline{I^{n+k}}$ and $I^n$. Later, Hochster-Huneke developed a tight closure version (\cite[Theorem~5.4]{HH}) which relates $\overline{I^{n+k}}$ and $(I^n)^*$ using a much simpler proof.  A similar version involving the Frobenius closure is proved in \cite[Discussion 7.6(c)]{HHSplit} or \cite[Theorem 2.2]{KZ}. Efforts have been made to uniformize the constant $k$, most notably Huneke in \cite{HunekeUniformity} who showed that, under some mild hypotheses, one can take $k$ to be uniform over all ideals $I$.  We note that the adjusted closure operations satisfy similar theorems. 

\begin{theorem}[A Brian\c{c}on-Skoda theorem for the adjusted  closure operation] \label{BS thm}
	Let $R$ be a ring, and $\mfa, I\subset R$ be ideals of positive height. Let $I$ be generated by $d$ elements and let $t \in \mathbb{R}_{> 0}$. 
	\begin{enumerate}[{\rm(a)}]
		\item Then $\overline{I^{d+m}}^{\mfa^t} \subset$ $(I^{m+1})^{*\mfa^t}$ for all $m \ge 0,$ and
				
		\item $\overline{I^{d+m}}^{\mfa^t} \subset (I^m)^{F\mfa^t},$ for all $m \geq 1.$
	\end{enumerate}
Further, suppose $(R,\mfm)$ is local of dimension $\dim R = n$ and $R/\mathfrak{m}$ is infinite.
\begin{enumerate}[{\rm(a)}]
\addtocounter{enumi}{2}
\item Then, $\overline{I^{n+m}}^{\mfa^t} \subset (I^{m+1})^{*\mfa^t}$ for all $m \ge 0$, and
\item $\overline{I^{n+m}}^{\mfa^t} \subset (I^m)^{F\mfa^t}$ for all $m \ge 1$. 
\end{enumerate}
\end{theorem}

\begin{proof}
	\begin{enumerate}[(a)]
		\item 
		 Let $x \in \overline{I^{d+m}}^{\mfa^t}.$ By definition, there exists $c \in R^\circ$ such that $cx^n\mfa^{\lceil tn \rceil} \subset I^{(d+m)n},$ for all $n$ large. In particular, for all $n=p^e$ with $e \gg 0$, from \cite[Lemma~2.1]{Tay}, we know that $I^{(d+m)p^e} \subset (I^{[p^e]})^{m+1}.$ Thus, we have $cx^{p^e} \mfa^{\lceil tp^e \rceil} \subset (I^{m+1})^{[p^e]}$ for all $e$ large so that $x \in (I^{m+1})^{*\mfa^t}$.
				
		\item 
We first assume that $R$ is reduced and thus has no embedded primes. In other words, every $c \in R^\circ$ is a nonzerodivisor. Let $x \in \overline{I^{d+m}}^{\mfa^t}.$ By definition, there exists $c \in R^\circ$ such that $cx^n\mfa^{\lceil tn \rceil} \subset I^{(d+m)n},$ for all $n$ large implying that $cx^n\mfa^{\lceil tn \rceil} \subset I^{(d+m)n} \cap (c)R$. Using the Artin-Rees lemma, there exists $n'$ such that for all $n$ large,
		\[ cx^n\mfa^{\lceil tn \rceil} \subset I^{(d+m)n} \cap (c)R \subset cI^{(d+m)n-n'}. \]	
		As $c$ is a nonzerodivisor, we get $x^n \mfa^{\lceil tn \rceil} \subset I^{(d+m)n-n'}$ for all $n$ large. As before, taking $n = p^e$, from \cite[Lemma~2.1]{Tay}, we know that $I^{(d+m)p^e-n'} \subset (I^{[p^e]})^{\lceil m-(n'/p^e) \rceil} = (I^{[p^e]})^m,$ for all $e$ large. Thus, $x^{p^e} \mfa^{\lceil tp^e \rceil} \subset (I^m)^{[p^e]}$ for all $e$ large implying that $x \in (I^m)^{F\mfa^t}.$
		
		Generally, if $R$ is not reduced, then from the previous case, we know that for all $m \geq 1,$
		\[ \overline{I^{d+m}}^{\mfa^t} = \overline{I^{d+m}}^{\mfa^t} + \sqrt{(0)} \subset (I^m)^{F\mfa^t} + \sqrt{(0)} = (I^m)^{F\mfa^t}, \]
		where the last equality follows from the observation that $\sqrt{(0)} = (0)^F \subset (I^m)^F.$  
		
		\item Let $\ell$ be the analytic spread of $I$, and we know $\ell \le n$. Then, under our hypotheses we have that $I$ has a minimal reduction $J$ generated by $\ell$ elements by \cite[Prop.~8.3.7]{hunekeSwanson}, say $J=(x_1,\ldots,x_\ell).$ Then $\overline{J^m} = \overline{I^m}$ for any $m\ge 0$. Hence, $\overline{I^{n+m}}^{\mfa^t} = \overline{J^{n+m}}^{\mfa^t}$ by \Cref{prop: may replace ideals with their ordinary closures}(c). By part (a) of this theorem, we have 
		\[\overline{J^{m+n}}^{\mfa^t} \subset (J^{m+(n-\ell)+1})^{*\mfa^t} \subset (J^{m+1})^{*\mfa^t} \subset (I^{m+1})^{*\mfa^t}\] 
		since $J^{m+1} \subset I^{m+1}$, which shows the result. The proof of (d) is nearly identical. \qedhere
	\end{enumerate}
\end{proof}

\begin{remark}
First note that when $\mfa = R,$ the above result recovers the usual Brian\c{c}on-Skoda theorem for tight closure. Further, a similar Brian\c{c}on-Skoda theorem holds for the Vraciu $\mfa^t$-integral and $\mfa^t$-tight closures, but since $^{\mfa^t}\overline{I}=\overline{I}$, the result is true simply by the usual tight closure Brian\c{c}on-Skoda; $^{\mfa^t}\overline{I^{m+n}} = \overline{I^{m+n}} \subset (I^{m+1})^* \subset$ $^{\mfa^t}(I^{m+1})^*.$
\end{remark}

\subsection{Comparisons with sharp \texorpdfstring{$F$}{F}-closure}\label{subsubsection: sharp F-closure}

In \cite{Sch}, Schwede defined the \emph{$\mfa^t$-sharp Frobenius closure} of $I$, denoted $I^{F^\sharp \mfa^t}$, to be the set of all $x\in R$ such that for all $e\gg 0$, $x^{p^e}\mfa^{\lceil t(p^e-1)\rceil}\subset I^{[p^e]}$. It is clear that the $\mfa^t$-sharp Frobenius closure is only a preclosure; $\bullet^{F^\sharp \mfa^t}$ is extensive and order-preserving by a similar proof to \Cref{prop: basics of adjusted Frobenius closure}, but taking the example $R=\mathbb{F}_p[x]$ and $I=\mfa=(x^2)$, one can easily show $(I^{F^\sharp \mfa})^{F^\sharp \mfa}$ is strictly larger than $I^{F^\sharp \mfa}$.

Schwede's definition differs from our Hara-Yoshida $\mfa^t$-Frobenius closure due to the smaller exponent on $\mfa$, as the following example shows.  Note that by definition we have that $I^{F^\sharp \mfa^t}\subset \ifat$.

\begin{ex}\label{example - sharp vs HY F-closure:label}
	Let $R=\mathbb{F}_p[x,y]$, $I=(x,y)$, $\mfa=(x,y)$, and $t=2$.  We claim that $1\in \ifat \setminus I^{F^\sharp \mfa^t}$.  First, note that for any $e\gg 0$, $1\mfa^{\lceil tp^e\rceil} = (x,y)^{2p^e}\subset (x^{p^e},y^{p^e}) = I^{[p^e]}$, so $1\in \ifat$.  On the other hand, $1\mfa^{\lceil t(p^e-1)\rceil} = (x,y)^{2(p^e-1)}$ contains $x^{p^e-1}y^{p^e-1}\notin I^{[p^e]}$.  Therefore $1\notin I^{F^\sharp \mfa^t}$.
\end{ex}

\Cref{example - sharp vs HY F-closure:label} shows that a small adjustment in the exponent of $\mfa$ in the definition of the adjusted Frobenius closure can give a different operation on ideals.  The following example shows that adjusting the exponent of $\mfa$ in the definition of the Vraciu Frobenius closure also provides a different operation on ideals.

\begin{ex}
	Let $R=\mathbb{F}_p[x,y]$, $I=(x^2,y^2)$, $\mfa=(x,y)$, and $t=2$. First, note that $\atif$ contains the set $\left\{x\in R\mid \forall e\gg 0, x^{p^e}\mfa^{\lceil t(p^e-1)\rceil}\subset I^{[p^e]}\mfa^{\lceil t(p^e-1)\rceil}\right\}$. 
	
	 Now, for any $e\geq 0$, the ideal $(xy)^{p^e}\mfa^{2p^e}$ is generated by monomials of the form $x^{p^e + i}y^{3p^e-i}$ for $0\leq i\leq 2p^e$.  If $i<p^e$, then $x^{p^e + i}y^{3p^e-i} =y^{2p^e}x^{p^e + i}y^{p^e-i}\subset I^{[p^e]}\mfa^{2p^e}$, and if $i\geq p^e$ then $x^{p^e + i}y^{3p^e-i} =x^{2p^e}x^{i - p^e}y^{3p^e-i}\subset I^{[p^e]}\mfa^{2p^e}$.  Thus $(xy)^{p^e}\mfa^{2p^e}\subset I^{[p^e]}\mfa^{2p^e}$, and so $xy\in \atif$.  On the other hand, it is not the case that $(xy)^{p^e}\mfa^{2(p^e-1)} \subset  I^{[p^e]}\mfa^{2(p^e-1)}$.  The former set contains the monomial $(xy)^{p^e}x^{p^e-1}y^{p^e-1} = x^{2p^e-1}y^{2p^e-1}$, which is not even contained in $I^{[p^e]}$.  Therefore, 
\[ \left\{x\in R\mid \forall e\gg 0, x^{p^e}\mfa^{\lceil t(p^e-1)\rceil}\subset I^{[p^e]}\mfa^{\lceil t(p^e-1)\rceil}\right\} \subsetneq \atif.
\]
\end{ex}

By contrast, the other adjusted closure operations, by virtue of including the multiplier $c\in R^\circ$ in their definitions, are stable under modification of the exponent on $\mfa$.

\begin{theorem}\label{lem: multiplier closures can have linear shift in exponent} Let $u\in \R$.
If $R$ is a ring, $I, \mfa$ are ideals of positive height, and $t \in \R_{>0}$, then 
\begin{enumerate}[{\rm(a)}]
 	\item $\overline{I}^{\mfa^t} = \left\{x\in R\mid \text{ for some } c\in R^\circ\text{ and all }n\gg 0, cx^n\mfa^{\lceil tn + u\rceil}\subset I^n\right\}$.
 	\item $\istarat = \left\{x\in R\mid \text{ for some } c\in R^\circ\text{ and all }e\gg 0, cx^{p^e}\mfa^{\lceil tp^e + u\rceil}\subset I^{[p^e]}\right\}$.
 	\item $\atistar = \left\{x\in R\mid \text{ for some } c\in R^\circ\text{ and all }e\gg 0, cx^{p^e}\mfa^{\lceil tp^e + u\rceil}\subset I^{[p^e]}\mfa^{\lceil tp^e + u\rceil}\right\}$.
 \end{enumerate}   	
\end{theorem}

\begin{proof}
(a) Let $K_u= \left\{x\in R\mid \text{ for some } c\in R^\circ\text{ and all }n\gg 0, cx^n\mfa^{\lceil tn + u\rceil}\subset I^n\right\}$.  We wish to show that $K_u = \overline{I}^{\mfa^t}$.

Suppose $u>0$, so that $\mfa^{\lceil tn + u\rceil}\subset \mfa^{\lceil tn\rceil}$ and hence $\overline{I}^{\mfa^t}\subset K_u$.  Now suppose that $x\in K_u$, let $c'\in R^\circ$ such that $c'x^n\mfa^{\lceil tn + u\rceil}\subset I^n$ for $n\gg 0$, and let $c''\in \mfa^{\lceil u\rceil}\cap R^\circ$.  Now, for $n\gg 0$,
\[c'c''x^n\mfa^{\lceil tn\rceil} \subset c'x^n\mfa^{\lceil tn \rceil+ \lceil u\rceil}\subset c'x^n\mfa^{\lceil tn + u\rceil}\subset I^n,\]
and so $x\in \overline{I}^{\mfa^t}$.  

On the other hand, if $u < 0$, then \emph{a priori} $K_u \subset \overline{I}^{\mfa^t}$, so let $x\in \overline{I}^{\mfa^t}$, with a choice of $c'\in R^\circ$ such that $c'x^n\mfa^{\lceil tn\rceil} \subset I^n$ for all $n\gg 0$.  Choose $c''\in \mfa^{\lceil -u\rceil}\cap R^\circ$, and note that for $n\gg 0$, 
\[c'c''x^n\mfa^{\lceil tn + u\rceil} \subset c'x^n\mfa^{\lceil tn + u \rceil+ \lceil -u\rceil} \subset c'x^n\mfa^{\lceil tn\rceil}\subset I^n,\]
hence $x\in K_u$.

(b) Replacing $I^n$ with $I^{[p^e]}$ in the previous two paragraphs shows that the same result holds for Hara-Yoshida tight closure.

(c) Let $L_u= \left\{x\in R\mid \text{ for some } c\in R^\circ\text{ and all }e\gg 0, cx^{p^e}\mfa^{\lceil tp^e + u\rceil}\subset I^{[p^e]}\mfa^{\lceil tp^e + u\rceil}\right\}$.  We wish to show that $L_u = \atistar$.

First suppose $u>0$.  Since for all $e\geq 0$, $(I^{[p^e]}\mfa^{\lceil tp^e\rceil}:\mfa^{\lceil tp^e\rceil}) \subset (I^{[p^e]}\mfa^{\lceil tp^e+u\rceil}:\mfa^{\lceil tp^e+u\rceil})$, we have that $\atistar\subset L_u$.  Suppose that $x\in L_u$ and choose $c'\in R^\circ$ such that $c'x^{p^e}\mfa^{\lceil tp^e + u\rceil}\subset  I^{[p^e]}\mfa^{\lceil tp^e + u\rceil}$ for all $e\gg 0$.  Choose $c''\in \mfa^{\lceil u\rceil}\cap R^\circ$ and note that
\[c'c''x^{p^e}\mfa^{\lceil tp^e\rceil} \subset c'x^{p^e}\mfa^{\lceil tp^e + u\rceil}\subset  I^{[p^e]}\mfa^{\lceil tp^e + u\rceil} \subset I^{[p^e]}\mfa^{\lceil tp^e\rceil},\]
and so $x\in \atistar$.

If $u<0$, then \emph{a priori} $L_u\subset\atistar$.  Let $x\in \atistar$, let $c'\in R^\circ$ such that for $e\gg 0$, $c'x^{p^e}\mfa^{\lceil tp^e\rceil}\subset  I^{[p^e]}\mfa^{\lceil tp^e\rceil}$.  Choose $c''\in\mfa^{\lceil -u\rceil}\cap R^\circ$, and note that
\[c'c''x^{p^e}\mfa^{\lceil tp^e+u\rceil} \subset c'x^n\mfa^{\lceil tp^e + u \rceil+ \lceil -u\rceil} \subset c'x^{p^e}\mfa^{\lceil tp^e\rceil}\subset I^{[p^e]}\mfa^{\lceil tp^e\rceil} \subset  I^{[p^e]}\mfa^{\lceil tp^e+u\rceil},
\]
hence $x\in L_u$.
\end{proof}

We also note that we have a similar Brian\c{c}on-Skoda type containment between the $\mfa^t$-integral closure and sharp $F$-closure.

\begin{prop}\label{prop: BS for sharp F-closure} Let $R$ be a ring and $\mfa, I\subset R$ be ideals of positive height.  Let $I$ be generated by $d$ elements and let $t\in \R_{>0}$.  Then $\overline{I^{d+m}}^{\mfa^t}\subset (I^m)^{F^\sharp \mfa^t}$ for all $m\geq 1$.
\end{prop}

\begin{proof}
	Let $x\in \overline{I^{d+m}}^{\mfa^t}$.	As in the proof of \Cref{BS thm}(b), we may first assume that $R$ is reduced and that there exists $n'$ such that for all $n\gg 0$, $x^n\mfa^{\lceil tn\rceil}\subset I^{(d+m)n-n'}$. Therefore, for $e\gg 0$,
	\[x^{p^e}\mfa^{\lceil t(p^e-1)\rceil} \subset x^{p^e-1}\mfa^{\lceil t(p^e-1)\rceil}\subset  I^{(d+m)(p^e-1)-n'} \subset (I^{[p^e]})^{\lceil m - (d + m + n')/p^e\rceil} = (I^{[p^e]})^m = (I^m)^{[p^e]}.\]
	Therefore $x\in(I^m)^{F^\sharp \mfa^t}$. Generally, if $R$ is not reduced, then from the previous case, we know that for all $m \geq 1,$
	\[ \overline{I^{d+m}}^{\mfa^t} = \overline{I^{d+m}}^{\mfa^t} + \sqrt{(0)} \subset (I^m)^{F^\sharp\mfa^t} + \sqrt{(0)} = (I^m)^{F^\sharp\mfa^t}, \]
	where the last equality follows from the observation that $\sqrt{(0)} = (0)^F \subset (I^m)^{F^\sharp\mfa^t}.$ 
\end{proof}

\subsection{Adjusted closures in graded rings}\label{subsubsection: adjusted closures in graded rings}

\newcommand{\mb}[1]{\mathbf{#1}}
\newcommand{\supp}{\operatorname{Supp}}
To conclude this section, we say a few words about these adjusted closure operations in a graded setting and prove a graded version of the Brian\c{c}on-Skoda Theorem. First, we note that in $\Z^N$-graded rings, the adjusted closure operations applied to homogeneous ideals return homogeneous ideals. If $R$ is $\Z^N$-graded, for an $x=\sum_{\mb j \in \Z^N} x_{\mb j}$ in $R$, write $\supp x = \{ \mb j \in \Z^N \mid x_{\mb j} \neq 0\}$. Furthermore, we will endow $\Z^N$ with a total ordering $\le$ which respects addition, e.g. the lexicographic order. The key idea of the proof below is to use that the equations defining membership in the adjusted closures tend to isolate elements of maximal graded degree with respect to $\le$.

\begin{proposition}\label{prop: adjusted closures are homogeneous}
Suppose $R = \oplus_{\mathbf{m} \in \Z^N} R_{\mathbf{m}}$ is a $\Z^N$-graded ring, $I$ and $\mfa$ are homogeneous ideals of positive height in $R$, and $t \in \R_{>0}$. Then, the $\mfa^t$-integral closure $\overline{I}^{\mfa^t}$ is also a homogeneous ideal. If $R$ is of prime characteristic $p>0$, then each of the adjusted closures $I^{F\mfa^t}$, $I^{*\mfa^t}$, ${}^FI^{\mfa^t}$, and ${}^*I^{\mfa^t}$ are homogeneous ideals as well.
\end{proposition}

\begin{proof}
Define $I_n = I^n : \mfa^{\lceil tn \rceil}$, $J_e = I^{\fbp{p^e}}:\mfa^{\lceil t p^e \rceil}$, and $J_e' = I^{\fbp{p^e}}\mfa^{\lceil t p^e \rceil}:\mfa^{\lceil t p^e \rceil}$. Since products and colons of homogeneous ideals are homogeneous, we have that $I_n$, $J_e$, and $J_e'$ are  homogeneous ideals for all $n,e\in \mbn$.

We first handle the adjusted integral closure. Let $x \in R$ be an element of $\overline{I}^{\mfa^t}$ so that there is a $c \in R^\circ$ such that $cx^n \in I_n$ for all $n$. By \Cref{prop: hy integral closure reduces to domains}, we can assume without loss of generality that $R$ is a domain. Our goal is that $x_{\mb j} \in \overline{I}^{\mfa^t}$ for all $\mb j \in \supp x$. We will induce on the size of $\supp x$. If $x$ itself is homogeneous, then there's nothing to show. Now, suppose that if $|\supp y| = m$ and $cy^n \in I_n$ for all $n$, then each homogeneous component of $y$ is in $\overline{I}^{\mfa^t}$, and now suppose $x$ has $|\supp x| = m+1$. 

Let $\mb j_0$ be the maximum element of $\supp x$ under $\le$ and let $\mb i_0$ be the maximal element of $\supp c$ under $\le$. Then, an arbitrary homogeneous term of the multinomial expansion (suppressing the multinomial coefficient) of $c(\sum_{\mb j\in \Z^N} x_{\mb j})^n$, has the form $c_{\mb i}x_{\mb j_0}^{e_0}x_{\mb j_1}^{e_1}\cdots x_{\mb j_m}^{e_m}$ where $\sum_k e_k = n$. This element has degree \[\mb i + \sum_k e_k \mb j_k \le \mb i_0 + \sum_k e_k \mb j_0 = \mb i_0 + n\mb j_0\] with equality if and only if $\mb i = \mb i_0$ and $e_k = 0$ for $k > 0$. Hence, $c_{\mb i_0}x_{\mb j_0}^n$ is the only homogeneous term of $cx^n$ of graded degree $\mb i_0 + n \mb j_0$. Thus, as $I_n$ is homogeneous, $c_{\mb i_0}x_{\mb j_0}^n \in I_n$ for all $n$. Hence $x_{\mb j_0} \in \overline{I}^{\mfa^t}$, and so $x-x_{\mb j_0} \in \overline{I}^{\mfa^t}$ as well. By the induction hypothesis, each homogeneous component of $x-x_{\mb j_0}$ is in $\overline{I}^{\mfa^t}$, so each homogeneous component of $x$ is in $\overline{I}^{\mfa^t}$, concluding the proof.

Now we suppose that $R$ has prime characteristic $p>0$, and we discontinue the supposition that $R$ is a domain since the Frobenius closures cannot be computed by passing to quotient domains.

As before, fix $x = \sum_{\mb j\in \Z^N} x_{\mb j}$ and $c =\sum_{\mb i \in \Z^N} c_{\mb i} \in R^\circ$. Write $\mb j_0$ for the maximal element of $\supp x$ under $\le$. Then, for $\mb j \neq \mb j_0$, we have that the support of $cx_{\mb j}^{p^e}$ has maximum $\mb i_0+p^e\mb j$, which for $e \gg 0$ is less than $\mb i + p^e \mb j_0$ for all $\mb i \in \supp c$. Hence, if $cx^{p^e} \in J_e$, then $cx_{\mb j_0}^{p^e} \in J_e$ as $J_e$ is homogeneous. 

If $c=1$, this argument demonstrates that $x \in I^{F\mfa^t}$ implies $x_{\mb j_0} \in I^{F\mfa^t}$. If $c$ is simply an arbitrary element of $R^\circ$, this argument demonstrates that $x \in I^{*\mfa^t}$ implies $x_{\mb j_0} \in I^{*\mfa^t}$. Similar reasoning for membership in $J_e'$ holds, showing if $x \in {}^{\mfa^t}I^*$ (respectively $x \in {}^{\mfa^t}I^F$) then $x_{\mb j_0} \in {}^{\mfa^t}I^*$ (respectively $x_{\mb j_0} \in {}^{\mfa^t}I^F$). Then, applying an induction argument on the size of $\supp x$ as before, we see that $I^{*\mfa^t}$, $I^{F\mfa^t}$, ${}^{\mfa^t}I^*$, and ${}^{\mfa^t}I^F$ are homogeneous. 
\end{proof}

We now switch to the more common setting of an $\mbn$-graded ring. 

\begin{setting}\label{setting: graded setup}
Let $R = \oplus_{n \geq 0} R_n$ be a reduced, finitely generated graded algebra over a field $R_0=\mathbb{F}$, and further let $R$ be of positive Krull dimension. Write $R_+ = \oplus_{n > 0} R_n$ for the unique maximal homogeneous ideal of $R,$ and let $y_1, \ldots, y_s$ be algebra generators for $R$ over $R_0$, with degrees $\beta_1, \ldots, \beta_s$ respectively. Set $\beta = \max(\beta_i)$ and $\beta' = \min(\beta_i).$ Finally, we let $\mfa$ be a homogeneous $R_+$-primary ideal, so that $R^k_+ \subset \mfa \subset R^l_+$ for some integers $l \leq k.$
\end{setting}

\begin{proposition}\label{at high degrees, HY and V correspond}
Given \Cref{setting: graded setup}, let $I = (f_1, \ldots, f_d)$ be a homogeneous ideal, $t\in \R_{>0}$, and $x \in R_{\geq N}$ with $N = (\beta k - \beta'l)t + \max(\deg(f_i)).$ 

\begin{enumerate}[{\rm(a)}]
\item Then, $x \in \overline{I}^{\mfa^t}$ if and only if $x \in$ $^{\mfa^t}\overline{I} = \overline{I}$, and 

\item If $t\in \Z[1/p]$ or $\deg(x) > N$, then $x \in I^{F\mfa^t}$ if and only if $x \in$ $^{\mfa^t}I^F.$
\end{enumerate}
\end{proposition}

\begin{proof} (a) Assume that $x \in \overline{I}^{\mfa^t}$ and $\deg(x) \geq N.$ By definition, there exists $c \in R^\circ$ such that $cx^n\mfa^{\lceil tn\rceil} \subset I^n$ for all $n$ large.  Without loss of generality we may assume that $\deg(c) \geq \beta k - \beta' l$. Fix $n$ and let $h \in \mfa^{\lceil tn\rceil}$ be a homogeneous element. Then $\deg(h) \geq \beta'l\lceil tn\rceil.$ Write $cx^nh = \sum_{|\vec m|=n} b_{\vec m} f_1^{m_1} \cdots f_d^{m_d}$ where $b_i \in R$ are homogeneous elements, so that 
	\begin{align*}
	 	\deg(b_{\vec m}) & \geq \deg(c) + n \deg(x) + \deg(h) - n\max(\deg(f_i)) \\ 
	 	 & \geq \deg(c) + \beta k tn +\beta' l (\lceil tn\rceil - tn)\\ 
	 	 & = \deg(c) + \beta k \lceil tn\rceil -(\beta k - \beta' l) (\lceil tn\rceil - tn)\\ 
	 	 &\geq \beta k \lceil tn\rceil 
	 \end{align*}  
	for each $\vec m.$ This implies that $b_{\vec m} \in R_{\geq \beta k\lceil t  n\rceil} \subset R_+^{k\lceil tn\rceil} \subset \mfa^{\lceil tn\rceil},$ hence the result.

	(b) Suppose that $x\in \ifat$, so that for all $e\gg 0$, $x^{p^e}\mfa^{\lceil tp^e\rceil} \subset I^{[p^e]}$.  Let $\delta = \deg(x) - N$.
	Fix $e$ and let  $h \in \mfa^{\lceil tp^e\rceil}$ be a homogeneous element.  Then $\deg(h) \geq \beta'l\lceil tp^e\rceil.$ Write $x^{p^e}h = \sum_{i = 1}^d b_i f_i^{p^{e}}$ where $b_i \in R$ are homogeneous elements, so that
	\begin{align*}
		\deg(b_i) &= p^e\deg(x) + \deg(h) - p^e\deg(f_i) \\ 
		& \geq p^e\left((\beta k - \beta'l)t + \max(\deg(f_i)) + \delta\right) + \beta'l\lceil tp^e\rceil - p^e\max(\deg(f_i)) \\ 
		& =(\beta k - \beta'l)tp^e + \delta p^e + \beta'l\lceil tp^e\rceil \\ 
		& = \beta k\lceil tp^e\rceil - (\beta k -\beta'l)(\lceil tp^e \rceil - tp^e) + \delta p^e.
	\end{align*}
	If $t\in \Z[1/p]$ and $e$ is large enough that $tp^e\in \Z$, then  
	\[\beta k\lceil tp^e\rceil - (\beta k -\beta'l)(\lceil tp^e \rceil - tp^e) + \delta p^e = \beta k\lceil tp^e\rceil + \delta p^e\geq \beta k\lceil tp^e\rceil.\]
	On the other hand, if $t\notin \Z[1/p]$, then $\delta >0$, and for all $e$ large enough that $\delta p^e \geq \beta k -\beta'l$,
	\[
		\beta k\lceil tp^e\rceil - (\beta k -\beta'l)(\lceil tp^e \rceil - tp^e) + \delta p^e > \beta k\lceil tp^e\rceil - (\beta k -\beta'l) + (\beta k -\beta'l) = \beta k\lceil tp^e\rceil.\]
		In either case, $\deg(b_i)\geq \beta k \lceil tp^e\rceil$, and so $b_i\in \mfa^{\lceil tp^e\rceil}$, showing that $x\in \atif$.
		\end{proof}

The Brian\c con-Skoda Theorem for graded rings and ideals generated by homogeneous system of parameters was proved in \cite[Theorem~7.7]{HHSplit} (and in \cite[Proposition~3.3]{Smith97} for homogeneous $R_+$-primary ideals). We prove a similar result for the Hara-Yoshida $\mfa$-tight closure. Since these ideals are bigger, as expected, more graded components of the ring are contained in the adjusted closure as compared to the ordinary tight closure. 

\begin{proposition}[Graded Brian\c con-Skoda Theorem for adjusted closures] \label{prop: graded briancon-skoda}
Given \Cref{setting: graded setup}, suppose $\operatorname{char} \mathbb{F} = p>0$ and $R$ is locally equidimensional. Let $I = (f_1, \ldots, f_d)$ be such that $f_1,\ldots,f_d$ are homogeneous system of parameters, $t\in \R_{>0}$, and $N = \sum_{i=1}^{d} \deg f_i  -t\beta'l.$

Suppose $R$ has a parameter test element\footnote{Recall an element $c \in R^\circ$ is a parameter test element if, for all $e \ge 0$ and all  ideals $\mathfrak{q}$ generated by a system of parameters of $R$, $c(\mathfrak{q}^*)^{\fbp{p^e}} \subset \mathfrak{q}^{\fbp{p^e}}$.}, e.g.\ $\mathbb{F}$ is $F$-finite. Let $x \in R_{\geq N}.$ Then $x \in I^{*\mfa^t}.$ In other words, $R_{\geq N} + I^{F\mfa^t} \subset I^{*\mfa^t}.$  
\end{proposition}

\begin{proof}
	By \Cref{prop: cl via min primes}, we may assume that $R$ is a domain.
	We first prove the case in which $\deg f_i = \deg f_1 = \alpha$ for all $1 \leq i \leq d$, so that $N = d\alpha - t\beta'\ell$. Pick $0\neq v \in \mfa$ a homogeneous element.  For any $r\in\N$, since $R$ is module-finite over $S = \mathbb{F}[f_1,\ldots,f_d],$ the monomial $x^{r} v^{\lfloor rt\rfloor}$ satisfies an integral equation $(x^{r} v^{\lfloor rt\rfloor})^m + b_1(x^{r} v^{\lfloor rt\rfloor})^{m-1} + \cdots + b_m = 0$ where $b_i \in S$ for all $i.$ We may assume that $b_i$ are homogeneous elements. So for any $b_i\neq 0$, 
	\[\deg b_i = i \cdot \deg(x^{r} v^{\lfloor rt\rfloor}) \geq  i\cdot \left(r N + \lfloor rt\rfloor\beta'\ell\right) = i \cdot\left(  rd\alpha - rt\beta'\ell + \lfloor rt\rfloor\beta'\ell\right).\]
	We will fix $r$ such that $rt\beta'\ell - \lfloor rt\rfloor\beta'\ell < 1$, so that the inequality above will imply that $\deg b_i\geq i\cdot rd\alpha$.  If $t$ is rational, then take $r$ such that $rt\in \Z$, so that $\lfloor rt\rfloor = rt$.  Otherwise, if $t$ is irrational, then we may take $r\in \N$ such that $rt - \lfloor rt\rfloor < 1/(\beta'\ell)$. Hence 
	\[ b_i \in R_{i\cdot \deg(x^rv^{\lfloor rt\rfloor})} \cap S \subset (f_1,\ldots,f_d)^{i\cdot \deg(x^rv^{\lfloor rt\rfloor})/\alpha} \subset (f_1,\ldots,f_d)^{i\cdot rd}.\] 
	This implies that $x^rv^{\lfloor rt\rfloor} \in \overline{I^{rd}}.$ Thus there exists $c \in R^\circ$ such that $c(x^rv^{\lfloor rt\rfloor})^{n} \in I^{nrd}$ for all $n \geq 0.$

	We claim that $x \in \overline{(I^d)}^{\mfa^t}.$ Let $\mfa = (v_1,\ldots,v_\mu).$ For $1 \leq i \leq \mu,$ we know that there exists $c_i \in R^{\circ}$ and $r_i\in \N$ such that $c_i(x^{r_i}v^{\lfloor r_it\rfloor})^{n} \in I^{nr_id}$ for all $n$. Fix $n_1,\ldots, n_\mu\in\N$ such that $n_1+\cdots+n_{\mu}=\lceil tn\rceil,$ i.e.\ $ v_1^{n_1} \cdots v_\mu^{n_\mu}$ is a generator of $\mfa^{\lceil  tn\rceil}$. For each $i=1,\ldots,\mu$, let $q_i = \lfloor n\cdot n_i/\lceil tn\rceil\rfloor$.  Note that $\sum_i q_i \leq \sum_i n\cdot n_i/\lceil tn\rceil =n$.  
	Also, note that $n_i/(r_it) = nn_i/(tnr_i) \geq nn_i/(\lceil tn\rceil r_i)\geq q_i/r_i$.

	Now we have that for any $i$,
	\[c_ix^{q_i}v_i^{n_i}
\in c_i\left(x^{r_i\lfloor q_i/r_i\rfloor}v_i^{\lfloor r_it\rfloor \lfloor n_i/(r_it)\rfloor}\right)
\subset c_i\left(x^{r_i\lfloor q_i/r_i\rfloor}v_i^{\lfloor r_it\rfloor \lfloor q_i/r_i\rfloor}\right)
\subset c_i\left(x^{r_i} v_i^{\lfloor r_it\rfloor }\right)^{\lfloor q_i/r_i\rfloor}
\subset I^{\lfloor q_i/r_i\rfloor r_id}.
	\]
	Therefore,
	\[ \left( \prod_{i=1}^\mu c_i\right) x^n v_1^{n_1} \cdots v_\mu^{n_\mu} \in 
	\prod_{i=1}^\mu \big( c_ix^{q_i} v_i^{n_i} \big) \subset 
	\prod_{i=1}^\mu I^{\lfloor q_i/r_i\rfloor r_id} = I^{d\sum_i\lfloor q_i/r_i \rfloor r_i}. \]
	Now 
	\[\sum_{i=1}^\mu\left\lfloor \dfrac{q_i}{r_i}\right\rfloor r_i > \sum_{i=1}^\mu \left(\dfrac{1}{r_i}\left\lfloor \dfrac{n \cdot n_i}{\lceil tn\rceil}\right\rfloor - 1 \right)r_i
	>\sum_{i=1}^\mu\left(\dfrac{n\cdot n_i}{\lceil tn\rceil}-1-r_i\right)= n - \sum_{i=1}^\mu (r_i+1),\] hence $d\sum_i\lfloor q_i/r_i\rfloor r_i > nd - d\sum_i(r_i+1)$.  Therefore, taking $0\neq c'\in I^{d\sum_i(r_i+1)}$, we have that for all $n\geq 1$,
	\[c'\left( \prod_{i=1}^\mu c_i\right)x^n\mfa^{\lceil tn\rceil}\subset I^{nd}\]
	implying that $x \in \overline{(I^d)}^{\mfa^t} \subset I^{*\mfa^t}$ using \Cref{BS thm}. Note that for this case we do not require the existence of a parameter test element or equidimensionality. 
		
	Now suppose that not all $f_i$'s have the same degree. Choose nonnegative integers $\alpha_i$ such that $F_1 = f_1^{\alpha_1+1},\ldots, F_d = f_d^{\alpha_d+1}$ have equal degree $D$. Let $x \in R_{\geq N}$. Then $xf_1^{\alpha_1} \cdots f_d^{\alpha_d}$ has degree at least 
	\[\sum_{i=1}^{d} \deg f_i -t\beta'l + \sum_{i=1}^d {\alpha_i\deg(f_i)}
	= dD -t\beta'l.\]
Hence, by the first case, $xf_1^{\alpha_1} \cdots f_d^{\alpha_d}\in (F_1,\ldots,F_d)^{*\mfa^t}.$ By definition, there exists $c \in R^\circ$ such that for all $e$ large,
	\[ c(xf_1^{\alpha_1} \cdots f_d^{\alpha_d})^{p^e} \mfa^{\lceil tp^e\rceil} \subset (f_1^{(\alpha_1+1)p^e}, \ldots, f_d^{(\alpha_d+1)p^e}). \]
	This implies that 
	\[ cx^{p^e} \mfa^{\lceil tp^e\rceil} \subset (f_1^{(\alpha_1+1)p^e}, \ldots, f_d^{(\alpha_d+1)p^e}) : (f_1^{\alpha_1} \cdots f_d^{\alpha_d})^{p^e} \subset (f_1^{p^e},\ldots,f_d^{p^e})^* \] where the last inclusion follows from the colon capturing property (see \cite[Theorem~2.3(b)]{SixLec}). Multiplying by a parameter test element $c'$, we get $cc'x^{p^e}\mfa^{\lceil tp^e\rceil} \subset (f_1,\ldots,f_d)^{[p^e]}$ for all $e$ large implying that $x \in I^{*\mfa^t}$.	
\end{proof}

\section{Applications: Adjusted normalizations of domains and \texorpdfstring{$F$}{F}-nilpotence of pairs}

Suppose $R$ is a domain. It is an elementary exercise to see that the normalization of $\overline{R}$ can be viewed as the set of fractions $r/s$ in $\Frac R$ such that $r$ is in $\overline{(s)}$. Similarly, for domains of prime characteristic, the weak normalization of $R$, written ${}^*R$, is the set of fractions $x/y \in \Frac(R)$ such that $x \in (y)^F$. From this perspective, it is natural to consider submodules of $\Frac R$ defined in terms of our adjusted closures as well. Recall an $R$-submodule $I$ of $\Frac R$ is called a \textbf{fractional ideal} if there is a nonzero $r \in R$ such that $rI \subset R$. Fractional ideals are exceptionally fruitful objects of study in the context of Dedekind domains found in number theory. 
 
\begin{defn}
Let $R$ be a domain, $\mfa\subset R$ a nonzero ideal, and $t \in \mathbb{R}_{> 0}.$ Then, we define the \textbf{(Hara-Yoshida) $\mfa^t$-normalization of $R$}, written $\overline{R}^{\mfa^t}$, to be the set \[
\overline{R}^{\mfa^t} = \left\lbrace \left. \frac{r}{s}\in \Frac R \,\right|\, r \in \overline{(s)}^{\mfa^t}\right\rbrace.
\]
\end{defn} 

First, note that membership in $\overline{R}^{\mfa^t}$ is well-defined. Indeed, if $r/s = x/y \in \Frac R$ and $r \in \overline{(s)}^{\mfa^t}$, then there is a $c\neq 0$ such that for all $n$ and all $w \in \mfa^{\lceil tn \rceil}$, there is a $u \in R$ such that $cr^n w = us^n.$ Multiplying $y^n$ on both sides we have $cr^ny^n w = us^ny^n$, however $r^ny^n = s^nx^n$ as the fractions $r/s$ and $x/y$ are equal in $\Frac R$. Hence $cr^ny^nw = cs^nx^nw = us^ny^n$, and canceling $s^n$ we have $cx^nw = uy^n$. As $n$ and $w \in \mfa^{\lceil tn \rceil}$ were arbitrary, this shows $x \in \overline{(y)}^{\mfa^t}$.

Recall from \cite{HY}, the \textbf{$\mfa^t$-test ideal of $R$} is the ideal generated by the $\mfa^t$-test elements, that is the elements $c \in R^\circ$ such that, for all ideals $I$ in $R$ and all $x \in I^{*\mfa^t}$, we have  $cx^{p^e}\mfa^{\lceil tp^e\rceil} \subset I^{\fbp{p^e}}$.

\begin{prop}\label{prop: a^t integral closure of ring is submodule}
Let $R$ be a domain, $\mfa\subset R$ a nonzero ideal, and $t\in \R_{>0}$. Then, $\overline{R}^{\mfa^t}$ is an $R$-submodule of $\Frac R$. Furthermore, when $R$ has prime characteristic $p>0$ and the $\mfa^t$-test ideal is nonzero (e.g. $R$ is $F$-finite), $\overline{R}^{\mfa^t}$ is a fractional ideal of $R$.
\end{prop}

\begin{proof}
Let $r/s$ and $x/y \in \overline{R}^{\mfa^t}$. Then, there are $c_1, c_2$ nonzero such that $c_1r^n \mfa^{\lceil tn \rceil} \subset (s^n)$ and $c_2x^n \mfa^{\lceil tn \rceil} \subset (y^n)$ for all $n.$ We want $(ry+sx)/(sy) \in \overline{R}^{\mfa^t}$, so we seek a $c\neq 0$ such that $c(ry+sx)^n \mfa^{\lceil tn\rceil} \subset (s^ny^n)$. Taking a term $(ry)^i(sx)^{n-i}$ of the binomial expansion, if we apply $c_1c_2$ we can rearrange to obtain $(c_1r^i)(c_2x^{n-i})y^is^{n-i}$. Furthermore, by picking $c_3 \in \mfa\setminus \{0\}$, $c_3\mfa^{\lceil tn \rceil} \subset \mfa^{\lceil ti \rceil}\mfa^{\lceil t(n-i)\rceil}$ since $\lceil tn\rceil + 1 \geq \lceil ti\rceil + \lceil t(n-i)\rceil$. Letting $c=c_1c_2c_3$, we then have \[
c((ry)^i(sx)^{n-i})\mfa^{\lceil tn \rceil} \subset (c_1r^i\mfa^{\lceil ti\rceil} s^{n-i})(c_2x^{n-i}\mfa^{\lceil t(n-i) \rceil}y^i)\subset (s^n)(y^n)=(s^ny^n),
\] for each term, so $c(ry+sx)^n \mfa^{\lceil tn \rceil} \subset (s^ny^n)$ as required. Now, if $r \in R$ and $x/y\in \overline{R}^{\mfa^t}$, we have $r(x/y)=(rx)/y$, and if $cx^n\mfa^{\lceil tn \rceil} \subset (y^n)$, then $c(rx)^n\mfa^{\lceil tn \rceil}\subset r^n(y^n)\subset (y^n)$, so $rx/y \in \overline{R}^{\mfa^t}$. This shows that $\overline{R}^{\mfa^t}$ is an $R$-submodule of $\Frac R$.

Finally, if $R$ is of prime characteristic $p>0$, by \Cref{prop: tight closure is integral closure for principal ideals}, $\overline{(s)}^{\mfa^t} = (s)^{*\mfa^t}$. So if $c$ is a nonzero element of the Hara-Yoshida $\mfa^t$-test ideal, then for any $r/s \in \overline{R}^{\mfa^t}$, $cr = us$ for some $u \in R$, so $cr/s = u/1 \in R$. Hence, $\overline{R}^{\mfa^t}$ is a fractional ideal.
\end{proof}

\begin{remark}
It is unclear to us whether $\overline{R}^{\mfa^t}$ is a fractional ideal outside of the case that $R$ is of prime characteristic $p>0$. This is equivalent to the ideal \[\mathfrak{c}=\bigcap_{x \neq 0} (x):\overline{(x)}^{\mfa^t}\] being nonzero. Further, it is clear that $\mathfrak{c}$ is contained in the conductor ideal of $R\rightarrow \overline{R}$. To show $\mathfrak{c}$ is nonzero, one would usually appeal to localization arguments, but $\overline{(x)}^{\mfa^t}$ does not readily localize.
\end{remark}

We can consider variants of the construction of $\overline{R}^{\mfa^t}$ for each one of the adjusted closures we've discussed, so we will write $R^{*\mfa^t}$, $R^{F\mfa^t}$, ${}^{\mfa^t}R^*$ and ${}^{\mfa^t}R^F$ for these objects. As noted in the proof of the previous proposition, since $\overline{(s)}^{\mfa^t}$ = $(s)^{*\mfa^t}$, we have $\overline{R}^{\mfa^t} = R^{*\mfa^t}$ and similarly ${}^{\mfa^t}R^* = \overline{R}$ as the Vraciu integral closure is the same as the ordinary integral closure. So, from our new closures, we obtain three new subsets of $\Frac R$; $\overline{R}^{\mfa^t}$, $R^{F\mfa^t}$, and ${}^{\mfa^t}R^F$. We refer to the latter two as the \textbf{Hara-Yoshida $\mfa^t$-weak normalization of $R$} and the \textbf{Vraciu $\mfa^t$-weak normalization of $R$} respectively. Furthermore, with minor alterations to the proof of the previous proposition, it is easy to see $R^{F\mfa^t}$ and ${}^{\mfa^t}R^F$ are also $R$-submodules of $\Frac R$. In fact, more is true of ${}^{\mfa^t}R^F$.

\begin{prop}\label{prop: V a^t Frob closure of ring is subring of frac field}
Let $R$ be a domain, $\mfa \subset R$ a nonzero ideal, and $t\in \R_{>0}$. Then, ${}^{\mfa^t}R^F$ is a subring of $\overline{R}$. 
\end{prop}

\begin{proof}
Write $S={}^{\mfa^t}R^F$. We have already discussed that $S$ is an $R$-submodule of $\Frac R$, so we need only show that $S$ is closed under multiplication. If $r/s$ and $x/y$ are in $S$, we need $rx \in {}^{\mfa^t}(sy)^F$. We have that \[(rx)^{p^e} \mfa^{\lceil tp^e\rceil} = r^{p^e}(x^{p^e}\mfa^{\lceil tp^e\rceil})\subset r^{p^e}((y^{p^e})\mfa^{\lceil tp^e\rceil})\subset (r^{p^e}\mfa^{\lceil tp^e\rceil})(y^{p^e}) \subset (s^{p^e})(y^{p^e})\mfa^{\lceil tp^e\rceil} = (sy)^{\fbp{p^e}}\mfa^{\lceil tp^e\rceil} \] which shows $(rx)/(sy)\in S$. Furthermore, that $S\subset \overline{R}$ is simply the fact that ${}^{\mfa^t}(s)^F \subset \overline{(s)}$. 
\end{proof}

\begin{remark}\label{rmk: little dipper}
This subring ${}^{\mfa^t}R^F$ contains $R$, but also contains the weak normalization of $R$. In particular, for any nonzero ideal $\mfa\subset R$ and $t \in \R_{>0}$, we have the following diagram of inclusions, essentially from the diagram in \Cref{rmk: closure comparison diagram}. The relative simplicity of this diagram to the other is explained by \Cref{prop: tight closure is integral closure for principal ideals}, as the tight closure row collapses.
\begin{center}
\begin{tikzpicture}
\node at (0,0) {
\begin{tikzcd}
R \arrow[hook]{r}& {}^*R \arrow[hook]{r} & {}^{\mfa^t}R^F \arrow[hook]{r} \arrow[hook]{d} & R^{F\mfa^t} \arrow[hook]{d} \\
\, & \, & \overline{R} \arrow[hook]{r} & \overline{R}^{\mfa^t}
\end{tikzcd}
};
\end{tikzpicture}
\end{center}
\end{remark}

We give examples to show that these inclusions may be proper.  First we show that the Hara-Yoshida type normalizations may be larger than the usual normalization.

\begin{ex}\label{ex: little dipper example 1}
	Let $k$ be a field of prime characteristic $p>0$, $R = k[x]$, $\mfa = (x^m)$ for some $m\in\N$, and $t\in \R_{>0}$.  We will show that $1/x^{\lfloor mt\rfloor} \in\Frac{R}$ is an element of $R^{F\mfa^t}$ by demonstrating that $1\in (x^{\lfloor mt\rfloor})^{F\mfa^t}$. For $e\gg 0$, $m\lceil tp^e\rceil \geq mtp^e \geq \lfloor mt\rfloor p^e$, and so 
	\[1\cdot\mfa^{\lceil tp^e\rceil}  = (x^{m\lceil tp^e\rceil}) \subset (x^{\lfloor mt\rfloor})^{[p^e]},\]
	showing that $1\in (x^{\lfloor mt\rfloor})^{F\mfa^t}$.  Therefore $1/x^{\lfloor mt\rfloor} \in R^{F\mfa^t}$, and if $\lfloor mt\rfloor \geq 1$, $1/x^{\lfloor mt\rfloor} \notin R = \overline{R}$.  Thus in this example the inclusions ${}^{\mfa^t}R^F\subsetneq R^{F\mfa^t}$ and $\overline{R} \subsetneq \overline{R}^{\mfa^t}$ are proper.
\end{ex}

The next example is of a ring $R$ and ideal $\mfa$ such that ${}^*R \subsetneq {}^{\mfa^t}R^F \subsetneq \overline{R}$ for all $t\in \R_{>0}$.

\begin{ex}\label{ex: little dipper example 2}
Let $k$ be a field of characteristic $3$, $S=k[x,y,z]$, and  $R$ the $k$-subalgebra of $S$ generated by monomials $x^4,y^4$, and $x^iy^iz$ for $0\leq i,j\leq 3$.  In other words, $R = k[x^4,y^4]\oplus zk[x,y,z]$ as $k$-vector spaces, or $R$ can be thought of as the affine semigroup ring $k[A]$ for $A$ the affine semigroup generated by the columns of the $3\times 18$ matrix below. \[
A = \left[
\begin{array}{cccccccccccccccccc}
4 & 0 & 0 & 1 & 2 & 3 & 0 & 1 & 2 & 3 & 0 & 1 & 2 & 3 & 0 & 1 & 2 & 3 \\
0 & 4 & 0 & 0 & 0 & 0 & 1 & 1 & 1 & 1 & 2 & 2 & 2 & 2 & 3 & 3 & 3 & 3 \\
0 & 0 & 1 & 1 & 1 & 1 & 1 & 1 & 1 & 1 & 1 & 1 & 1 & 1 & 1 & 1 & 1 & 1 
\end{array}\right]
\]
 We will show that the inclusions ${}^*R\subsetneq {}^{\mfa^t}R^F \subsetneq \overline{R}$ are proper by determining which monomials $x^ay^b$ belong to them.

We have that $\overline{R}=S$: for any $x^ay^b\in S$, we can write $x^ay^b = (x^ay^bz)/z$.  Now for any $n\in\N$, $z(x^ay^bz)^n = x^{an}y^{bn}z\cdot z^n \in (z)^n$. Therefore $x^ay^bz\in\overline{(z)}$, and so $x^ay^b = (x^ay^bz)/z\in\overline{R}$.  Since $S$ is normal and $S\subset \overline{R}$, they are equal.

Next we consider the weak normalization ${}^*R$ of $R$.  We have that $x^ay^b = (x^ay^bz)/z\in {}^*R$ if and only if $x^ay^bz\in (z)^F$.  We claim that $x^ay^bz\in (z)^F$ if and only if $4|a$ and $4|b$, i.e.\ $x^ay^b\in R$.  The backwards direction is immediate.
For the forward direction, suppose that $x^ay^bz\in (z)^F$, so that for all $e\gg 0$, $x^{3^ea}y^{3^eb}z^{3^e} = z^{3^e}f$ for some monomial $f\in R$.  This implies that $f = x^{3^ea}y^{3^eb}\in R$, and so $4|3^ea$ and $4|3^eb$.  Therefore $4|a$ and $4|b$.
Hence $x^ay^b\in {}^*R$ if and only if $x^ay^b\in R$.

Now we turn our attention to the Vraciu $\mfa^t$-weak normalization  ${}^{\mfa^t}R^F$ of $R$, with $\mfa = (x^2yz, xy^2z)$ and $t\in \R_{>0}$.  We claim that $x^ay^b\in {}^{\mfa^t}R^F$ exactly when $4|(a+b)$.  We will show this by showing that $x^ay^bz\in {}^{\mfa^t}(z)^F$.  

Suppose that $4$ divides $a + b$.  Also, suppose that $4$ divides neither $a$ nor $b$, since in that case we already have that $x^ay^b\in R$.  Note that, for a fixed $e$,
\begin{equation}\label{vraciu f-closure of rings eq 1}
	(x^ay^bz)^{3^e}\mfa^{\lceil 3^et\rceil} = (x^ay^bz)^{3^e}(x^2yz,xy^2z)^{\lceil 3^et\rceil}.
\end{equation}
As generators of $\mfa^{\lceil 3^et\rceil}$ are of the form $(x^2yz)^i(xy^2z)^{\lceil 3^et\rceil - i}$, $0\leq i\leq \lceil 3^et\rceil$, the left hand side of \eqref{vraciu f-closure of rings eq 1} is generated by all monomials of the form
\begin{align*}
	(x^ay^bz)^{3^e}(x^2yz)^i(xy^2z)^{\lceil 3^et\rceil - i}
	& = x^{3^ea + \lceil 3^et\rceil +i}y^{3^eb + 2\lceil 3^e t\rceil -i}z^{3^e + \lceil 3^e t\rceil}.
\end{align*}

Let $q$ be the smallest nonnegative integer such that $j = \lceil 3^et\rceil -3^ea - i + 4q$ is nonnegative.  We have assumed that $a>0$ and we may assume that $e\geq 2$, so that $3^ea>4$, which implies that $j \leq \lceil 3^e t\rceil$.  Now we have that 
\[3^ea + \lceil 3^et\rceil +i = 4q + 2\lceil 3^et\rceil - j \quad
\text{and}\quad
3^eb + 2\lceil 3^e t\rceil -i = 3^eb + \lceil 3^et\rceil +3^ea-4q + j.\]
By the construction of $q$ we have that either $q=0$ or $j\leq 3$, and in the second case, 
\[3^ea+3^eb-4q = 3^eb+\lceil 3^et\rceil -i-j\geq 3^eb-j\geq 0\]
since we have assumed that $b>0$.  Hence, in any case, $3^ea+3^eb-4q \geq 0$.
Therefore,
\begin{align*}
	x^{3^ea + \lceil 3^et\rceil +i}y^{3^eb + 2\lceil 3^e t\rceil -i}z^{3^e + \lceil 3^e t\rceil}
	 & = x^{4q + 2\lceil 3^et\rceil - j}y^{3^eb + \lceil 3^et\rceil +3^ea-4q + j}z^{3^e + \lceil 3^e t\rceil} \\ 
	& = x^{4q}y^{3^ea+3^eb - 4q}z^{3^e}(x^2yz)^{\lceil 3^et\rceil - j}(xy^2z)^{j}\\ 
	&\in (z)^{[3^e]}\mfa^{\lceil 3^et\rceil}
\end{align*}
since $x^{4q}y^{3^ea+3^eb - 4q}\in R$. Hence $x^ay^bz\in {}^{\mfa^t}(z)^F$.

For the other direction, suppose that $a+b$ is not divisible by $4$.  For any $e$, we have that
$(x^ay^bz)^{3^e}\mfa^{\lceil 3^et\rceil}$ contains the monomial $x^{3^ea} y^{3^e b}z^{3^e}(xy^2z)^{\lceil 3^et\rceil} = x^{3^ea+\lceil 3^et\rceil}y^{3^eb+2 \lceil 3^et\rceil}z^{3^e + \lceil 3^et\rceil}$.  If this were in $(z)^{[3^e]}\mfa^{\lceil 3^et\rceil} = z^{3^e}(x^2yz,xy^2z)^{\lceil 3^et\rceil},$ then there would need to exist some $0\leq i\leq \lceil 3^et\rceil$ such that
\[x^{3^ea+\lceil 3^et\rceil}y^{3^eb+2\lceil 3^et\rceil}z^{3^e + \lceil 3^et\rceil}
\in (z^{3^e}(x^2yz)^i(xy^2z)^{\lceil 3^et\rceil - i}) = (x^{\lceil 3^et\rceil +i}y^{2\lceil 3^et\rceil -i}z^{3^e+\lceil 3^et\rceil}),    \]
 which would mean that $x^{3^ea-i}y^{3^eb + i}\in R$. This implies that $4$ divides both $3^ea - i$ and $3^b + i$, hence divides $3^e(a+b)$, which is a contradiction.  Therefore $x^ay^b\in {}^{\mfa^t}R^F$ if and only if $4\mid (a + b)$, so that ${}^{\mfa^t}R^F = k[x^4,x^3y,x^2y^2,xy^3,y^4]\oplus zk[x,y,z]$.

Consequently, we have $R \subsetneq {}^{\mfa^t}R^F \subsetneq \overline{R}$, and in fact that ${}^{\mfa^t}R^F$ is the affine semigroup ring $k[B]$ for $B$ the affine semigroup generated by the columns of the matrix below. \[
B=\left[
\begin{array}{ccccccccccccccc}
4 & 0 & 1 & 2 & 3 & 0 & 1 & 2 & 3 & 0 & 1 & 2 & 0 & 1 & 0\\
0 & 4 & 3 & 2 & 1 & 0 & 0 & 0 & 0 & 1 & 1 & 1 & 2 & 2 & 3\\
0 & 0 & 0 & 0 & 0 & 1 & 1 & 1 & 1 & 1 & 1 & 1 & 1 & 1 & 1
\end{array}
\right]
\]
\end{ex}

We conclude our discussion of the adjusted closures of rings by showing that when $\mfa$ contains a regular sequence of length 2, then $\overline{R} = \overline{R}^{\mfa^t}$.  First we prove a more general result.

\begin{prop}\label{lem:length 2 regular sequence colon ideal} Let $R$ be a ring, $g$ a regular element of $R$, and $x,y$ a regular sequence in $R$.  For any $n\in \N$, $(g):(x,y)^n = (g)$.	
\end{prop}

\begin{proof}
	We induce on $n$, the  result holding when $n=0$ since $(g)~:~(x,y)^0 = (g)~:~R = (g)$.  Suppose now that $n\geq 1$ and that $(g):(x,y)^{n-1} = (g)$.  Let $f\in (g):(x,y)^n$. For any $0\leq i\leq n-1$, $x^{n-i}y^{i}$ and $x^{n-i-1}y^{i+1}$ are in $(x,y)^n$, and so we have that $fx^{n-i}y^{i} = gh$ and $fx^{n-i-1}y^{i+1} = gh'$ for some $h,h'\in R$.  Multiplying the first equation by $y$ and the second equation by $x$ we have that 
	\[ghy = fx^{n-i}y^{i+1} = gh'x.\]
	Since $g$ is a regular element, $hy = h'x$.  Since $y$ is regular over $R/(x)$, $x$ must divide $h$.  Let $h''\in R$ be such that $h = xh''$.  Now, we have that $fx^{n-i}y^{i} = gh = gxh''$, and since $x$ is a regular element, $fx^{n-1-i}y^i=gh'' \in (g)$.
  	
	Since this holds for all $0\leq i\leq n-1$, and since the elements $x^{n-1-i}y^{i}$ are the generators of $(x,y)^{n-1}$, we have that $f\in (g):(x,y)^{n-1} = (g)$.\end{proof}

\begin{corollary}If $R$ is a domain, and $\mfa\subset R$ is an ideal with $\mathrm{grade}(\mfa)\geq 2$, then $\overline{R}^{\mfa^t} = \overline{R}$ for all $t\in\R_{>0}$.
	
\end{corollary}

\begin{proof}Let $\dfrac{f}{g}\in \overline{R}^{\mfa^t}$, so that there exists $0\neq c\in R$ such that for all $n$, $cf^n \in (g)^n:\mfa^{\lceil tn\rceil}$.  Let $x,y\in \mfa$ be a regular sequence, and note that since $(x,y)\subset \mfa$, we also have that for all $n$, by \Cref{lem:length 2 regular sequence colon ideal},
\[cf^n\in (g)^n:\mfa^{\lceil tn\rceil}\subset (g^n):(x,y)^{\lceil tn\rceil} = (g^n).\]
Since this holds for all $n$, $f\in \overline{(g)}$, and so $\dfrac{f}{g}\in \overline{R}$.	
\end{proof}

The Hara-Yoshida tight closure has been used to define $F$-rationality of ideal pairs. See \cite{ST} for this discussion. In a similar vein, we can define $F$-nilpotence of pairs, recall \Cref{defn: F-nilpotent ring}.

\begin{defn}\label{defn: F-nilp}
Let $(R,\mfm)$ be an excellent, equidimensional, reduced local ring, $t \in \mathbb{R}_{> 0}$, and suppose $\mfa$ is an ideal of $R$ of positive height. We say the pair $(R,\mfa^t)$ is \textbf{$F$-nilpotent} if $\mfq^{*\mfa^t}=\mfq^F$ for all ideals $\mfq \subset R$ generated by a system of parameters for $R$. 
\end{defn}

Though the definition only includes the classical Frobenius closure and the Hara-Yoshida tight closure, in our study of $F$-nilpotence of pairs, we will need to factor through our other new closures as well. We now demonstrate that $F$-nilpotence of pairs is analogous to $F$-rationality of pairs, compare the following result with \cite[Prop.~6.2]{ST}.

\begin{prop}\label{prop: elementary properties of a^t F-nilpotent}
Let $(R,\mfm)$ be an excellent, equidimensional, reduced local ring, $t \in \mathbb{R}_{> 0}$, and $\mfa\subset \mfb $ be ideals of $R$ such that $\mfa$ is of positive height. Then, we have the following.
\begin{enumerate}[{\rm(a)}]
\item If $(R,\mfa^t)$ is $F$-nilpotent, then $(R,\mfb^s)$ is $F$-nilpotent for all $0 < s \le t$. 
\item If $\mfa$ is a reduction of $\mfb$, then $(R,\mfa^t)$ is $F$-nilpotent if and only if $(R,\mfb^t)$ is $F$-nilpotent.
\item If $(R,\mfa^t)$ is $F$-nilpotent, then $R$ is $F$-nilpotent.
\end{enumerate}
\end{prop}

\begin{proof}
It is evident that the Hara-Yoshida $\mfa^t$-tight closure enjoys a similar set of monotonicity properties as the Hara-Yoshida $\mfa^t$-integral closure demonstrated in \Cref{prop: basic properties of integral closures}, part (c). Hence, part (a) of this proposition is immediate. Similarly, the first half of part (b) is immediate, and the second half follows from \cite[Prop.~1.3]{HY}. Finally, (c) follows from the fact that $\mfq^F\subset \mfq^* \subset \mfq^{*\mfa^t}$, so if $\mfq^{*\mfa^t} = \mfq^F$ for all parameter ideals $\mfq$, then $\mfq^*=\mfq^F$ for all parameter ideals $\mfq$, so $R$ is $F$-nilpotent by \cite[Thm.~A]{PQ}.
\end{proof}

We now demonstrate that the condition defining $\mfa^t$ $F$-nilpotence descends from full systems of parameters to partial systems of parameters. Crucially, the proof uses that $F$-nilpotent rings have a uniform Frobenius test exponent for parameter ideals,  shown by Quy in \cite{Q}. 

\begin{lemma}\label{lem: a^t F-nilpotent works for partial sops too}
Suppose $(R,\mfm)$ is an excellent local domain and $(R,\mfa^t)$ is $F$-nilpotent. Then, for an ideal $I$ generated by a partial system of parameters, we also have $I^{*\mfa^t} = I^F$. In particular, for each nonzero $x\in R$, we have $(x)^{*\mfa^t} = \overline{(x)}^{\mfa^t}=(x)^F$.
\end{lemma}

\begin{proof}
Since $(R,\mfa^t)$ is $F$-nilpotent, by \Cref{prop: elementary properties of a^t F-nilpotent} we have $R$ itself is $F$-nilpotent, so by \cite[Main~Thm.]{Q} we have that $R$ has a finite Frobenius test exponent for parameter ideals, that is, there is an $e_0 \in \mbn$ such that for all ideals $\mfq$ generated by a full system of parameters, we have $(\mfq^F)^{\fbp{p^{e_0}}} = \mfq^{\fbp{p^{e_0}}}$. 

Write $d=\dim R$ and let $x_1,\ldots,x_t$ be a partial system of parameters and $I=(x_1,\ldots,x_t)$. Complete to $x_1,\ldots,x_t,y_{t+1},\ldots,y_d$ a full system of parameters for $R$, and write \[\mfq_n = (x_1,\ldots,x_t,y_{t+1}^n,\ldots,y_d^n),\] notably generated by a full system of parameters. We then have $I^F \subset \mfq_n^F$ for all $n$. Conversely if $x \in \mfq_n^F$ for all $n$ we have \[x^{p^{e_0}} \in \left(x_1^{p^{e_0}},\ldots,x_t^{p^{e_0}},(y_{t+1}^{p^{e_0}})^n,\ldots,(y_d^{p^{e_0}})^n\right) \] for all $n$. After intersecting over all $n$, we then obtain \[ x^{p^{e_0}} \in \bigcap_n (x_1^{p^{e_0}},\ldots,x_t^{p^{e_0}},(y_{t+1}^{p^{e_0}})^n,\ldots,(y_d^{p^{e_0}})^n) = I^{\fbp{p^{e_0}}}\] by Krull Intersection. Hence, $x \in I^F$. Now, then $I^{*\mfa^t} \subset \mfq_n^{*\mfa^t}$ for all $n$, so $I^{*\mfa^t} \subset \cap_n \mfq_n^{*\mfa^t} = \cap_n \mfq_n^F = I^F$. But, $I^F \subset I^{*\mfa^t}$ holds in general, so we have $I^F = I^{*\mfa^t}$.
\end{proof}

\begin{corollary}
Suppose $(R,\mfa^t)$ is $F$-nilpotent. Then, $\overline{R}^{\mfa^t}=\,{}^*R$. In particular, $\overline{R}^{\mfa^t}$ is a subring of $\Frac R$, and is local when $R$ is.
\end{corollary}

\begin{proof}
Given the preceding theorem, the only part of the corollary left unproven is the claim that $\overline{R}^{\mfa^t}$ is local. However, this is true since $R\rightarrow \overline{R}^{\mfa^t}={}^*R$ is a purely inseparable extension of rings, and hence induces a bijection on spectra. 
\end{proof}

In \cite{PQ}, Polstra and Quy proved the cohomological definition of $F$-nilpotence given by Srinivas and Takagi is equivalent to a definition in terms of tight and Frobenius closure parameter ideals. We now demonstrate that a similar equivalence follows in the context of our definition for $F$-nilpotence of pairs. First, we recall from Hara-Yoshida \cite[Definition~1.1]{HY} the definition of the $\mfa^t$-tight closure of submodules.

\begin{defn}
Let $R$ be a ring of prime characteristic $p>0$, $\mfa\subset R$ an ideal of positive height, and $t\in \mbr_{> 0}$. Then, for a submodule $N$ of an $R$-module $M$, the \textbf{$\mfa^t$-tight closure of $N$ in $M$} is the submodule $N^{*\mfa^t}_M$ defined by \[
N^{*\mfa^t}_M = \left\lbrace m \in M \mid \exists c \in R^\circ \text{ such that for all } e \gg 0, \, c\mfa^{\lceil tp^e\rceil }F^e_*(m) \subset F^e_*(N) \right\rbrace,
\] where $F^e_*(N)$ is the usual restriction of scalars of $N$ along $F^e:R\rightarrow R$.
\end{defn}

Hara-Yoshida investigated $0^{*\mfa}_{H^d_\mfm(R)}$ for a local ring $(R,\mfm)$ of dimension $d$, and demonstrated in \cite[Discussion~1.14]{HY} that its annihilator is equal to the $\mfa$-tight closure test ideal in the Gorenstein case. Further, they also related it to the natural Frobenius action on local cohomology in \cite[Proposition~1.15]{HY}, which we record below. The referenced proposition does not include the real parameter $t$, but the proof easily extends to this case.

\begin{prop}[\cite{HY}]
Let $(R,\mfm)$ be a $d$-dimensional excellent normal local ring of prime characteristic $p>0$ and let $\mfa \subset R$ be an ideal of positive height. Then, $0^{*\mfa^t}_{H^d_\mfm(R)}$ is the unique maximal proper submodule $N$ of $H^d_\mfm(R)$ with respect to the property that $\mfa^{\lceil tp^e \rceil} F^e(N) \subset N$ for all $e \gg 0$, where $F: H^d_\mfm(R)\rightarrow H^d_\mfm(R)$ is the canonical Frobenius action on local cohomology.
\end{prop}

Now we can connect the $F$-nilpotence of pairs property to the local cohomology modules of $R$.

\begin{theorem}\label{thm: CM a^t F-nilpotent is characterized by LC}
Let $(R,\mfm)$ be an excellent, Cohen-Macaulay local domain of prime characteristic $p>0$. Then, $(R,\mfa^t)$ is $F$-nilpotent if and only if $0^{*\mfa^t}_{H^d_\mfm(R)} = 0^F_{H^d_\mfm(R)}$.
\end{theorem}

\begin{proof}

Fix a parameter ideal $\mfq = (x_1,\ldots,x_d)$ and writing $\mfq_n=(x_1^n,\ldots,x_d^n)$, we will show that \[
0^{*\mfa^t}_{H^d_\mfm(R)} \cong \varinjlim \dfrac{\mfq^{*\mfa^t}_n}{\mfq_n}
\] where the direct limit maps are multiplication by $x=x_1\cdots x_d$. First, let $[z+\mfq_m] \in \varinjlim \mfq_n^{*\mfa^t}/\mfq_n$. Then, for all $e\ge 0$ and each $v_e \in \mfa^{\lceil tp^e\rceil}$, we have $v_e F^e[z+\mfq_m] = [v_ez^{p^e}+\mfq_{mp^e}] = 0$ since $z\in \mfq_m^{*\mfa^t}$. Conversely, if $[z+\mfq_m] \in 0^{*\mfa^t}_{H^d_\mfm(R)}$, then for all $e \ge 0$ and each $v_e \in \mfa^{\lceil tp^e \rceil}$, we have have as before $[v_ez^{p^e}+\mfq_{mp^e}] = 0$. But since $R$ is Cohen-Macaulay, the direct limit system is injective, and hence $v_ez^{p^e} \in \mfq_{mp^e}$, which implies $z \in \mfq_m^{*\mfa^t}$. 

Now suppose $(R,\mfa^t)$ is $F$-nilpotent. Then, $\mfq^{*\mfa^t} = \mfq^F$ for all parameter ideal $\mfq$, and consequently, $\varinjlim \mfq^{*\mfa^t}_n/\mfq_n = \varinjlim \mfq^F_n/\mfq_n$, and now we use that $\varinjlim \mfq^F_n /\mfq_n \cong 0^F_{H^d_\mfm(R)}$, see \cite[Remark~2.6.3]{PQ}. For the converse, we again use that the direct limit system is injective, so that $0^{*\mfa^t}_{H^d_\mfm(R)} = 0^F_{H^d_\mfm(R)}$ implies $\varinjlim \mfq^{*\mfa^t}_n/\mfq_n = \varinjlim \mfq^F_n/\mfq_n$ implies $\mfq^{*\mfa^t} = \mfq^F$ for all parameter ideals $\mfq$ (as $\mfq$ was arbitrary), thus $(R,\mfa^t)$ is $F$-nilpotent.
\end{proof}

It seems likely that, in the absence of the Cohen-Macaulay hypothesis, the theorem should still hold, so long as we impose that the lower local cohomology modules are also nilpotent under $F$, i.e. that the ring is \textit{weakly $F$-nilpotent}. 

\section{Adjusted ideal closures in semigroup rings}  \label{sec:semigroup}

We now consider the class of rings $k[S]$, where $k$ is a field  and $S$ is a submonoid of the free abelian group $M=\Z^d$ of rank $d$.  We have that $M$ naturally sits inside $V=M\otimes_\Z \R\cong \R^d$ and we can utilize convex geometry in $\R^d$ to analyze monomial ideals in $k[S]$. We will briefly describe the objects we need; we refer the reader to books such as \cite{Fulton} or \cite{CLS} for further background on affine semigroup rings. 

\begin{setting} \label{setting - semigroup ring} Fix a field $k$. Let $M$ be a rank $d$ free abelian group, $S$ a finitely generated submonoid of $M$ such that $S$ generates $M$ as a group, and $V=M\otimes_\Z\R$.  Let $\sigma^\vee = \mathrm{cone}(S)\subset V$, and $\sigma=\mathrm{cone}(\vec c_1,\ldots, \vec c_n)\subset V^*$ the dual cone to $\sigma^\vee$.  That is, $\vec v\in \sigma^\vee$ if and only if $\vec v\in V$ and $\langle \vec c_i,\vec v\rangle \geq 0$ for all $i=1,\ldots, n$.  Let $\overline{S}=\sigma^\vee\cap M$.  Thus, $\overline{S}$ is the saturation of $S$ inside $M$, and $\overline{R} = k[\overline{S}]$ is the integral closure of $R=k[S]$ inside $k[M]$. Given a monomial ideal $I\subset k[M]$, we will write $\Exp I\subset M$ for the set of exponents appearing in monomials in $I$, that is, $\Exp I = \{\vec v\in S\mid x^{\vec v}\in I\}$. Note that for any ideal $J$ of $R$, $(\Exp(J)+\sigma^\vee)\cap M = \Exp(J\overline{R})$.
	
	Let $I$ and $\mfa$ be monomial ideals of $R$, so that there exist finite sets $\mathcal{A}, \mathcal{B}\subset S$ such that $\mfa = (x^{\vec a} \mid \vec a\in \mathcal{A})$ and $I=(x^{\vec b}\mid \vec b\in \mathcal{B})$. Let $A=(\vec a_1|\cdots |\vec a_r)$, and let $\Delta = \{\vec u\in\R_{\geq 0}^r\mid |\vec u|_1=1\}$, where $|\vec u|_1$ is the sum of the components of $\vec u$.  Let $H=\mathrm{Hull}(\vec a_1,\ldots,\vec a_r)$ be the convex hull of the columns of $A$, and notice that $H=\{A\vec u\mid \vec u\in \Delta\}$.
\end{setting}

\begin{lemma}\label{lem: working with hulls}
	Given \Cref{setting - semigroup ring}, for each $n \in \mathbb{Z}_{\geq 0}$ and $s\in \mathbb{R}_{>0}$, we have 
	\begin{enumerate}[{\rm(a)}]
		\item $\Exp(\mfa^n)\subset nH+\sigma^\vee$, and
		\item $sH\subset \Exp(\mfa^{\lfloor s-d\rfloor})+\sigma^\vee$ and $sH\cap \overline{S} \subseteq \Exp(\mfa^{\lfloor s-d\rfloor}\overline{R})$.
	\end{enumerate}
\end{lemma}

\begin{proof}
	\begin{enumerate}[(a)]
		\item The ideal $\mfa^n$ is generated by monomials of the form $x^{A\vec w}$, where $\vec w\in \N^r$ and $|\vec w|_1 = n$.  Letting $\vec u = \frac{1}{n}\vec w$, we have that $\vec u\in \Delta$ and so $A\vec w=nA\vec u\in nH$.  Therefore $\Exp(\mfa^n)\subset nH+\sigma^\vee.$
		
		\item On the other hand, if $\vec v\in sH$ and $s > 0$, then there exists $\vec u\in \Delta$ such that $\vec v = sA\vec u$.  Furthermore, by Carath\'eodory's theorem, we may assume that at most $d+1$ of the components of $\vec u$ are nonzero.  Let $\vec w\in\N^r$ be given by $w_i=\lfloor su_i\rfloor$.  Thus each component of $s\vec u - \vec w$ is nonnegative, and so $A(s\vec u - \vec w)$ is a nonnegative linear combinations of the vectors $\vec a_i$.  Since each $\vec a_i$ is in $\sigma^\vee$, so is $A(s\vec u - \vec w)$.  Hence
		\[\vec v = \vec v - A\vec w +A\vec w=  A(s\vec u-\vec w)+A\vec w\in A\vec w+\sigma^\vee.\]
		Since $|\vec w|_1 = \sum_i\lfloor su_i\rfloor>\sum_{i:u_i\neq 0}(su_i-1) \geq s-d-1$, we have that $|\vec w|_1\geq \lfloor s-d\rfloor$ and so $A\vec w\in\Exp(\mfa^{\lfloor s-d\rfloor})$.  Thus, we have that $sH\subset \Exp(\mfa^{\lfloor s-d\rfloor})+\sigma^\vee$, and so $sH\cap \overline{S} \subset \Exp(\mfa^{\lfloor s-d\rfloor}\overline{R})$.
	\end{enumerate}
\end{proof}

\subsection{Joint Hilbert-Kunz multiplicity} In \cite{Vraciu}, the author defined the joint Hilbert-Kunz multiplicity, $e_{HK}(\mfa^t;I),$ of two finite colength ideals $I, \mfa$ in a local ring $(R, \mfm)$ of dimension $d \geq 1.$ Letting $t \in \mathbb{R}_{> 0}$ be a fixed real number,   
\[ e_{HK}(\mfa^t; I) = \lim\limits_{e \rightarrow \infty} \frac{1}{p^{ed}} \lambda_R \left(\frac{R}{\mfa^{\lceil tp^e \rceil}I^{[p^e]}}\right). \]
It is known that the Hilbert-Kunz multiplicity of finite colength ideals in a semigroup ring can be computed as the volume of a finite union of rational polytopes due to work of Watanabe \cite{Watanabe} and Eto \cite{eto}. The following result proves that the same holds for the joint Hilbert-Kunz multiplicity. 

\begin{theorem}\label{joint HK volume}
	Given \Cref{setting - semigroup ring}, suppose $k$ is of prime characteristic $p>0$, and let $I$ and $\mfa$ be ideals in $R$ of finite colength, and let $t\in \R_{>0}$. If we set \[ \mathcal{P}_t = \sigma^\vee \setminus (\mathcal{B} + tH + \sigma^\vee),\] 
	then $e_{HK}(\mfa^t;I) = \operatorname{vol} \ \overline{\mathcal{P}_t}$, where $\overline{\mathcal{P}_t}$ is the closure of $\mathcal{P}_t$ in the Euclidean topology on $\mathbb{R}^d$ and $\operatorname{vol}$ denotes the usual Euclidean volume on $\mathbb{R}^d$.
\end{theorem}

\begin{proof}
	First note that it is enough to prove the statement when $R$ is a normal ring. Indeed, if $R$ is not normal, then observe that $\overline{R}$ is a finite $R$-module and taking $W = R \setminus \{0\},$ we have $W^{-1}R = W^{-1}\overline{R} = \Frac(R)$ implying that $e_{HK}(\mfa^t;I) = e_{HK}(\mfa^t \overline{R}; I \overline{R})$ (see \cite[Lemma~3.3]{Vraciu}).
	
	Let $g_1,\ldots,g_n$ be the generators of $\sigma^\vee$ and let $F_1,\ldots,F_\alpha$ be the cones divided by all the hyperplanes spanned by $g_1,\ldots,g_n$ and $\vec b+t\vec a_i$, where $b \in \mathcal{B}$ and $1 \leq i \leq r.$ Then for each $\vec b \in \mathcal{B},$ the closure of the complement of $(\vec b + tH + \sigma^\vee)$ in $\sigma^\vee$ is a finite union of convex sets $\overline{F_i \cap (\vec b + tH + \sigma^\vee)^c}.$ Hence,
	\begin{align*} 
		\mathcal{P}_t = \sigma^\vee \bigcap \left( \bigcap_{\vec b \in \mathcal{B}} \overline{\big(\vec b + tH + \sigma^\vee \big)^c} \right) 
		= \bigcap_{\vec b \in \mathcal{B}} \left( \bigcup_{i=1}^{\alpha}  \overline{F_i \cap \big( \vec b + tH + \sigma^\vee \big)^c} \right)
	\end{align*}
	is a finite union of rational polytopes, say $P_{t,1}, \ldots, P_{t,\beta},$ containing the origin such that for all $i \neq j,$ $P_{t,i} \cap P_{t,j}$ is a rational polytope of dimension $< d.$ Using the same arguments as in the proof of \cite[Theorem~2.2]{eto}, we get
	\begin{align} \label{vol}
		\operatorname{vol} \overline{\mathcal{P}_t} = \sum_{i=1}^{\beta} \operatorname{vol} P_{t,i} = \sum_{i=1}^{\beta} \lim\limits_{p^e \rightarrow \infty} \frac{\mid p^eP_{t,i} \cap \mathbb{Z}^d \mid}{p^{ed}} 
		= \lim\limits_{e \rightarrow \infty} \frac{\mid p^e\mathcal{P}_t \cap \mathbb{Z}^d \mid}{p^{ed}}.
	\end{align}

	In order to compute the joint Hilbert-Kunz multiplicity, we observe that if $x^{\vec v}\in \mfa^{\lceil tp^e\rceil}I^{[p^e]}$, then $\vec v\in A\vec w + p^e\mathcal{B}+\sigma^\vee$, where $\vec w\in\N^r$ and $|\vec w|_1 = \lceil tp^e\rceil$, so we have
	\[ \vec v\in \lceil tp^e\rceil H  + p^e\mathcal{B}+\sigma^\vee \subset p^e(tH+\mathcal{B}+\sigma^\vee) = \sigma^\vee\setminus p^e\mathcal{P}_t. \]  
	Therefore, if $\vec v\in p^e\mathcal{P}_t$, then $x^{\vec v}\notin\mfa^{\lceil tp^e\rceil}I^{[p^e]}$, and so $\lambda_R(R/\mfa^{\lceil tp^e \rceil} I^{[p^e]})\geq |p^e\mathcal{P}_t\cap\mathbb{Z}^d|$. 
	
	On the other hand, suppose that $s>t$, and $p^e$ is large enough that $\lfloor sp^e-d\rfloor\geq \lceil tp^e\rceil$.  If $\vec v\in S\setminus p^e\mathcal{P}_s$, then 
	\begin{align*}
		\vec v\in p^e\mathcal{B}+sp^eH+\sigma^\vee
		\subset&  \Exp(I^{[p^e]})+\Exp(\mfa^{\lfloor sp^e-d\rfloor})+\sigma^\vee \\
		=& \Exp(\mfa^{\lfloor sp^e-d\rfloor}I^{[p^e]})+\sigma^\vee 
		\subset \Exp(\mfa^{\lceil tp^e\rceil}I^{[p^e]})+\sigma^\vee.
	\end{align*}
	Therefore $x^{\vec v}\in \mfa^{\lceil tp^e\rceil}I^{[p^e]}$.  Hence $\lambda_R(R/\mfa^{\lceil tp^e \rceil} I^{[p^e]}) \leq |p^e\mathcal{P}_s\cap\mathbb{Z}^d|$.
	
	Thus, we have that, for any $s>t$ and $e\gg 0$, 
	\[|p^e\mathcal{P}_t\cap\mathbb{Z}^d|\leq \lambda_R(R/\mfa^{\lceil tp^e \rceil} I^{[p^e]})\leq  |p^e\mathcal{P}_s\cap\mathbb{Z}^d|\]
	and hence, for any $s > t$, 
	\[\lim_{e\to\infty} \frac{|p^e\mathcal{P}_t\cap\mathbb{Z}^d|}{p^{ed}}\leq e_{HK}(\mfa^t;I)\leq \lim_{e\to\infty} \frac{|p^e\mathcal{P}_s\cap\mathbb{Z}^d|}{p^{ed}}. \]
	From \eqref{vol}, we get $\operatorname{vol} \overline{\mathcal{P}_t}\leq e_{HK}(\mfa^t;I)\leq \lim_{s\to t^+} \operatorname{vol} \overline{\mathcal{P}_s} = \operatorname{vol} \overline{\mathcal{P}_t}.$
\end{proof}

\subsection{The adjusted closures in affine semigroup rings}

In this subsection we show that the calculation of the adjusted closure operations in affine semigroup rings are computable in terms of the convex geometry of the associated exponent sets. This is well-known in the setting of usual integral closure, see for example \cite[Prop.~1.4.6]{hunekeSwanson}.

\begin{theorem} \label{theorem - closures to convex geoemtry} Given \Cref{setting - semigroup ring}, suppose $k$ is of prime characteristic $p>0$, and let $t\in\R_{>0}$ and $x^{\vec v}\in R$ be a monomial.
	\begin{enumerate}[{\rm(a)}]
		\item \label{HY semigroup condition} If $x^{\vec v} \in I^{*\mfa^t}$, in particular if $x^{\vec v}\in \ifat$, then $\vec v+tH\subset \mathcal{B}+\sigma^\vee$.
		\item \label{Vraciu tight closure semigroup condition} If $x^{\vec v}\in {{}^{\mfa^t}I^*}$, in particular if $x^{\vec v}\in \atif$ then $\vec v+tH\subset\mathcal{B}+tH+\sigma^\vee$.
	\end{enumerate}
\end{theorem}

\begin{proof} 
	Suppose that there exists $x^{\vec c}\in R$, $s\geq 0$ such that for $p^e\gg 0$, $x^{\vec c}x^{p^e\vec v}\mfa^{\lceil tp^e\rceil}\subset I^{[p^e]}\mfa^{\lceil sp^e\rceil}$.  That is, 
	\[\vec c + p^e\vec v + \Exp(\mfa^{\lceil tp^e\rceil})\subset p^e\mathcal{B}+\Exp(\mfa^{\lceil sp^e\rceil})+\sigma^\vee.
	\]
	This implies that
	\begin{align*}\vec c + p^e\vec v + \lceil tp^e+d\rceil H &\subset 
		\vec c + p^e\vec v + \Exp(\mfa^{\lceil tp^e\rceil}) + \sigma^\vee \\
		&\subset p^e\mathcal{B}+\Exp(\mfa^{\lceil sp^e\rceil})+\sigma^\vee\\
		&\subset p^e\mathcal{B}+\lceil sp^e\rceil H+\sigma^\vee.
	\end{align*}
	Since $H\subset \sigma^\vee$, $\lceil sp^e\rceil H+\sigma^\vee\subset sp^e H +\sigma^\vee$.  Therefore, dividing through by $p^e$, we have that for $e\gg 0$,
	\[\frac{1}{p^e} \vec c + \vec v + \frac{\lceil tp^e+d\rceil}{p^e}H \subset \mathcal{B} + sH +\sigma^\vee.
	\]
	Now, $\mathcal{B}$ and $sH$ are each closed and bounded, hence compact, and therefore so is their sum $\mathcal{B} + sH$.  Therefore, since $\sigma^\vee$ is closed, so is $\mathcal{B}+sH+\sigma^\vee$.  Hence for any $\vec h\in H$, 
	\[\vec v + t\vec h = \lim_{e\to \infty}\left(\frac{1}{p^e}\vec c + \vec v + \frac{\lceil tp^e+d\rceil}{p^e}\vec h\right)\in \mathcal{B}+sH+\sigma^\vee.
	\]
	Therefore $\vec v+tH\subset \mathcal{B}+sH+\sigma^\vee$.  
	
	Taking $s=0$ shows that if $x^{\vec v}\in I^{*\mfa^t}$, then $\vec v+tH\subset \mathcal{B}+\sigma^\vee$, and the same holds if  $x^{\vec v}\in \ifat$, since $\ifat \subset I^{*\mfa^t}$.  Taking $s=t$ shows that if $x^{\vec v}\in {{}^{\mfa^t}I^*}$, then $\vec v+tH\subset \mathcal{B}+tH+\sigma^\vee$, and the same holds if $x\in \atif$, since $\atif\subset {{}^{\mfa^t}I^*}$.
\end{proof}

\begin{theorem}\label{theorem - convex geoemtry to closures}  Given \Cref{setting - semigroup ring}, suppose $k$ is of prime characteristic $p>0$, and let $t \in \mathbb{R}_{> 0}$ and $x^{\vec v}\in R$ be a monomial.
	\begin{enumerate}[{\rm(a)}]
		\item If $\vec v+tH\subset \mathcal{B}+\sigma^\vee$, then $x^{\vec v}\in (I\overline{R})^{F\mfa^t}$ and $x^{\vec v}\in (I\overline{R})^{*\mfa^t}$.
		\item If $\vec v + tH\subset \mathcal{B}+tH+\sigma^\vee$, then $x^{\vec v}\in {{}^{\mfa^t}(I\overline{R})^*}$.
		\item If $\vec v + tH\subset \mathcal{B}+sH+\sigma^\vee$ for some $s>t$, then $x^{\vec v}\in {}^{\mfa^t}(I\overline{R})^F$.
	\end{enumerate}
	
\end{theorem}
\begin{proof}
	Let $s\geq 0$ be a real number, and suppose that $\vec v+tH\subset\mathcal{B}+sH+\sigma^\vee$.  Now, for any $e\geq 0$, we have that 
	\[p^e\vec v+\Exp(\mfa^{\lceil tp^e\rceil})\subset p^e\vec v+tp^eH+\sigma^\vee \subset p^e\mathcal{B}+sp^eH+\sigma^\vee \subset p^e\mathcal{B}+\Exp(\mfa^{\lfloor sp^e-d\rfloor})+\sigma^\vee.
	\]
	Note that $(p^e\mathcal{B}+\Exp(\mfa^{\lfloor sp^e-d\rfloor})+\sigma^\vee)\cap \overline{S} = \Exp(I^{[p^e]}\mfa^{\lfloor sp^e-d\rfloor}\overline{R})$.
	
	Suppose $s=0$, so that we have $p^e\vec v+\Exp(\mfa^{\lceil tp^e\rceil})\subset \Exp(I^{[p^e]}\overline{R})$, and so $x^{p^e\vec v}\mfa^{\lceil tp^e\rceil}\subset I^{[p^e]}\overline{R}$.  Hence, if $\vec v + tH\subset \mathcal{B}+\sigma^\vee$, then $x^{\vec v}\in (I\overline{R})^{F\mfa^t}\subset (I\overline{R})^{*\mfa^t}$.
	
	Suppose $s=t$, and choose $x^{\vec c}\in \mfa^{d+1}$.  Now, for $e\geq 0$, we have that 
	\[\vec c + p^e\vec v+\Exp(\mfa^{\lceil tp^e\rceil})\subset \Exp(x^{\vec c}I^{[p^e]}\mfa^{\lfloor tp^e-d\rfloor}) \subset \Exp(I^{[p^e]}\mfa^{\lfloor tp^e+1\rfloor})\subset \Exp(I^{[p^e]}\mfa^{\lceil tp^e\rceil}).
	\]
	Therefore $x^{\vec c}x^{p^e\vec v}\mfa^{\lceil tp^e\rceil}\subset I^{[p^e]}\mfa^{\lceil tp^e\rceil}$ for all $e\geq 0$.  Hence, if $\vec v + tH\subset \mathcal{B}+tH+\sigma^\vee$, then $x^{\vec v}\in {{}^{\mfa^t}(I\overline{R})^*}$.
	
	Suppose $s>t$, so that for $e\gg0$, we have that $sp^e-d \geq \lceil tp^e\rceil$, and so $ \Exp(I^{[p^e]}\mfa^{\lfloor sp^e-d \rfloor}\overline{R})\subset \Exp(I^{[p^e]}\mfa^{\lceil tp^e\rceil}\overline{R})$.  Therefore, \[p^e\vec v+\Exp(\mfa^{\lceil tp^e\rceil})\subset\Exp(I^{[p^e]}\mfa^{\lfloor sp^e-d \rfloor}\overline{R})\subset \Exp(I^{[p^e]}\mfa^{\lceil tp^e\rceil}\overline{R}).\]  Hence, if $\vec v + tH\subset \mathcal{B}+sH+\sigma^\vee$ for some $s>t$, then $x^{\vec v}\in {}^{\mfa^t}(I\overline{R})^F$.
\end{proof}

\begin{corollary}\label{corollary - normal semigroup} Given \Cref{setting - semigroup ring}, if $k$ is of prime characteristic $p>0$ and $R=k[S]$ is a normal semigroup ring, for a monomial $x^{\vec v}\in R$, we have
	\begin{enumerate}[{\rm(a)}]
		\item $x^{\vec v}\in \ifat$ if and only if $x^{\vec v} \in I^{*\mfa^t}$ if and only if $\vec v + tH\subset \mathcal{B}+\sigma^\vee$,
		\item $x^{\vec v}\in {{}^{\mfa^t}I^*}$ if and only if $\vec v + tH\subset \mathcal{B}+tH+\sigma^\vee$,
		\item \label{item - atif implies v+tH} if $x^{\vec v}\in \atif$ then $\vec v + tH\subset \mathcal{B}+tH+\sigma^\vee$, and 
		\item if there exists $s>t$ such that $\vec v + tH\subset \mathcal{B}+sH+\sigma^\vee$, then $x^{\vec v}\in \atif$.
	\end{enumerate}
	
\end{corollary}

\begin{remark}  The converse of \Cref{corollary - normal semigroup}(c) does not hold in general, that is, the Vraciu $\mfa^t$-tight and Frobenius closures may differ.  For example, let $k$ be a field of odd prime characteristic $p>0$, $R=k[x,y]$, $I=\mfa=(x^2,y^2)$, and $t=1$, so $tH=\mathrm{Hull}((2,0),(0,2))$.  Then, $(1,1)+tH  = \mathrm{Hull}((3,1),(1,3))$, and $\{(2,0),(0,2)\} + tH = \mathrm{Hull}((4,0),(0,4))$.  Therefore $(1,1)+tH\subset \{(2,0),(0,2)\} + tH$, demonstrating $xy\in \atistar$.  However, $xy\notin \atif$; for any $e\geq 0$, we have that $x^{3p^e}y^{p^e}\in (xy)^{p^e}\mfa^{p^e}$.  But, this monomial is not in $I^{[p^e]}\mfa^{p^e}$, since all generators of the latter are of degree $4p^e$ but have only even exponents.
	
	Also, the normality condition in \Cref{corollary - normal semigroup} is necessary for the implications coming from \Cref{theorem - convex geoemtry to closures}.  To see this, let $R=k[x^2,x^3,y,xy]$.  Note that $R$ contains all monomials in $k[x,y]$ except $x$.  Now, let $I=(x^2y)$, $\mfa = (x^2,x^3)$, and $t=1$.  This makes $H=\mathrm{Hull}((2,0),(3,0))$, and so 
	\[(0,1)+tH = \mathrm{Hull}((2,1),(3,1))\subset (2,1)+\sigma^\vee = \Exp(I)+\sigma^\vee.\]
	Therefore, if \Cref{corollary - normal semigroup} were to hold without the normality assumption, we would have that $y\in \ifat$.  However, for any $e\geq 0$ we have that
	\[y^{p^e}\mfa^{p^e} = (x^{2a+3b}y^{p^e}\mid a+b=p^e) = (x^{2p^e}y^{p^e},x^{2p^e+1}y^{p^e})
	\]
	but $I^{[p^e]} = (x^{2p^e}y^{p^e})$.
	Since $x\notin R$, $x^{2p^e+1}y^{p^e}\notin I^{[p^e]}$.  Therefore $y\notin \ifat$.
\end{remark}

Independent of the characteristic of $k$, we also get similar convex geometry statements for the $\mfa^t$-integral closure.
These statements can be seen as restatements of \Cref{valuative criterion for integral closure:label}, as the defining equations of the convex hull of the exponent set of $I$ exactly correspond to the Rees valuations of $I$ (see \cite[Theorem~A]{BDHM}).

\begin{theorem}\label{thm: at integral closure convex geometry}
	Given \Cref{setting - semigroup ring}, if $x^{\vec v} \in \overline{I}^{\mfa^t}$ then $\vec v+tH \subset \operatorname{Hull} (\operatorname{Exp} I)$, and if $k[S]$ is normal then the converse holds.
\end{theorem}

\begin{proof}
	The proof is similar to \Cref{theorem - convex geoemtry to closures}, but we provide it for completeness. For simplicity, write $E = \operatorname{Hull}(\operatorname{Exp} I)$. 
	
	First, suppose $x^{\vec v} \in \overline{I}^{\mfa^t}$. So, there is a $\vec c \in S$ with $x^{\vec c} x^{n\vec v} \mfa^{\lceil tn \rceil} \subset I^n$ for all $n \gg 0$. In vector form, this is equivalent to $\vec c + n \vec v + \operatorname{Exp}(\mfa^{\lceil nt \rceil}) \subset \operatorname{Exp}(I^n)$. By \Cref{lem: working with hulls}, we have 
	\begin{align*}
		\vec c + n \vec v + \lceil tn+d\rceil H &\subset  \vec c + n \vec v + \operatorname{Exp}(\mfa^{\lceil tn \rceil}) + \sigma^\vee \\
		&\subset  \operatorname{Exp}(I^n)+\sigma^\vee \\
		&\subset  nE + \sigma^\vee = nE \end{align*} since $nE + \sigma^\vee = nE$. Then, given $\vec c + n\vec v + \lceil tn+d \rceil H \subset nE$, we divide both sides by $n$, viewed as subsets of $V = M\otimes_\mathbb{Z} \mathbb{R}$ to get \[\dfrac{\vec c}{n} + \vec v + \dfrac{\lceil tn+d \rceil}{n} H \subset E.\] As $E$ is a closed subset of $\mathbb{R}^d$ in the usual Euclidean topology, we may take the limit in $n$ and stay inside $E$, that is, $\vec v+tH \subset E$. 
	
	Now, suppose $k[S]$ is normal and $\vec v + tH \subset E$, and we want to show $x^{\vec v} \in \overline{I}^{\mfa^t}$. For some $n \in \mathbb{N}$, let $\vec a \in \operatorname{Exp}(\mfa^{\lceil nt\rceil})$ be such that $x^{\vec a}$ is a nonzero monomial generator of $\mfa^{\lceil tn \rceil}$. Then, there is a $\vec w$ in $\mathbb{N}^r$ such that $|\vec w|_1 = \lceil tn \rceil$ and $\vec a = A\vec w$. Also, fix $\vec c \in \operatorname{Exp}(I^\ell)$, for $\ell = \mu(I)-1$. We then have \begin{align*}
		\vec c + n\vec v + \vec a = & n\left(\dfrac{\vec c}{n} + \vec v + \dfrac{\vec a}{n} \right)\\ 
		=& n\left(\dfrac{\vec c}{n} + \vec v + \dfrac{1}{n} A \vec w \right) \\
		=& n\left(\dfrac{\vec c}{n} + \vec v + \dfrac{t}{\lceil tn\rceil} A \vec w + \left( \dfrac{1}{n} - \dfrac{t}{\lceil tn \rceil}\right)A\vec w \right).
	\end{align*}
	
	Notably, the term $\vec v + (t/\lceil tn\rceil) A\vec w$ is inside $\vec v + tH \subset E$ and the term $(1/n - t/\lceil tn \rceil) A\vec w$ is in $\sigma^\vee$ since $t/\lceil tn \rceil < t/tn = 1/n$. Thus, 
	\[\vec c + n \vec v+ \vec a\in\vec (c + nE + \sigma^\vee)\cap S \subset \vec c + (nE\cap S)\subset \vec c + \operatorname{Exp}(\overline{I^n}).\]  Finally, by the Brian{\c c}on-Skoda theorem, $\overline{I^n} \subset I^{n-\ell}$ for all $n$. This shows $\vec c + \operatorname{Exp}(I^{n-\ell}) \subset \operatorname{Exp}(I^n)$ by our choice of $\vec c$, demonstrating that $x^{\vec v} \in \overline{I}^{\mfa^t}$, as required. 
\end{proof}

\begin{remark}\label{rmk: independent of characteristic}
We note that the $\mfa^t$-tight and integral closure calculations above are entirely independent of characteristic, so for a fixed affine semigroup and ideal, these closure operations cannot vary as the characteristic of the underlying field varies. However, we demonstrate characteristic dependence in the $\mfa^t$-Frobenius closures in the next example.
\end{remark}

In the example below, the statements about the adjusted integral closures hold over fields of arbitrary characteristic $k$, but the statements involving the joint Hilbert-Kunz multiplicity and the adjusted tight and Frobenius closures require that $k$ is of prime characteristic $p>0$.

\begin{xmp} \label{MegaExample}
	Let $R=k[x,xy,x^2y^3] = k[S]$, where $S=\sigma^\vee\cap \Z^2$ (notably, a normal semigroup) and $\sigma^\vee$ is the cone generated by $(1,0)$ and $(2,3)$, which is dual to the cone $\sigma$ generated by $\vec c_1 = (0,1)$ and $\vec c_2 = (3,-2)$. Let $I=(x^4, x^2y^3)$ and $\mfa = (x^2,x^2y^3)$, so that $\vec b_1=(4,0)$, $\vec b_2=(2,3)$, $\vec a_1 = (2,0)$, and $\vec a_2= (2,3)$. 
	
	The joint Hilbert-Kunz multiplicity as a function of $t$ is given below.
	\[e_{HK}(\mfa^t;I) = 
	\begin{cases}
		6t^2+12 & \text{ if } t < 1\\
		3t^2+6t+9 & \text{ if } t\geq 1
	\end{cases}\]
	The Hara-Yoshida $\mfa^t$-integral closures are given below.
	\begin{align} \label{meg:int}
		\overline{I}^{\mfa^t}= 
		\begin{cases}
			I_1:=\overline{I} = (x^4,x^4y,x^3y^2, x^2y^3) &\text{ if } t < \frac{1}{6}\\
			I_2:=I_1 + (x^3y) = (x^4,x^3y,x^3y^2, x^2y^3)  &\text{ if } \frac{1}{6}\leq t< \frac{1}{3}\\
			I_3:=I_2 + (x^2y^2) = (x^4, x^3y, x^2y^2, x^2y^3) & \text{ if } \frac{1}{3}\leq t< \frac{1}{2}\\
			I_4:= I_3 + (x^3) = (x^3, x^3y, x^2y^2, x^2y^3) & \text{ if }\frac{1}{2}\leq t < \frac{2}{3}\\
			I_5:= I_4 + (x^2y) = (x^3, x^2y, x^2y^2, x^2y^3)& \text{ if }\frac{2}{3}\leq t < 1\\
			I_6:=I_5+(x^2) = (x^2, x^2y, x^2y^2, x^2y^3)&\text{ if }1\leq t<\frac{7}{6}\\
			I_7:=I_6+(xy) = (x^2, xy, x^2y^3) & \text{ if }\frac{7}{6}\leq t < \frac{3}{2}\\
			I_8:=I_7+(x) = (x, xy, x^2y^3) & \text{ if }\frac{3}{2}\leq t < 2\\
			I_9:=I_8+(1) = (1) &\text{ if } t \geq 2
		\end{cases}
	\end{align}
	The Hara-Yoshida $\mfa^t$-tight and $\mfa^t$-Frobenius closures are given below.
	\begin{align}  \label{meg:HY}
		\istarat = \ifat= 
		\begin{cases}
			J_1:=I = (x^4, x^2y^3) &\text{ if } t < \frac{1}{2}\\
			J_2:=J_1 + (x^5y^2) = (x^4,x^5y^2, x^2y^3)  &\text{ if } \frac{1}{2}\leq t< 1\\
			J_3:=J_2 + (x^4y, x^4y^2) = (x^4, x^4y, x^4y^2, x^2y^3) & \text{ if } 1\leq t< \frac{3}{2}\\
			J_4:= J_3 + (x^3,x^3y,x^3y^2) = (x^3, x^3y, x^3y^2, x^2y^3) & \text{ if }\frac{3}{2}\leq t < 2\\
			J_5:= J_4 + (x^2,x^2y,x^2y^2) = (x^2, x^2y, x^2y^2, x^2y^3)& \text{ if }2\leq t < \frac{5}{2}\\
			J_6:=J_5+(x,xy) = (x, xy, x^2y^3)&\text{ if }\frac{5}{2}\leq t<3\\
			J_7:=J_6+(1) = (1) & \text{ if }t\geq 3
		\end{cases}
	\end{align}
	The Vraciu $\mfa^t$-tight closures are given below.
	\begin{align} \label{meg:Vtight}
		\atistar = 
		\begin{cases}
			K_1 := I = (x^4,x^2y^3) &\text{ if } t < \frac{1}{2}\\
			K_2 := K_1 + (x^5y^2) = (x^4,x^5y,x^2y^3) &\text{ if } \frac{1}{2}\leq t< 1\\
			K_3 := K_2 + (x^4y,x^4y^2) = (x^4,x^4y,x^4y^2,x^2y^3) &\text{ if } t\geq 1
		\end{cases}
	\end{align}
	Finally, the Vraciu $\mfa^t$-Frobenius closures are given below.
	\begin{align} \label{meg:VFrob}
		\atif = 
		\begin{cases}
			K_1 = I &\text{ if } t <\frac{1}{2}\\
			K_2  = K_1+(x^5y^2) =(x^4,x^5y^2, x^2y^3) & \text{ if } t\geq \frac{1}{2}, \text{ and } t<1 \text{ or } p\neq 3\\
			K_3  = K_2+(x^4y,x^4y^2) = (x^4,x^4y,x^4y^2,x^2y^3) & \text{ if } t\geq 1 \text{ and } p = 3 
		\end{cases}
	\end{align}
\end{xmp}

We defer the details of this computation to the appendix below.

\changelocaltocdepth{1}

\begin{appendices}
\section{Appendix: Calculating the closures in affine semigroup rings}

In this section we provide the calculations behind \Cref{MegaExample}, which demonstrate the ways in which these closures and multiplicity may be computed.  We use the data of the linear algebra of the cone $\sigma^\vee$ and the generators of $I$ and $\mfa$ to determine when a monomial $x^{\vec v}$ is in the various $\mfa^t$-closures of $I$. We first handle the joint Hilbert-Kunz multiplicity and the $\mfa^t$-integral closure, as those are simpler to calculate.

We continue to assume that we are in \Cref{setting - semigroup ring}, and in particular note that $R=k[S]$ where $S\subset \Z^d$ is a semigroup, $\sigma^\vee=\mathrm{cone}(S)$, and $\vec c_1,\ldots, \vec c_n\in\Z^d$ vectors such that $\vec v\in \sigma^\vee$ if and only if $\langle \vec c_i,\vec v\rangle \geq 0$ for all $i=1,\ldots, n$. 
Let  $C=(\vec c_1|\vec c_2|\cdots|\vec c_n)$ be the matrix whose columns are the vectors $\vec c_i$, which means that a vector $\vec v$ is in $\sigma^\vee$ if and only if $C^T\vec v\geq \vec 0$ componentwise. 

Let $\mathcal{A}=\{\vec a_1,\ldots,\vec a_r\}, \mathcal{B}=\{\vec b_1,\ldots, \vec b_n\}\subset S$, $\mfa = (x^{\vec a} \mid \vec a\in \mathcal{A})$, and $I=(x^{\vec b}\mid \vec b\in \mathcal{B})$. 
Let $A=(\vec a_1|\cdots |\vec a_r)$, and let $\Delta = \{\vec u\in\R_{\geq 0}^r\mid |\vec u|_1=1\}$, where $|\vec u|_1$ is the sum of (the absolute values of) the components of $\vec u$.  Let $H=\mathrm{Hull}(\vec a_1,\ldots,\vec a_r)$ be the convex hull of the columns of $A$, and notice that $H=\{A\vec u\mid \vec u\in \Delta\}$.

	\subsection{The joint Hilbert-Kunz multiplicity}
	
	By \Cref{joint HK volume}, we can compute $e_{HK}(\mfa^t;I)$ by finding the area of a certain region in the plane.  In the diagrams below, the cone $\sigma^\vee$ is bounded by the two rays meeting at the origin; the  vertical line segment joining $\vec a_1$ and $\vec a_2$ is $H$, the convex hull of the exponent vectors of the generators of $\mfa$, and the other vertical line segments are the sets $\vec b_1 + tH$ and $\vec b_2+tH$, for $t<1$ on the left, and for $t\geq 1$ on the right.  The shaded regions are the sets $\mathcal{B}+tH+\sigma^\vee$.   The joint Hilbert-Kunz multiplicity is, therefore, the area of the region inside the cone but outside the shaded region, which can be readily computed from the diagrams.
	
	\def\dotwidth{0.05}
	\def\xmax{8}
	\def\ymax{7}

	\def\tval{0.7}
	
	\begin{tikzpicture}[box/.style = {draw,inner sep=1pt,rounded corners=5pt,thick}]
		\node (a) at (0,0) {\begin{tikzpicture}[scale = 0.8]
				\begin{scope}
					\clip (0,0) rectangle (\xmax,\ymax);
					\path[fill=lightgray] (10,0)--({4+2*\tval},0)-- ({4+2*\tval},{3*\tval}) -- (6,3) -- ({2+2*\tval},3)--({2+2*\tval},{3+3*\tval})--(10,15);
					\draw[thick] (10,0) -- (0,0) -- (10,15);
					\draw[dashed] (4,0) -- (6,3) -- (2,3);
					\draw[thick] ({4+2*\tval},0)-- ({4+2*\tval},{3*\tval})-- (6,3) -- ({2+2*\tval},3)--({2+2*\tval},{3+3*\tval});
				\end{scope}
				\foreach \x in {0,...,{\xmax}} {
					\foreach \y in {0,...,{\ymax}} {
						\draw[fill=black] (\x,\y) circle (\dotwidth);
					}
				}
				\node[below] at (4,0) {$\vec b_1$};
				\node[above left] at (2,3) {$\vec b_2,\vec a_2$};
				\node[below] at (2,0) {$\vec a_1$};
				\node[above] at (6,3) {$(6,3)$};
				\draw[fill=black] ({4+2*\tval},0) circle (\dotwidth) node[below] {$(4+2t,0)$};
				\draw[fill=black] ({4+2*\tval},{3*\tval}) circle (\dotwidth) node[above right] {$(4+2t,3t)$};
				\draw[fill=black] ({2+2*\tval},3) circle (\dotwidth) node[below] {$(2+2t,3)$};
				\draw[fill=black] ({2+2*\tval},{3+3*\tval}) circle (\dotwidth) node[above left] {$(2+2t,3+3t)$};
				\draw[ultra thick] (2,0)--(2,3) node[midway, right] {$H$};
		\end{tikzpicture}};
		\def\tval{1.2}
		\node (b) at (8,0) {\begin{tikzpicture}[scale = 0.8]
				\begin{scope}
					\clip (0,0) rectangle (\xmax,\ymax);
					\path[fill=lightgray] (10,0) -- ({4+2*\tval},0) -- ({4+2*\tval},3) -- ({2+2*\tval},3)--({2+2*\tval},{3+3*\tval})--(10,15);
					\draw[thick] (10,0) -- (0,0) -- (10,15);
					\draw[dashed] (6,3) -- (2,3);
					\draw[dashed] (4,0) -- ({4+2*\tval},{3*\tval});
					\draw[thick] ({4+2*\tval},0) -- ({4+2*\tval},{3*\tval});
					\draw[thick] ({4+2*\tval},0)-- ({4+2*\tval},3)-- ({2+2*\tval},3)--({2+2*\tval},{3+3*\tval});
				\end{scope}
				\foreach \x in {0,...,{\xmax}} {
					\foreach \y in {0,...,{\ymax}} {
						\draw[fill=black] (\x,\y) circle (\dotwidth);
					}
				}
				\node[below] at (4,0) {$\vec b_1$};
				\node[above left] at (2,3) {$\vec b_2,\vec a_2$};
				\node[below] at (2,0) {$\vec a_1$};
				
				\draw[fill=black] ({4+2*\tval},0) circle (\dotwidth) node[below] {$(4+2t,0)$};
				\draw[fill=black] ({4+2*\tval},3) circle (\dotwidth) node[above right] {$(4+2t,3)$};
				\draw[fill=black] ({2+2*\tval},3) circle (\dotwidth) node[below] {$(2+2t,3)$};
				\draw[fill=black] ({2+2*\tval},{3+3*\tval}) circle (\dotwidth) node[right] {$(2+2t,3+3t)$};
				\draw[ultra thick] (2,0)--(2,3) node[midway, right] {$H$};
		\end{tikzpicture}};
		\node[box,fit=(a)] {};
		\node[box,fit=(b)] {};
		\node at (8,-4) {$t\ge 1$};
		\node at (0,-4) {$0<t<1$};
	\end{tikzpicture}
	
	\subsection{Computing the integral closures}  
	
	A vector $\vec v$ is in the Newton polytope of $I$ if and only if $\langle \vec c_1,\vec v\rangle\geq 0$, $\langle \vec c_2,\vec v\rangle \geq 0$, and $\langle (3,2),\vec v\rangle \geq 12$, with the last inequality describing the half-plane with boundary through the vectors $\vec b_1$ and $\vec b_2$. Note that since the generators of $I$ lie on the facets of the cone $\sigma^\vee$, a vector in $\sigma^\vee$ is on the Newton polytope of $I$ if and only if $\langle (3,2),\vec v\rangle \geq 12$, equivalently $3v_1+2v_2 \geq 12$.
	
	The integral closure of $I$ is generated by monomials in the Newton polytope, and so we have that, for all $t$,
	\[\overline{I}= (x^4,x^4y,x^3y^2, x^2y^3).\]
	
	A monomial $x^{v_1}y^{v_2}$ is in $\overline{I}^{\mfa^t}$ if and only if $\vec v+tH$ is contained in the Newton polytope, where $H$ is the convex hull of $\vec a_1$ and $\vec a_2$.  Since the Newton polytope is convex, it suffices to check that $\vec v+t\vec a_1$ and $\vec v+t\vec a_2$ are in the Newton polytope.  Thus, $x^{\vec v}\in \overline{I}^{\mfa^t}$ if and only if the following hold.
	\begin{align*}
		\langle \vec (3,2),\vec v+t\vec a_1\rangle&\geq 12 &&\text{ equivalently, }& 3v_1+ 2v_2+6t&\geq 12\\
		\langle \vec (3,2),\vec v + t\vec a_2\rangle&\geq 12 &&\text{ equivalently, }& 3v_1+2v_2+12t&\geq 12.
	\end{align*}
	Therefore $x^{\vec v}\in \overline{I}^{\mfa^t}$ if and only if $3v_1+ 2v_2+6t\geq 12$, i.e.\ $t\geq \frac{12-3v_1-2v_2}{6}$.  For each of the eight monomials in $R$ outside $\overline{I}$, we can compute the value of $t$ which puts that monomial into $\overline{I}^{\mfa^t}$ and describe this ideal for all $t$.
	\[ 
	\begin{array}{|c|c|c||c|c|c|}
		\hline
		x^{\vec v} &  12 - 3v_1-2v_2   & t \text{ such that }x^{\vec v}\in \overline{I}^{\mfa^t} & x^{\vec v} &  12 - 3v_1-2v_2   & t \text{ such that }x^{\vec v}\in \overline{I}^{\mfa^t}\\
		\hline
		x^3y & 1 & t \geq 1/6 &x^2 & 6 &  t\geq 1 \\
		x^2y^2 & 2 & t\geq 1/3 &xy & 7 &t \geq 7/6\\
		x^3 & 3 & t \geq 1/2 & x &  9 & t\geq 3/2 \\
		x^2y & 4 &  t\geq 2/3&1 & 12 & t \geq 2\\		
		\hline
	\end{array}
	\]
	This proves \eqref{meg:int} in Example \ref{MegaExample}.

	\subsection{Computing the tight and Frobenius closures}

To use \Cref{theorem - closures to convex geoemtry} and \Cref{theorem - convex geoemtry to closures}, we will rewrite the condition $\vec v + tH \subset \mathcal{B} + sH +\sigma^{\vee}$ in terms of the components of the vectors.

For any $s\geq 0$, $t \in \mathbb{R}_{> 0}$, and $\vec v\in S$, $\vec v + tH \subset \mathcal{B}+sH+\sigma^\vee$ if and only if for all $\vec u\in \Delta$ there exists $\vec w\in \Delta$ and $i$ such that $\vec v + tA\vec u \subset \vec b_i+sA\vec w+\sigma^\vee$, that is,
\begin{equation}\label{eq:general matrix intequality}
\vec v + tH \subset \mathcal{B}+sH+\sigma^\vee\quad \iff \forall \vec u\in\Delta \; \exists \vec w\in \Delta\;\exists i\;\text{s.t.}\;C^T(\vec v -\vec b_i+A(t\vec u -s\vec w))\geq \vec 0
\end{equation}
For our calculations, we will focus on the case $\mu(\mfa)=2$. Suppose that $A=(\vec a_1\mid \vec a_2)$, and hence $\Delta = \{(1-\gamma,\gamma)^T\mid \gamma\in [0,1]\}$.  Now, for $\vec u =(1-\gamma,\gamma)^T\in \Delta$, $\vec w = (1-\delta,\delta)^T\in \Delta$, and $s\geq 0$, $t \in \mathbb{R}_{> 0}$,  we have that 
\[A(t\vec u - s\vec w) =(\vec a_1\mid \vec a_2)\left(\begin{array}{c} t-s -(t\gamma - s\delta) \\ t\gamma-s\delta\end{array}\right) = (t-s)\vec a_1+ (t\gamma-s\delta)(\vec a_2-\vec a_1).\]
Now, for a fixed $i$, $\vec u$, and $\vec w$, the containment on the right hand side of \eqref{eq:general matrix intequality} can be rewritten as
\begin{align}
	&C^T(\vec v -\vec b_i+A(t\vec u -s\vec w))\geq \vec 0 \nonumber \\
	\iff & C^TA(t\vec u -s\vec w)) \geq C^T\vec b_i - C^t\vec v \nonumber \\
	\iff &\forall j,\; \langle \vec c_j, (t-s)\vec a_1 + (t\gamma-s\delta)(\vec a_2-\vec a_1)\rangle \geq \langle \vec c_j,\vec b_i\rangle - \langle \vec c_j,\vec v\rangle\nonumber \\
	\iff &\forall j,\; (t\gamma-s\delta)\langle \vec c_j, \vec a_2-\vec a_1\rangle \geq \langle \vec c_j,\vec b_i\rangle - \langle \vec c_j,\vec v\rangle -(t-s)\langle\vec c_j,\vec a_1\rangle.\label{eq:simplified matrix inequality}
\end{align}
Now for all $j$ with $\langle \vec c_j,\vec a_2-\vec a_1\rangle >0$, \eqref{eq:simplified matrix inequality} becomes
\begin{equation}
	t\gamma-s\delta \geq \frac{ \langle \vec c_j,\vec b_i\rangle - \langle \vec c_j,\vec v\rangle -(t-s)\langle\vec c_j,\vec a_1\rangle}{\langle \vec c_j, \vec a_2-\vec a_1\rangle}.
\end{equation}
For each $i$, let 
\[m_i(\vec v)=\max\left\{\frac{ \langle \vec c_j,\vec b_i\rangle - \langle \vec c_j,\vec v\rangle -(t-s)\langle\vec c_j,\vec a_1\rangle}{\langle \vec c_j, \vec a_2-\vec a_1\rangle}\mid j:\langle \vec c_j,\vec a_2-\vec a_1\rangle >0\right\}. \]

For all $j$ with $\langle \vec c_j,\vec a_2-\vec a_1\rangle <0$, \eqref{eq:simplified matrix inequality} becomes
\begin{equation}
	t\gamma-s\delta \leq \frac{ \langle \vec c_j,\vec b_i\rangle - \langle \vec c_j,\vec v\rangle -(t-s)\langle\vec c_j,\vec a_1\rangle}{\langle \vec c_j, \vec a_2-\vec a_1\rangle}.
\end{equation}
For each $i$, let 
\[M_i(\vec v)=\min\left\{\frac{ \langle \vec c_j,\vec b_i\rangle - \langle \vec c_j,\vec v\rangle -(t-s)\langle\vec c_j,\vec a_1\rangle}{\langle \vec c_j, \vec a_2-\vec a_1\rangle}\mid j:\langle \vec c_j,\vec a_2-\vec a_1\rangle <0\right\}. \]

For any $j$ with $\langle \vec c_j,\vec a_2-\vec a_1\rangle =0$, \eqref{eq:simplified matrix inequality} is
\begin{equation}
	\langle \vec c_j,\vec b_i\rangle - \langle \vec c_j,\vec v\rangle -(t-s)\langle\vec c_j,\vec a_1\rangle\leq 0 \quad \iff \quad \langle \vec c_j,\vec v\rangle \geq \langle \vec c_j,\vec b_i-(t-s)\vec a_1\rangle.\label{eq:zero inequality}
\end{equation}
Inequality \eqref{eq:zero inequality} does not depend on $\gamma$ or $\delta$, and so is a necessary condition on any $\vec v$ with $\vec v + tA\vec u\in \vec b_i+sH+\sigma^\vee$.  We will say that  \emph{$\vec v$ and $\vec b_i$ satisfy \eqref{eq:zero inequality}}  if for all $j$ with $\langle \vec c_j,\vec a_2-\vec a_1\rangle =0$, we have $\langle \vec c_j,\vec v\rangle \geq \langle \vec c_j,\vec b_i-(t-s)\vec a_1\rangle$.

Thus, we have that for a particular $i$, $\vec u = (1-\gamma,\gamma)^T$, and $\vec w = (1-\delta, \delta)^T$, that $\vec v + tA\vec u\in \vec b_i +sA\vec w+\sigma^\vee$ if and only if $m_i(\vec v) \leq t\gamma - s\delta\leq M_i(\vec v)$ and $\vec v$  satisfies \eqref{eq:zero inequality}.

Therefore, so long as $\vec v$ and $\vec b_i$ satisfy \eqref{eq:zero inequality}, for a fixed $\vec u = (1-\gamma,\gamma)^T\in\Delta$, $\vec v + tA\vec u\in \vec b_i+sH+\sigma^\vee$ if and only if there exists $\delta\in [0,1]$ such that $m_i(\vec v)\leq t\gamma-s\delta\leq M_i(\vec v)$.  So long as $m_i(\vec v)\leq M_i(\vec v)$, it will be possible to find a $\delta$ making this inequality true if and only if $m_i(\vec v)\leq t\gamma\leq M_i(\vec v)+s$.  For each $i$, let
\[U_i(\vec v) 
=\left\lbrace 
	\begin{array}{cl}
	\left[\dfrac{m_i(\vec v)}{t},\dfrac{M_i(\vec v)+s}{t}\right] &\text{ if  $\vec v$ and $\vec b_i$ satisfy \eqref{eq:zero inequality} and } m_i(\vec v)\leq M_i(\vec v) \\
		\varnothing & \text{ otherwise}
	\end{array}
	\right.\]
Then each $U_i$ contains the values of $\gamma$ for which $\vec v + tA\vec u\in \vec b_i+sH+\sigma^\vee$.  Therefore,
\begin{equation}\label{2-generated m functions}
	\vec v + tA\vec u\in \mathcal{B}+sH+\sigma^\vee \iff [0,1]\subseteq \bigcup_i U_i(\vec v).
\end{equation}
	
	We are now ready to provide details of the computation of \Cref{MegaExample}. Recall that $R=k[x,xy,x^2y^3] = k[S]$, where $S=\sigma^\vee\cap \Z^2$ and $\sigma^\vee$ is the cone generated by $(1,0)$ and $(2,3)$, which is dual to the cone $\sigma$ generated by $\vec c_1 = (0,1)$ and $\vec c_2 = (3,-2)$. Let $I=(x^4, x^2y^3)$ and $\mfa = (x^2,x^2y^3)$, so that $\vec b_1=(4,0)$, $\vec b_2=(2,3)$, $\vec a_1 = (2,0)$, and $\vec a_2= (2,3)$.

	We will use \Cref{2-generated m functions} and \Cref{corollary - normal semigroup} to compute the adjusted tight and Frobenius closures.  We have that $\vec a_2-\vec a_1 = (0,3)$, and so $\langle \vec c_1,\vec a_2-\vec a_1\rangle  = 3$ and $\langle\vec c_2,\vec a_2-\vec a_1\rangle = -6$.  So, for all $i$ and all vectors $\vec v$, $\vec v$ and $\vec b_i$ satisfy \eqref{eq:zero inequality}.

	For any $\vec v = (v_1,v_2)\in S$, we have that
	\[m_i(\vec v) = \frac{ \langle \vec c_1,\vec b_i\rangle - \langle \vec c_1,\vec v\rangle -(t-s)\langle\vec c_1,\vec a_1\rangle}{\langle \vec c_1, \vec a_2-\vec a_1\rangle}
	= \frac{\langle \vec c_1,\vec b_i\rangle - v_2}{3} 
	= \begin{cases}
		-\frac{v_2}{3}&\text{if }i=1\\ \frac{3-v_2}{3}&\text{if }i=2
	\end{cases} = \frac{3i - 3 - v_2}{3}
	\]
	and
	\[M_i(\vec v) =\frac{ \langle \vec c_2,\vec b_i\rangle - \langle \vec c_2,\vec v\rangle -(t-s)\langle\vec c_2,\vec a_1\rangle}{\langle \vec c_2, \vec a_2-\vec a_1\rangle}
	=\frac{ \langle \vec c_2,\vec b_i\rangle - (3v_1-2v_2)-6(t-s)}{-6}
	\]
	\[=
	\begin{cases}
		-\frac{1}{6}(12-3v_1+2v_2-6(t-s)) &\text{if }i=1\\
		-\frac{1}{6}(-3v_1+2v_2-6(t-s)) &\text{if }i=2
	\end{cases} = \frac{12i - 24 + 3v_1-2v_2}{6} + t-s. \]
	Thus, for $i=1,2$, we have
	\[U_i = \left\lbrace 
	\begin{array}{cl}
	\left[\dfrac{3i-3-v_2}{3t},\dfrac{12i-24 +3v_1-2v_2}{6t} + 1\right] &\text{ if }m_i(\vec v)\leq M_i(\vec v),\text{ i.e., } v_1\geq 6-2i+2(s-t) \\
		\varnothing & \text{ otherwise}
	\end{array}
	\right.
	\]
	By \Cref{2-generated m functions} we have that $\vec v + tH\subset \mathcal{B}+sH+\sigma^\vee$ if and only if $[0,1]\subset U_1\cup U_2$. 
	Suppose that $[0,1]\subset U_1$.  This implies that
	\[\frac{-v_2}{3t} \leq 0 \text{ and }\frac{-12 +3v_1-2v_2}{6t} + 1 \geq 1 \iff v_2\geq 0 \text{ and } 3v_1-2v_2\geq 12. \]
	
	In particular, this means that $\langle \vec c_1, \vec v - \vec b_1\rangle = v_2 \geq 0$ and $\langle \vec c_2,\vec v-\vec b_1\rangle = 3v_1-2v_2 - 12 \geq 0$.  Therefore $\vec v-\vec b_1\in \sigma^\vee$, implying that $\vec v\in \vec b_1+\sigma^\vee$ and so $x^{\vec v}\in I$.
	
	A similar computation shows that if $[0,1]\subset U_2$, then 
	 $\vec v\in \vec b_2+\sigma^\vee$ and so $x^{\vec v}\in I$.

	Therefore, the only situation in which $
	\vec v + tH\subset \mathcal{B}+sH+\sigma^\vee$ but $x^{\vec v}\notin I$ is when $[0,1]$ is contained in the union of $U_1$ and $U_2$ but is contained in neither $U_1$ and $U_2$.  For this situation to occur we must have that
	\begin{align*}
		v_1&\geq 4-2(t-s) & \text{ so that }U_i\neq \varnothing\text{ for }i=1,2\\
		\frac{m_1(\vec v)}{t} = \frac{-v_2}{3t} &\leq 0 &\text{ so that }0\in U_1\\
		 \frac{-12+3v_1-2v_2}{6t}+1  = \frac{M_1(\vec v) + s}{t}&\geq   \frac{m_2(\vec v)}{t} =\frac{3-v_2}{3t}& \text{ so that }U_1\cap U_2\neq\varnothing\\
		\frac{M_2(\vec v)+s}{t} = \frac{3v_1-2v_2}{6t}+1 &\geq 1 & \text{ so that } 1\in U_2.
	\end{align*}
	This system simplifies to
	\begin{align}
		\label{eq:subeq1}v_1&\geq 4-2(t-s),\\
		\label{eq:subeq2}v_2 & \geq 0, \\
		\label{eq:subeq3}v_1&\geq 6-2t, \text{ and }\\
		\label{eq:subeq4}3v_1-2v_2&\geq 0.
	\end{align}
	Inequalities \eqref{eq:subeq2} and \eqref{eq:subeq4} are satisfied for all $\vec v\in S$.  Thus, we have that for $\vec v\in S$, 
	\begin{equation}\label{eq:convex containment condition}
		\vec v + tH\subset \mathcal{B}+sH+\sigma^\vee \iff x^{\vec v}\in I \text{ or }v_1\geq \max\{4-2(t-s), 6-2t\}.
	\end{equation}

	\subsection{The Hara-Yoshida tight and Frobenius closures}  By \Cref{corollary - normal semigroup}, $x^{\vec v}\in \istarat = \ifat$ if and only if $
	\vec v + tH\subset \mathcal{B}+\sigma^\vee$. Taking $s=0$ in \eqref{eq:convex containment condition}, we see that this is satisfied for any $\vec v$ with $x^{\vec v}\in I$ or $v_1\geq \max\{4-2(t-0), 6-2t\} = 6-2t$.  So, we have that $\istarat = I+(x^{\vec v}\mid \vec v\in S, v_1\geq 6-2t)$, which gives \eqref{meg:HY} in \Cref{MegaExample}.

	\subsection{The Vraciu tight closure} By \Cref{corollary - normal semigroup}, $x^{\vec v}\in \atistar$ if and only if  $\vec v+tH\subset \mathcal{B}+tH+\sigma^\vee$.  Taking $s=t$ in \eqref{eq:convex containment condition}, this is satisfied for any $\vec v$ with $x^{\vec v}\in I$ or $v_1\geq \max\{4-2(t-t),6-2t\} = \max\{4,6-2t\}$.  
	This proves \eqref{meg:Vtight} in \Cref{MegaExample}.

		\subsection{The Vraciu Frobenius closure} By \Cref{corollary - normal semigroup}, a sufficient condition for $x^{\vec v}$ to be in $\atif$ is for $\vec v+tH\subset \mathcal{B}+sH+\sigma^\vee$ for some $s>t$.  Letting $s=t+\epsilon$ for some small $\epsilon>0$ in \eqref{eq:convex containment condition}, we have that this condition is satisfied for any $\vec v$ with $x^{\vec v}\in I$ or $v_1\geq \max\{4+2\epsilon, 6-2t\}$.  Thus, we have that if $t \geq 1/2$, then $\atif \supseteq I+(x^5y^2)$.
		
		Now, if $t<1/2$, then $I\subset \atif\subset\atistar = I$, so $\atif = I$.  If $1/2 \leq t < 1$, then $I+(x^5y^2)\subset \atif \subset \atistar = I+(x^5y^2)$, and so $\atif = I+(x^5y^2)$.  If $t\geq 1$, then $I+(x^5y^2)\subset \atif \subset \atistar = I+(x^4y,x^4y^2)$.

		To determine $\atif$ for $t\geq 1$, we use the definition to check whether the monomial $x^4y^r$ belongs to $\atif$ for $r\in \{1,2\}$. 

		If $x^4y^r\in \atif$, then since $x^2\in \mfa$, there exists $e$ such that 
		\begin{equation}\label{in closure => p=3 eq 1}
			x^{4p^e}y^{rp^e}x^{2\lceil t p^e\rceil}\in (x^4y)^{p^e}\mfa^{\lceil tp^e\rceil } \subset (x^4,x^2y^3)^{[p^e]}(x^2,x^2y^3)^{\lceil tp^e\rceil}.
		\end{equation}
		Since comparing degrees of $y$ shows that $x^{4p^e}y^{rp^e}x^{2\lceil t p^e\rceil}\notin (x^2y^3)^{[p^e]}$, if \eqref{in closure => p=3 eq 1} holds then there exists $0\leq j\leq \lceil tp^e\rceil$ such that 
		\[x^{4p^e + 2\lceil tp^e\rceil}y^{rp^e} =  x^{4p^e}y^{rp^e}x^{2\lceil t p^e\rceil}\in x^{4p^e}(x^{2\lceil tp^e\rceil}y^{3j}) = (x^{4p^e + 2\lceil tp^e\rceil}y^{3j}). \]
		Therefore $y^{rp^e - 3j}\in R$.  Since the only power of $y$ in $R$ is $1$, $rp^e = 3j$, and so $p=3$.

		On the other hand, suppose $p=3$ and let $e\geq 1$.  The ideal $(x^4y^r)^{3^e}\mfa^{\lceil 3^et\rceil}$ is generated by the monomials $(x^4y^r)^{3^e}(x^{\lceil 3^et\rceil} y^{3i}) = x^{4\cdot 3^e + \lceil 3^e t\rceil}y^{3^er+3i}$ for $0\leq i\leq \lceil 3^et\rceil$.

		If $i \leq 3^e - 3^{e-1}r$, then let $j=3^{e-1}r + i\leq 3^e\leq \lceil 3^et\rceil$, and note that
		\[x^{4\cdot 3^e + 2\lceil 3^e t\rceil}y^{3^er+3i}
		=(x^4)^{3^e}x^{2\lceil 3^e t\rceil}y^{3^er+3(j-3^{e-1}r)}
		= (x^4)^{3^e}x^{2\lceil 3^e t\rceil}y^{3j}
		\in (x^4,x^2y^3)^{[3^e]}(x^2,x^2y^3)^{\lceil 3^e t\rceil}.
		\]
		If $i \geq 3^e - 3^{e-1}r$, then let $j = i - 3^e + 3^{e-1}r\geq 0$, and note that
		\[x^{4\cdot 3^e + 2\lceil 3^e t\rceil}y^{3^er+3i}
		=(x^2y^3)^{3^e}x^{2\cdot 3^e+2\lceil 3^e t\rceil}y^{3^er + 3i - 3^{e+1}}
		=(x^2y^3)^{3^e}x^{2\cdot 3^e+2\lceil 3^e t\rceil}y^{3^er + 3(j + 3^e - 3^{e-1}r) - 3^{e+1}}\]\[
		=(x^2y^3)^{3^e}x^{2\cdot 3^e}x^{2\lceil 3^e t\rceil}y^{3j}
		\in (x^4,x^2y^3)^{[3^e]}(x^2,x^2y^3)^{\lceil 3^e t\rceil}.
		\]
		Therefore $x^4y^r\in \atif$ in this case.

		So, we have found that $\atif$ for $t\geq 1$ depends on $p$.  Namely, $\atif$ contains $x^4y$ and $x^4y^2$ if and only if $t\geq 1$ and $p =3$. Hence membership in $\atif$ cannot be detected simply using the convex geometry of the affine semigroup.  This shows \eqref{meg:VFrob} in \Cref{MegaExample}.
\end{appendices}

\bibliographystyle{alpha}
\bibliography{ref.bib}

\begin{thebibliography}{BDHM24}

\bibitem[BB05]{BB}
Manuel Blickle and Raphael Bondu.
\newblock Local cohomology multiplicities in terms of \'etale cohomology.
\newblock {\em Ann. Inst. Fourier (Grenoble)}, 55(7):2239--2256, 2005.

\bibitem[BDHM24]{BDHM}
Sankhaneel Bisui, Sudipta Das, Tài~Huy Hà, and Jonathan Montaño.
\newblock Rational powers, invariant ideals, and the summation formula, 2024.
\newblock arXiv:2402.12350.

\bibitem[Ciu20]{CC}
Cătălin Ciupercă.
\newblock Derivations and rational powers of ideals.
\newblock {\em Arch. Math. (Basel)}, 114(2):135--145, 2020.

\bibitem[CLS11]{CLS}
David~A. Cox, John~B. Little, and Henry~K. Schenck.
\newblock {\em Toric varieties}, volume 124 of {\em Graduate Studies in
  Mathematics}.
\newblock American Mathematical Society, Providence, RI, 2011.

\bibitem[ERGV23]{EJRG}
Neil Epstein, Rebecca R.~G., and Janet Vassilev.
\newblock Integral closure, basically full closure, and duals of nonresidual
  closure operations.
\newblock {\em J. Pure Appl. Algebra}, 227(4):Paper No. 107256, 33, 2023.

\bibitem[Eto02]{eto}
Kazufumi Eto.
\newblock Multiplicity and {H}ilbert-{K}unz multiplicity of monoid rings.
\newblock {\em Tokyo J. Math.}, 25(2):241--245, 2002.

\bibitem[Ful93]{Fulton}
William Fulton.
\newblock {\em Introduction to toric varieties}, volume 131 of {\em Annals of
  Mathematics Studies}.
\newblock Princeton University Press, Princeton, NJ, 1993.
\newblock The William H. Roever Lectures in Geometry.

\bibitem[Har01]{H}
Nobuo Hara.
\newblock Geometric interpretation of tight closure and test ideals.
\newblock {\em Trans. Amer. Math. Soc.}, 353(5):1885--1906, 2001.

\bibitem[HH90]{HH}
Melvin Hochster and Craig Huneke.
\newblock Tight closure, invariant theory, and the {B}rian\c con-{S}koda
  theorem.
\newblock {\em J. Amer. Math. Soc.}, 3(1):31--116, 1990.

\bibitem[HH94]{HHSplit}
Melvin Hochster and Craig Huneke.
\newblock Tight closure of parameter ideals and splitting in module-finite
  extensions.
\newblock {\em J. Algebraic Geom.}, 3(4):599--670, 1994.

\bibitem[Hun92]{HunekeUniformity}
Craig Huneke.
\newblock Uniform bounds in {N}oetherian rings.
\newblock {\em Invent. Math.}, 107(1):203--223, 1992.

\bibitem[Hun98]{SixLec}
Craig Huneke.
\newblock Tight {C}losure, {P}arameter {I}deals, and {G}eometry.
\newblock In J.~M. Giral, R.~M. Mir\'o-Roig, and S.~Zarzuela, editors, {\em Six
  lectures on commutative algebra}, volume 166 of {\em Progress in
  Mathematics}, pages 187--239. Birkh\"auser Verlag, Basel, 1998.

\bibitem[HY03]{HY}
Nobuo Hara and Ken-Ichi Yoshida.
\newblock A generalization of tight closure and multiplier ideals.
\newblock {\em Trans. Amer. Math. Soc.}, 355(8):3143--3174, 2003.

\bibitem[KZ19]{KZ}
Mordechai Katzman and Wenliang Zhang.
\newblock Multiplicity bounds in prime characteristic.
\newblock {\em Comm. Algebra}, 47(6):2450--2456, 2019.

\bibitem[LJT08]{LJT}
Monique Lejeune-Jalabert and Bernard Teissier.
\newblock Cl\^oture int\'egrale des id\'eaux et \'equisingularit\'e.
\newblock {\em Ann. Fac. Sci. Toulouse Math. (6)}, 17(4):781--859, 2008.
\newblock With an appendix by Jean-Jacques Risler.

\bibitem[LT81]{LT81}
Joseph Lipman and Bernard Teissier.
\newblock Pseudorational local rings and a theorem of {B}rian\c con-{S}koda
  about integral closures of ideals.
\newblock {\em Michigan Math. J.}, 28(1):97--116, 1981.

\bibitem[MP25]{MaPolstra}
Linquan Ma and Thomas Polstra.
\newblock {\em {F}-singularities: a commutative algebra approach}.
\newblock unpublished, 2025.
\newblock Available at
  \href{https://www.math.purdue.edu/~ma326/F-singularitiesBook.pdf}{https://www.math.purdue.edu/{$\sim$}ma326/F-singularitiesBook.pdf}.

\bibitem[PQ19]{PQ}
Thomas Polstra and Pham~Hung Quy.
\newblock Nilpotence of {F}robenius actions on local cohomology and {F}robenius
  closure of ideals.
\newblock {\em J. Algebra}, 529:196--225, 2019.

\bibitem[Quy19]{Q}
Pham~Hung Quy.
\newblock On the uniform bound of {F}robenius test exponents.
\newblock {\em J. Algebra}, 518:119--128, 2019.

\bibitem[Ree56]{ReesValuations}
David Rees.
\newblock Valuations associated with ideals. {II}.
\newblock {\em J. London Math. Soc.}, 31:221--228, 1956.

\bibitem[SBc74]{BS74}
Henri Skoda and Jo\"el {B}rian\c con.
\newblock Sur la cl\^oture int\'egrale d'un id\'eal de germes de fonctions
  holomorphes en un point de {${\bf C}\sp{n}$}.
\newblock {\em C. R. Acad. Sci. Paris S\'er. A}, 278:949--951, 1974.

\bibitem[Sch08]{Sch}
Karl Schwede.
\newblock Generalized test ideals, sharp {$F$}-purity, and sharp test elements.
\newblock {\em Math. Res. Lett.}, 15(6):1251--1261, 2008.

\bibitem[SH06]{hunekeSwanson}
Irena Swanson and Craig Huneke.
\newblock {\em Integral closure of ideals, rings, and modules}, volume 336 of
  {\em London Mathematical Society Lecture Note Series}.
\newblock Cambridge University Press, Cambridge, 2006.

\bibitem[Smi97]{Smith97}
Karen~E. Smith.
\newblock Tight closure in graded rings.
\newblock {\em J. Math. Kyoto Univ.}, 37(1):35--53, 1997.

\bibitem[Smi00]{Smith}
Karen~E. Smith.
\newblock The multiplier ideal is a universal test ideal.
\newblock volume~28, pages 5915--5929. 2000.
\newblock Special issue in honor of Robin Hartshorne.

\bibitem[SS25]{SchwedeSmith}
Karl Schwede and Karen Smith.
\newblock {\em Singularities defined in terms of the {F}robenius map}.
\newblock unpublished, 2025.
\newblock Available at
  \href{https://github.com/kschwede/FrobeniusSingularitiesBook}{https://github.com/kschwede/FrobeniusSingularitiesBook}.

\bibitem[ST12]{ST}
Karl Schwede and Kevin Tucker.
\newblock A survey of test ideals.
\newblock In {\em Progress in commutative algebra 2}, pages 39--99. Walter de
  Gruyter, Berlin, 2012.

\bibitem[ST17]{SriTak}
Vasudevan Srinivas and Shunsuke Takagi.
\newblock Nilpotence of {F}robenius action and the {H}odge filtration on local
  cohomology.
\newblock {\em Adv. Math.}, 305:456--478, 2017.

\bibitem[Tay21]{Tay}
William~D. Taylor.
\newblock Fundamental results on {$s$}-closures.
\newblock {\em J. Pure Appl. Algebra}, 225(4):Paper No. 106565, 14, 2021.

\bibitem[Vas05]{Vasc}
Wolmer Vasconcelos.
\newblock {\em Integral closure}.
\newblock Springer Monographs in Mathematics. Springer-Verlag, Berlin, 2005.
\newblock Rees algebras, multiplicities, algorithms.

\bibitem[Vra08]{Vraciu}
Adela Vraciu.
\newblock A new version of {${\mathfrak a}$}-tight closure.
\newblock {\em Nagoya Math. J.}, 192:1--25, 2008.

\bibitem[Wat00]{Watanabe}
Kei-ichi Watanabe.
\newblock {H}ilbert-{K}unz multiplicity of toric rings.
\newblock {\em Proc. of the Inst. of Natural Sciences}, 35:173 -- 177, 2000.

\end{thebibliography}

\end{document}